\documentclass[11pt,reqno]{amsart}
\usepackage{graphicx} 
\usepackage{amsmath,amsthm,amssymb,amsxtra,xcolor}
\usepackage[utf8]{inputenc}
\usepackage[english]{babel}
\usepackage{hyperref}
\usepackage{wrapfig}
\usepackage[square,numbers]{natbib}
\usepackage{anysize}

\usepackage{caption}
\usepackage{subcaption}

\usepackage{cancel}

\usepackage[letterpaper, margin=3.3cm]{geometry}
\numberwithin{equation}{section}
\providecommand{\abs}[1]{\lvert#1\rvert} 

\theoremstyle{thm}
\newtheorem{thm}{Theorem}[section]
\newtheorem{lem}[thm]{Lemma}
\newtheorem{cor}[thm]{Corollary}

\theoremstyle{definition}
\newtheorem{defn}[thm]{Definition}

\theoremstyle{remark}
\newtheorem{rem}[thm]{Remark}

\DeclareMathOperator{\sech}{sech}

\DeclareMathOperator{\GL}{GL}
\DeclareMathOperator{\Gr}{Gr}

\DeclareMathOperator{\Wr}{Wr}

\title[Uniqueness of KP structures]{On uniqueness of KP soliton structures}

\author[Alegr\'ia]{Francisco Alegr\'ia}
\address{Instituto de Ciencias F\'isicas y Matem\'aticas, Facultad de Ciencias, Universidad Austral de Chile, Valdivia, Chile.}
\email{francisco.alegria@uach.cl}
\thanks{F.A.'s work is partially supported by ANID Exploration project 13220060.}

\author[Chen]{Gong Chen}
\address {School of Mathematics, Georgia Institute of Technology, Atlanta, GA 30332-0160, USA }
\email{gc@math.gatech.edu}
\thanks{G.C. was partially funded by Simons Foundation  MP-TSM-00002258.}

\author[Mu\~noz]{Claudio Mu\~noz}
	\address{Departamento de Ingenier\'{\i}a Matem\'atica and Centro
de Modelamiento Matem\'atico (UMI 2807 CNRS), Universidad de Chile, Casilla
170 Correo 3, Santiago, Chile.}
	\email{cmunoz@dim.uchile.cl}
	\thanks{C.M. was partially funded by Chilean research grants FONDECYT 1231250, ANID Exploraci\'on 13220060, MathAmSud WAFFLE 23-MATH-18, and Basal CMM FB210005.}

\author[Poblete]{Felipe Poblete}
\address{Instituto de Ciencias F\'isicas y Matem\'aticas, Facultad de Ciencias, Universidad Austral de Chile, Valdivia, Chile.}
\email{felipe.poblete@uach.cl}
\thanks{F.P.'s work is partially supported by ANID Exploration project 13220060, ANID project FONDECYT 1221076 and MathAmSud WAFFLE 23-MATH-18.}

\author[Tardy]{Benjam\'in Tardy}
	\address{Departamento de Ingenier\'{\i}a Matem\'atica, Universidad de Chile, Casilla
170 Correo 3, Santiago, Chile.}
	\email{btardy@dim.uchile.cl}
	\thanks{B.T. was partially funded by Chilean research grants FONDECYTs 1231250 and 1221076 and Basal CMM FB210005.}
\date{\today}

\begin{document}

\maketitle

\begin{abstract}
We consider the Kadomtsev-Petviashvili  II (KP) model placed in $\mathbb R_t \times \mathbb R_{x,y}^2$, in the case of smooth data that are not necessarily in a Sobolev space. In this paper, the subclass of smooth solutions we study  is of ``soliton type'', characterized by a phase $\Theta=\Theta(t,x,y)$ and a unidimensional profile $F$. In particular, every classical KP soliton and multi-soliton falls into this category with suitable  $\Theta$ and $F$. We establish concrete characterizations of KP solitons by means of a natural set of nonlinear differential equations and inclusions of functionals of Wronskian, Airy and Heat types, among others. These functional equations only depend on the new variables $\Theta$ and $F$. A distinct characteristic of this set of functionals is its special and rigid structure tailored to the considered soliton. By analyzing $\Theta$ and $F$,   we establish the uniqueness of line-solitons, multi-solitons, and other degenerate solutions among a large class of KP solutions. Our results are also valid for other 2D dispersive models such as the quadratic and cubic Zakharov-Kuznetsov equations.  

\end{abstract}


\section{Introduction and Main Results}

\subsection{Setting of the problem}%
Let $t\in\mathbb R$, and $\left(x,y\right)\in\mathbb R^2$. In this work we consider the KP-II model
\begin{equation}\label{eq:KP}
\left(-4u_t +u_{xxx} +6 u u_{x}\right)_{x} + 3u_{yy} =0,
\end{equation}
where $u=u\left(t,x,y\right)\in\mathbb R$ is the unknown. The KP equations are canonical integrable models in 2D and were first introduced by Kadomtsev and Petviashvili in 1970 \cite{KadomtsevPetviashvili1970} for modeling ``long and weakly nonlinear waves'' propagating essentially along the $x$ direction, but with a small dependence on the variable $y$. A rigorous derivation of both models from the Boussinesq system was obtained by Lannes and Lannes-Saut \cite{Lannes2,LS}. 

\medskip

The KP-II model \eqref{eq:KP} (KP from now on, if there is no confusion) has an important set of symmetries. If $u=u\left(t,x,y\right)$ is a solution to \eqref{eq:KP}, then $u\left(t+t_0, x+x_0,y+y_0\right)$, with $t_0,x_0,y_0\in\mathbb R$,  $c u\left( c^{3/2} t , c^{1/2} x, c y\right)$, if $c>0$, and  if $\beta\in\mathbb R$ is a given speed, 
\[
u\left(t,x - \frac{4\beta}{3}y + \frac{4\beta ^{2}}{3}t , y-2\beta t\right)
\]
(Galilean invariance)  define new  solutions to KP. 

\medskip

Our purpose here is to establish uniqueness results for KP solitons in the class of smooth solutions that is of soliton type, complementing the results obtained in \cite{FFMP}. More precisely, assume that $u$ in \eqref{eq:KP} is sufficiently smooth and has the form
\begin{equation}\label{eqn:FT}
u\left(t,x,y\right)= 2 \partial_x^2 F\left(\Theta\left(t,x,y\right)\right),
\end{equation}
where, for some $s_0>0$, $F: \left[s_0,\infty\right) \longrightarrow \mathbb R$ and $\Theta=\Theta\left(t,x,y\right) \in \left[s_0,\infty\right)$ are smooth functions. Later, we will justify that without loss of generality, one can choose $s_0 = 1$. Consequently, we shall assume this particular choice throughout this paper. Since $F$ can be changed by any linear affine function, we can also assume that $F\left(1\right)=0$ and $F'\left(1\right)=1$. We shall call $\Theta$ the \emph{phase} of the solution $u$, and $F$ will be the \emph{profile}. Every classical KP soliton is in this class for suitable $F$ and $\Theta$. Indeed, one of the most important examples of solutions in the form of \eqref{eqn:FT} is provided by classical solitons and multi-solitons (see, e.g., Kodama \cite{Kodama2017}), namely solutions of the KP equation in the form of
\begin{equation}\label{ppal_form}
\begin{aligned}
u\left(t,x,y\right) =&~{} 2 \partial_x^2 \log \left(\Theta\left(t,x,y\right)\right), \\
\Theta\left(t,x,y\right) := &~{}  \hbox{Wr} \left(\Theta_1,...,\Theta_n\right)\left(t,x,y\right),\\
 \Theta_i \left(t,x,y\right):= &~{} \sum_{j=1}^{M_{i}} a_{ij} \exp \left( k_{ij}x + k^{2}_{ij}y + k^{3}_{ij}t \right),
\end{aligned}
\end{equation}
where $\hbox{Wr}$ represents the classical Wronskian of $n$ functions, and $k_{ij}, a_{ij}$ are in principle just real-valued, although specific values determine precise solutions. In this case, $F=\log$ and $\Theta$ has the additional scaling symmetry: if $\Theta$ is a valid phase, then $\lambda \Theta$ in \eqref{ppal_form} also does, for any $\lambda>0$. This fact motivates the reason why requiring $\Theta>1$ is in principle not extremely restrictive. Note that, in the particular case where $F=\log$, given $u=u\left(t,x,y\right)$ solution of \eqref{eq:KP}, the formula $\Theta= \exp\left(\int_0^x\int_0^t \frac12 u\left(s,r,y\right)\mathrm{d}r \mathrm{d}s \right)$ returns a valid phase $\Theta$. Consequently, despite some loss of regularity, solving an equation for the profile $F$ and the phase $\Theta$ may be considered as general as dealing with the original KP model \eqref{eq:KP} for $u$.

\medskip

Taking into account the great diversity of KP solutions, our study will focus on a simple but useful characterization of soliton solutions. It turns out that this question is interesting and quite challenging, since $\Theta$ can assume complicated values, from simple linear functions to periodic ones. Also, the question of whether or not profile $F=\log $ is the only possibility for $F$ is also extremely interesting.
 
\medskip

The first step is to rewrite KP \eqref{eq:KP} in terms of $F$ and $\Theta$ as in \eqref{eqn:FT}. After rearrangements, and assuming convergence to zero at infinity, we obtain the following fourth order equation for $\left(F,\Theta\right)$:
\begin{equation}\label{KP3partes}
\begin{aligned}
    0= &\left(F'''' + 6F''^{2}\right)\left(\Theta\right)\Theta^{4}_{x} + 6\left(F'' + F'^{2}\right)'\left(\Theta\right)\Theta^{2}_{x}\Theta_{xx} + 3\left(F'' + F'^{2}\right)\left(\Theta\right)\left(\Theta^{2}_{xx} + \Theta^{2}_{y}\right)  \\
    &-4F''\left(\Theta\right)\Theta_{x}\left(\Theta_{t} - \Theta_{xxx}\right) + F'\left(\Theta\right)\left(-4\left(\Theta_{t} - \Theta_{xxx}\right)_{x}\right)  \\
    &+ 3F'\left(\Theta\right)\left(\Theta_{yy} - \Theta_{xxxx}\right) + 3F'^{2}\left(\Theta\right)\left(\Theta^{2}_{xx} - \Theta^{2}_{y}\right).
\end{aligned}
\end{equation}
This is the equation that will be worked in this paper. It represents a highly nonlinear equation for the two unknowns $F$ and $\Theta$, but its nature is certainly better than \eqref{eq:KP}, since it possesses hidden structures. The most important is a splitting phenomenon between some parts concerning only $F$-motivated terms, and others only dealing with $\Theta$-related terms. Of course, this is far from being an exact ``separation of variables'' as in classical linear PDEs, but we shall mention important similarities. Indeed, it is possible to divide \eqref{KP3partes} into three somehow well-defined sub-equations:
\begin{enumerate}
	\item[(a)] The first line, that can be written in terms of 
	\begin{equation}\label{rho_def0}
	\rho (s): = \left(F'' + F'^{2}\right)(s),
	\end{equation}
	corresponding to those terms that are equal to 0 when $\rho \equiv 0$. And $\rho \equiv 0$ when $F=\log$ and some particular initial conditions are met. 
	\item[(b)] The second line involves a modified Airy function $Ai\left(\Theta\right):=\Theta_{t} - \Theta_{xxx}$ and its derivative with respect to the variable $x$. This is a reminiscent of the 1D variable $x$ that has a natural Airy structure for $\Theta$ (see Appendix \ref{rem:KdVmKdV} for further details).
	\item[(c)] The third line has a complex structure represented by a Heat type term defined by $H\left(\Theta\right):= \Theta_{y} - \Theta_{xx}$. Notice that $H$ is an operator in the variables $y$ and $x$ only. 
	\item[(d)] Additionally, there is a hidden structure in \eqref{KP3partes} represented by Wronskian type functions. Later, in Definition \ref{Thetas} we will explain better this structure.
\end{enumerate}
The Heat and Airy functions are classical in the KP literature, see e.g. Kodama \cite{Kodama2017}, but Wronskians and $\rho$ functions are, as far as we understand, not so well-understood. Translated to the equation \eqref{KP3partes}, the purpose of this work is to find suitable conditions on the Airy, Heat, Wronskian and $\rho$ functions that characterize different soliton solutions. We consider the following definitions:

\begin{defn}[Classification of phases $\Theta$]\label{Thetas}
We shall say that $\Theta$ as in \eqref{eqn:FT} 
\begin{enumerate}
\item[(i)] is of Airy type if for all $\left(t,x,y\right)\in\mathbb R^{3}$,
\[
Ai\left(\Theta\right):= \Theta_t - \Theta_{xxx} =0;
\]
\item[(ii)] is of Heat type if for all $\left(t,x,y\right)\in\mathbb R^{3}$,
\[
 H\left(\Theta\right):= \Theta_{y} - \Theta_{xx} = 0;
\]
\item[(iii)] is of $x$-Wronskian type and $y$-Wronskian type if $\Theta>0$, and
\begin{equation}\label{Wronskians}
W_x\left(\Theta\right):=\Theta_{xxxx} - \frac{\Theta^{2}_{xx}}{\Theta} =0,   \quad W_y\left(\Theta\right): = \Theta_{yy} - \frac{\Theta^{2}_{y}}{\Theta}=0,
\end{equation}
respectively, for all $\left(t,x,y\right)\in\mathbb R^{3}$;
\smallskip
\item[(iv)] is of $\mathcal T$-type if for $F$ fixed,
\begin{equation}\label{eqn:T_type}
\begin{aligned}
\mathcal T\left(\Theta\right):= &~{} -4F''\left(\Theta\right)Ai\left(\Theta\right)\Theta_x \\
&~{} + F'\left(\Theta\right)\left(-4Ai\left(\Theta\right)_x + 3\left(H\left(\Theta\right)_y + H\left(\Theta\right)_{xx}\right)\right) \\
&~{} - 3 F'^{2}\left(\Theta\right)H\left(\Theta\right)\left(\Theta_y + \Theta_{xx}\right)  =0,
\end{aligned}
\end{equation}
for all $\left(t,x,y\right)\in\mathbb R^{3}$.
\end{enumerate}
\end{defn}
Notice that we ask for equality for all $\left(t,x,y\right)\in\mathbb R^{3}$, while relaxing these conditions are interesting options not considered in this work. Before stating our main results, some important comments are necessary:

\begin{rem}[On the classical operators]
Heat and Airy operators are naturally involved in \eqref{KP3partes}. Indeed, it can be proved (see Section \ref{Sec:2} for further details) that \eqref{KP3partes} can be written as
\begin{equation}\label{rhoKP}
\begin{aligned}
    & \left(\rho'' -2F'\left(\Theta\right) \rho' +4F''\left(\Theta\right) \rho\right)\Theta_{x}^{4} + 6\rho' \Theta_{x}^{2} \Theta_{xx}  \\
    &\quad -4F''\left(\Theta\right)\Theta_xAi\left(\Theta\right) + F'\left(\Theta\right)\left(-4Ai\left(\Theta\right)_x + 3\left(H\left(\Theta\right)_y + H\left(\Theta\right)_{xx}\right)\right)  \\
    &\quad +3\left(\rho\left(\Theta^2_{xx} + \Theta^2_y\right) -  F'^2\left(\Theta\right)H\left(\Theta\right)\left(\Theta_y + \Theta_{xx}\right)\right)=0,
\end{aligned}
\end{equation}
where from \eqref{rho_def0} $\rho\left(s\right)= F''\left(s\right) + F'^2\left(s\right)$. On the other hand, the operator $\mathcal T$ is the natural counterpart of the ODE type satisfied by $\rho$, in the sense that \eqref{KP3partes}-\eqref{rhoKP} reads
\begin{equation}\label{casicasi}
\begin{aligned}
&  \Theta_{x}^{4}\rho'' + 2\left(6\Theta^2_x\Theta_{xx} - 2F'\left(\Theta\right)\Theta^4_x\right) \rho'   + \left(3\left(\Theta^2_{xx} + \Theta^2_y\right) + 4F''\left(\Theta\right)\Theta^4_x\right) \rho +\mathcal T\left(\Theta\right) = 0.
 \end{aligned}
\end{equation}
Finally, notice that $\Theta$ being of $\mathcal T$-type is a condition depending on the profile $F$, and consequently is a more complex condition than being of Airy or Heat type, which are independent of the profile $F$.
\end{rem}

\begin{rem}[On the Wronskian operators]
The emergence of the Wronskians \eqref{Wronskians} in \eqref{KP3partes} seems obscure and nonstandard. However, it is possible to rewrite \eqref{KP3partes} as
\begin{equation}\label{KP3partes_new}
\begin{aligned}
    &\left(F'''' + 6F''^{2}\right)\left(\Theta\right)\Theta^{4}_{x} + 6\left(F'' + F'^{2}\right)'\left(\Theta\right)\Theta^{2}_{x}\Theta_{xx} + 3\left(F'' + F'^{2}\right)\left(\Theta\right)\left(\Theta^{2}_{xx} + \Theta^{2}_{y}\right)  \\
    &\quad -4F''\left(\Theta\right)\Theta_{x}Ai\left(\Theta\right)  -4F'\left(\Theta\right)Ai\left(\Theta\right)_{x}  \\
    &\quad + 3F'\left(\Theta\right)\left(W^{F}_{y}\left(\Theta\right) - W^{F}_{x}\left(\Theta\right)\right)=0,
\end{aligned}
\end{equation}
with $W^F_x$ and $W_y^F$ generalized Wronskian functionals,
\begin{equation}\label{general_W1}
	\begin{aligned}
		W^{F}_{y}\left(\Theta\right) :=&~{} \Theta_{yy} - F'\left(\Theta\right)\Theta^{2}_{y}, \\
		W^{F}_{x}\left(\Theta\right) :=&~{} \Theta_{xxxx} - F'\left(\Theta\right)\Theta^{2}_{xx}.
	\end{aligned}
\end{equation} Later we will prove that under $H\left(\Theta\right)=0$, one has $W_x^F\left(\Theta\right)-W_y^F\left(\Theta\right)=W_x\left(\Theta\right)-W_y\left(\Theta\right)=0$, namely one can assume that $F=\log$ in \eqref{general_W1}, leading to the natural definitions in \eqref{Wronskians}. In that sense, null Wronskians are naturally related to the Heat condition $H\left(\Theta\right)=0$, however, the equivalence will not be as exact as one would prefer.
\end{rem}

\subsection{KP solitons} The soliton family stands out as one of the most distinctive features within the KP model. Distinguished by their complexity and rich character, numerous works have been dedicated to understanding them, employing integrability, algebraic, and analytic techniques. Among these we find the works by Kodama and Williams \cite{Kodama2013,Kodama2014}, which provide a precise description of KP-solitons within the positive Grassmannian. For a comprehensive and detailed overview of this line of research, see also Kodama's monograph  \cite{Kodama2017}.

\medskip

The line soliton family (see \cite{Kodama2017}) is given by
\begin{equation}\label{linesoliton}
   \Theta \left(t,x,y\right)= a_{1}\exp\left(\theta_{1}\right) + a_{2}\exp\left(\theta_{2}\right),
\end{equation}
where $a_1,a_2>0$, and $\theta_j: =k_j x + k_j^2 y + k_j^3 t$, $k_1,k_2\in\mathbb R$. Assuming $F=\log$, the corresponding KP solution via \eqref{eqn:FT} is given by
\begin{equation}\label{soliton family}
  u\left(t,x,y\right)=\frac{1}{2}\left(k_1-k_2\right)^2 \sech^2 \left( \frac{1}{2}\left(\theta_1-\theta_2\right) \right).
\end{equation}
See Fig. \ref{line_soliton} (left) for details. The classical KdV soliton is recovered by setting $k_1=-k_2=k$, and in this case $u$ becomes 
\begin{equation}\label{Qk}
Q_k\left(t,x\right):=2k^{2}\sech^{2}\left(kx+k^{3}t\right).
\end{equation}
The next case of KP solution is the \emph{resonant multi-soliton}. This corresponds to the case 
\begin{equation}\label{resonant}
\Theta\left(t,x,y\right) = \sum^{M}_{i=1}a_{i}\exp \left(\theta_i\right) =  \sum^{M}_{i=1}a_{i}\exp\left( k_{i}x + k^{2}_{i}y + k^{3}_{i}t \right),
\end{equation}
where to ensure the positivity and nondegeneracy of $\Theta$ we impose that each $a_i>0$ and $k_1<k_2<\cdots<k_M$. A special case of resonant soliton is given in Fig. \ref{line_soliton} (right). 
\begin{figure}[htb]
\centering
\includegraphics[scale=0.4]{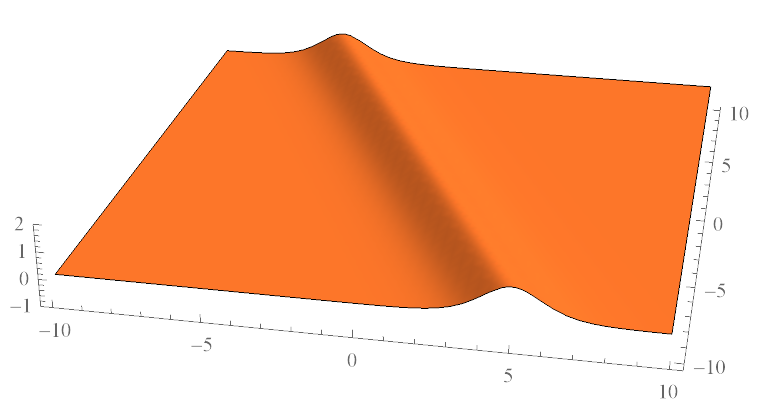}
\includegraphics[scale=0.27]{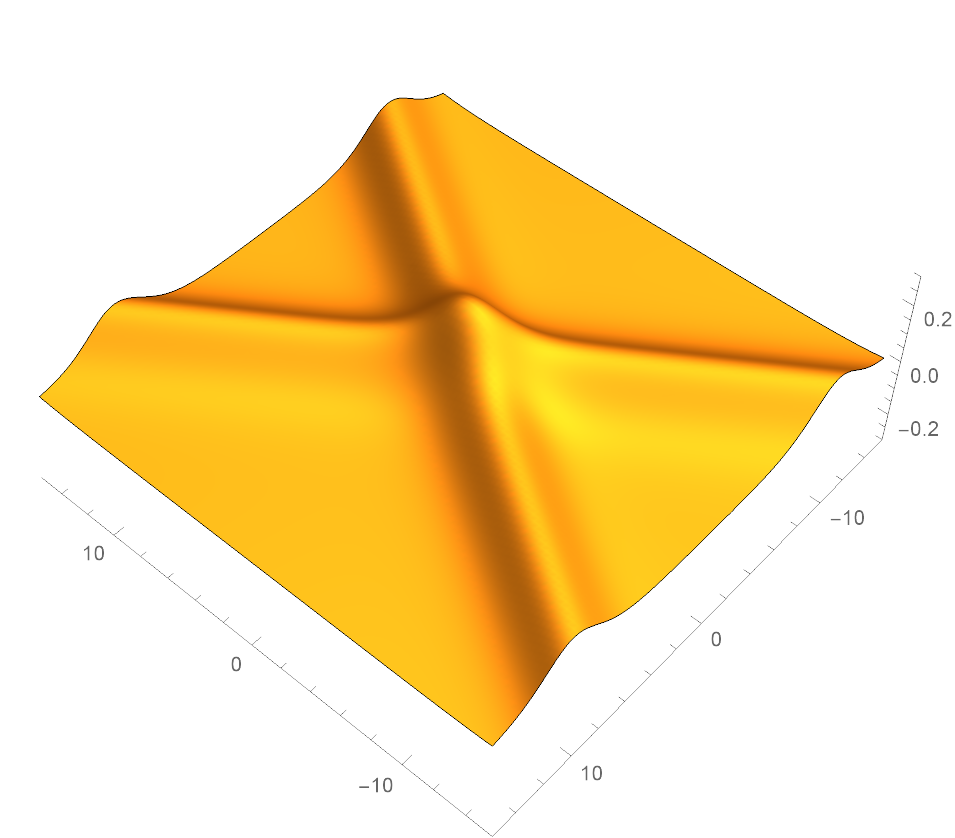}
\includegraphics[scale=0.4]{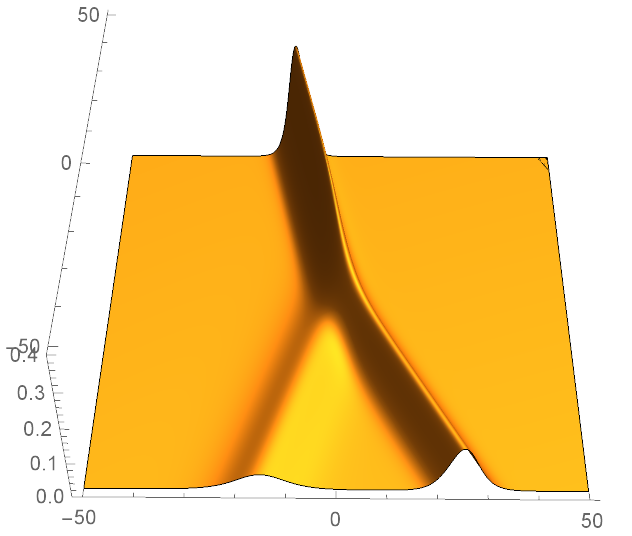}
\caption{Left: One line-soliton solution \eqref{soliton family} with $k_{1}=-0.5$, $k_{2}=1$ and $t=0$. This solution divides the plane into two regions according to the sum \eqref{linesoliton}, and on each region a different exponential dominates.
Center: A 2-soliton of KP with $k_{1}=-1$, $k_{2}=-0.5$, $k_{3}=0.5$ and $k_{4}=1$, at time $t=0$. Right: A $Y$-soliton characterized by $\Theta= \exp\left(\theta_1\right)+ \exp\left(\theta_2\right) +\exp\left(\theta_3\right)$ in \eqref{resonant}. Notice that the coefficients $a_j$ are set equal to 1, so that the three solitons meet at the origin at $t = 0$. Here, $\left(k_{1}, k_{2}, k_{3}\right) = \left(-0.3, 0, 0.5\right).$}\label{line_soliton}
\end{figure}

\medskip

Now we recall the KP 2-soliton. In this case $\Theta = \hbox{Wr}\left(\Theta_{1}, \Theta_{2}\right)$, where $\hbox{Wr}$ is the Wronskian of two functions $\Theta_1=\exp\left(\theta_{1}\right) + \exp\left(\theta_{2}\right)$  and $\Theta_2=\exp\left(\theta_{3}\right) + \exp\left(\theta_{4}\right)$ being 1-soliton phases. Calculating the phase $\Theta$, one obtains the classical formula
\begin{equation}\label{2soliton}
\begin{aligned}
\Theta =&~{} \left(k_{3} - k_{1}\right)\exp\left(\theta_{1} + \theta_{3}\right) + \left(k_{4} - k_{1}\right)\exp\left(\theta_{1} + \theta_{4}\right) \\
&~{} + \left(k_{3} - k_{2}\right)\exp\left(\theta_{2} + \theta_{3}\right) + \left(k_{4} - k_{2}\right)\exp\left(\theta_{2} + \theta_{4}\right).
\end{aligned}
\end{equation}
In order to ensure the positivity and nondegeneracy of $\Theta$, we require $ k_1<k_2<k_3<k_4$. See Fig. \ref{line_soliton} (center) for further details on the family of KP 2-solitons. 

\subsection{Main Results} This paper represents a departure from previous approaches, in a sense to be explained below. We adopt the perspective that each KP soliton should adhere to a specific ``variational'' characterization, manifested in their critical points of a suitable nonlinear functional. Our primary aim is to offer clear and straightforward characterizations of the most distinct KP solitons using simple ``trapping'' functionals that exhibit rigidity properties. This represents an initial step in the direction previously outlined. We believe that this concept holds promising potential for applications not only in elucidating more complex KP solutions but also in other related dispersive models. 

\medskip

Our first result is a characterization of the KdV line soliton as a KP solution.

\begin{thm}\label{MT1}
 Let $u$ be a smooth solution to KP \eqref{eq:KP} of the form \eqref{eqn:FT}, with a smooth profile $F$ such that $F\left(1\right)=0$, $F'\left(1\right)=1$, $F''\left(1\right)=-1$, and $F'''\left(1\right)=2$. 
Then   $u$ is a KdV soliton and $F=\log$ if and only if  $H\left(\Theta\right) = Ai\left(\Theta\right) = W_{y} \left(\Theta\right) = 0.$ 
\end{thm}

\begin{rem}
The four initial conditions on $F$ may seem extremely demanding, however they are naturally explained by the fourth order equation representing \eqref{eq:KP}. Consequently, in order to determine $F$, the four derivatives on $x$ at \eqref{eq:KP} induce corresponding initial conditions for $F$ and its three first derivatives. It is also easy to see that different initial conditions may lead to other solutions ($F$ periodic, for instance), as it happens in the simpler KdV case. See \cite{KKS2,AMP} for instance for examples of space-periodic profiles $F$.
\end{rem}

\begin{rem}
It will be proved below (see Lemma \ref{HWxWy}) that the Heat condition $H\left(\Theta\right)=0$ implies $ W_{x}^F\left(\Theta\right) = W_{y}^F\left(\Theta\right)$ and $ W_{x}\left(\Theta\right) = W_{y}\left(\Theta\right)$. So, additionally, one has $W_x\left(\Theta\right)=0$.
\end{rem}

As a natural consequence of Theorem \ref{MT1}, we obtain the following uniqueness result of KdV solitons as extended KP solutions. As usual, we require
\begin{equation}\label{F_conds}
F\left(1\right)=0, \quad F'\left(1\right)=1, \quad F''\left(1\right)=-1, \quad \hbox{and} \quad F'''\left(1\right)=2.
\end{equation}

\begin{cor}
Assume \eqref{F_conds}. Let $u$ be a nontrivial KP solution of the form \eqref{eqn:FT} such that $\Theta>0$ is a solution to $H\left(\Theta\right) = Ai\left(\Theta\right) = W_{y} \left(\Theta\right) = 0.$ Then $F=\log$ and $u=Q_k$ in \eqref{Qk} for some $k\in \mathbb R-\left\{0\right\}$.
\end{cor}

Theorem \ref{MT1} can be extended to general KP line-solitons \eqref{soliton family} as in Fig. \ref{line_soliton} left, and not necessarily of KdV type. In this case, we denote them oblique line-solitons. From now on, the structure $\Theta W_y(\Theta)$ from \eqref{Wronskians} will be essential. 

\begin{thm}\label{MT1b}
Let $u$ be a smooth solution to \eqref{eq:KP} of the form \eqref{eqn:FT}, with a smooth profile $F$ such that \eqref{F_conds} is satisfied.  Then $u$ is an oblique line-soliton of the form \eqref{linesoliton}-\eqref{soliton family} and $F=\log$ if and only if  $H\left(\Theta\right) = Ai\left(\Theta\right) = 0$, and 
\begin{equation}\label{origen}
\Theta W_{x}\left(\Theta\right) = \Theta W_{y}\left(\Theta\right) = A\left(t,x\right) \exp \left( k\left(t,x\right) y \right), 
\end{equation}
for some particular functions $A,k>0$ everywhere.
\end{thm}

It is worth to mention that the conditions $H\left(\Theta\right) = Ai\left(\Theta\right) = 0$ do not ensure the validity of Theorem \ref{MT1b}. Indeed, it will be proved that the class of phases satisfying these two conditions is large enough to contain many multi-soliton solutions, such as the $Y$ structure defined below, which is not a line soliton. Consequently, \eqref{origen} is {\color{black} just} a necessary condition.

\begin{rem}
Notice that $\Theta W_{x}\left(\Theta\right)$ remains unchanged after a Galilean transformation (see \eqref{Galilean} in Appendix \ref{AppA}). This is not the case for $\Theta W_{y}\left(\Theta\right)$. However, one can prove that if $\Theta$ is of the form \eqref{linesoliton}, then coincidentally for $\beta=k_1+k_2$ one has that the Galilean transformation of $\Theta$, denoted $\Theta_\beta$ (see \eqref{galilean_theta}), satisfies $\Theta_\beta W_{y}\left(\Theta_\beta \right) =0$. It is an interesting problem to fully elucidate the role of Galilean transforms in the classification of solitons as proposed in this paper.
\end{rem}

\medskip

Theorem \ref{MT1b} puts in evidence an intriguing new mathematical structure, a natural finite-dimensional cone in the variable $y$. 

\begin{defn}[Invariant $\mathcal W_n$ cones]
Given $n\in \left\{1,2,3,\ldots\right\}$,  consider the linear, positively generated cone
\begin{equation}\label{Wn}
\mathcal W_n :=\left\{ \sum_{j=1}^n  a_j\left(t,x\right)\exp\left(k_j \left(t,x\right) y\right)  ~ : ~  \begin{matrix}  \exists \, 0\leq k_1\left(t,x\right)<k_2\left(t,x\right)\cdots <k_n\left(t,x\right), \\
 \exists \, a_1\left(t,x\right), a_2\left(t,x\right),\ldots, a_n\left(t,x\right) \geq 0 \end{matrix}  \right\}. 
\end{equation}
\end{defn}
Using the terminology above,  \eqref{origen} can be recast as $\Theta W_{x}\left(\Theta\right) = \Theta W_{y}\left(\Theta\right) \in \mathcal W_1$. 

\medskip

The space $\mathcal W_n$ has interesting properties, in particular its behavior under the nonlinear mapping $\Theta W_y\left(\Theta\right)$ is key to understand complex multi-soliton structures. First of all, $\Theta\in \mathcal W_M$ implies that $\Theta W_y\left(\Theta\right) \in \mathcal W_{\frac12M\left(M-1\right)}$ (Lemma \ref{WM}).  Second, there is a natural ``kernel'' given by the function $\exp\left(ky\right)$, $k$ arbitrary: one has $\exp\left(ky\right) W_{y}\left(\exp\left(ky\right)\right)=0 $ for any $k=k\left(t,x\right)$. Additionally, under the gauge $\mathcal W_M \ni \Theta \longmapsto \Theta_k:= \exp\left(k y\right) \Theta$, one has $\Theta_k W_y\left(\Theta_k\right)=\exp\left(2ky\right) \Theta W_y\left(\Theta\right) \in \mathcal W_{\frac12M(M-1)}$, revealing that \emph{there is no unique nontrivial solution} to the set inclusion $\Theta W_{y}\left(\Theta\right) \in \mathcal W_{\frac12M\left(M-1\right)}$ for $\Theta\in \mathcal W_M$.  One way to repair this gauge freedom is to ask for $\Theta\left(y=0\right)$ and $\Theta_y\left(y=0\right)$ uniquely defined, as it is done below.
For natural reasons, we will work in a slightly larger class of resonant phases than \eqref{resonant}, given by 
\begin{equation}\label{resonant_0}
\Theta\left(t,x,y\right) =  \sum^{M}_{i=1}\left(a_{i,1}\exp\left( -k_{i}x + k^{2}_{i}y - k^{3}_{i}t \right) +a_{i,2}\exp\left( k_{i}x + k^{2}_{i}y + k^{3}_{i}t \right) \right),
\end{equation}
with coefficients $  k_1<k_2<\cdots <k_M$ and $a_{i,j}\geq 0$.

\begin{thm}[Resonant multisolitons]\label{MT2}
Let $u$ be a solution of \eqref{eq:KP} of the form \eqref{eqn:FT} with a smooth real-valued phase $\Theta>0$ satisfying for $k=0,1,2,3,$
\begin{equation}\label{hypohypo}
\partial_x^k \Theta\left(t,0,0\right), \quad \partial_x^k \partial_y \Theta\left(t,0,0\right) \quad \hbox{uniquely prescribed.}
\end{equation}
Assume that the smooth profile $F$ satisfies \eqref{F_conds}. Then $\Theta$ corresponds to an $M$ resonant multi-soliton \eqref{resonant_0} and $F=\log$ if and only if $\Theta$ satisfies $H\left(\Theta\right) = Ai\left(\Theta\right) = 0$  and  $\Theta W_y\left(\Theta\right)= \Theta W_x\left(\Theta\right)$ has a unique value in $\mathcal W_{\frac12M\left(M-1\right)}$. 
\end{thm}

\begin{rem}\label{cagazo}
Notice that the condition \eqref{hypohypo} only requires information of $\Theta$ at $x=y=0$. This is necessary to ensure the uniqueness of the solution $\Theta$ to $\Theta W_y\left(\Theta\right)= \Theta W_x\left(\Theta\right) \in \mathcal W_{\frac12M\left(M-1\right)}$ within the class $\mathcal W_{M}$. Due to the exponential growth in $x$ of the functions solving the equations for $\Theta$, a less demanding sufficient condition will require to establish a Cauchy theory for the linear Airy equation $\Theta_t -\Theta_{xxx} =0$ with initial conditions in the distributional class $D'(\mathbb R_x)$, a problem that is far from trivial due to the oscillatory character of the Airy kernel.
\end{rem}

The key in the proof of Theorem \ref{MT2} is the property that  $\Theta\in \mathcal W_M$ implies $\Theta W_y\left(\Theta\right) \in \mathcal W_{\frac12M\left(M-1\right)}$ (Lemma \ref{WM}). This property allows us to estimate the size of the cone representing the image of $\mathcal W_M$ under the nonlinear mapping $\Theta W_y\left(\Theta\right)$. Then one has to establish a sort of uniqueness in the representation of $\Theta$, which in the case $M=1$ is easy to obtain (see Theorem \ref{MT1b}),  but it is not known if it is maintained in general. Under the additional prescribed data at $x=y=0$, uniqueness is recovered and Theorem \ref{MT2} establishes the equivalence between resonant multi-solitons and Airy-Heat type phases with finite dimensional Wronskians.

\medskip

Resonant solitons of $Y$-type, or Miles type (see Fig. \ref{line_soliton} right), are essential KP solutions included in the previous result. These are usually given by \cite{Kodama2017} ($k_1<k_2<k_3$)
\[
\Theta= a_1\exp\left(\theta_1\right)+a_2 \exp\left(\theta_2\right) +a_3\exp\left(\theta_3\right), \quad a_i>0, \quad  \theta_i:= k_{i}x + k^{2}_{i}y + k^{3}_{i}t .
\]
Sometimes referred as resonant interacting 3-solitons, Theorem \ref{MT2} states that they are characterized as having zero Heat and Airy operators, but having $\Theta W_y\left(\Theta\right)$ with one more dimension than the one obtained in Theorem \ref{MT1b}, measured in terms of the subspace $\mathcal W_2$. Additionally, in this case $\Theta \in\mathcal W_3$ and $\Theta W_y\left(\Theta\right) \in\mathcal W_3$, being the only phases (as far as we understand) that have this invariance under the nonlinear mapping $\Theta \longmapsto \Theta W_y\left(\Theta\right)$.

\medskip

Our last result concerns the characterization of crossed 2-solitons, with $\Theta$ given in \eqref{2soliton} (see also Fig. \ref{line_soliton} right panel). Recall the subspace $\mathcal W_n$ defined in \eqref{Wn}.

\begin{thm}[2-solitons]\label{MT3}
Let $u$ be a solution of \eqref{eq:KP} of the form \eqref{eqn:FT} with $F=\log$ and with a smooth real-valued phase $\Theta>0$ satisfying \eqref{hypohypo} and being at most exponentially growing in $x$: there are $C_1,c_2>0$ such that 
\[
|\Theta(t,x,y)|\leq C_1 e^{c_2|x|}.
\] 
Then $\Theta$ corresponds to a 2-soliton \eqref{2soliton} with $ k_1<k_2<k_3<k_4$ if and only if  
\begin{enumerate}
\item $H\left(\Theta\right), Ai\left(\Theta\right)$ are contained in $\mathcal W_4$, 
\item $ \Theta W_y\left(\Theta\right), \Theta W_x\left(\Theta\right)$ describe unique elements in $\mathcal W_5$, and
\item   $Ai\left(\Theta\right) = \frac32 \partial_x H\left(\Theta\right)$.  
\end{enumerate}
\end{thm}

\begin{rem}
Notice that Theorem \ref{MT3} does not assume that $\Theta = \hbox{Wr} \left(\Theta_1,\Theta_2\right)$ (standard Wronskian of $\Theta_1$ and $\Theta_2$), with $H\left(\Theta_i\right)=Ai\left(\Theta_i\right)=0$ and $\Theta_1,\Theta_2 \in \mathcal W_2$. This is the standard and well-known definition of the 2-soliton that assumes the key Wronskian substructure. Here we lift that condition and only ask conditions on $\Theta$ itself. The Wronskian structure is recovered from the proof.
\end{rem}

\begin{rem}
Theorem \ref{MT3} can be recast as follows. It is easy to check that 2-solitons \eqref{2soliton} solve $\mathcal T\left(\Theta\right)=0$ when $F=\log$ (Corollary \ref{coro2soliton}). However,  this equation  has plenty of additional, more complicated solutions \cite{Kodama2017}, and a suitable characterization of the 2-soliton subclass is desirable. In that sense, Theorem \ref{MT3} and the $\mathcal W_n$ structure give a precise equivalence that separates 2-solitons of other more complex KP solutions.
\end{rem}

\subsection{Further results}  Our last comments are related to possible extensions of the results presented in this paper. We believe that with some work it is possible to give a suitable characterization of KP-I line-soliton and lumps in terms of particular phases. The challenge is to get a good understanding of the fact that KP-I lumps are degenerate soliton solutions, in a sense already described in \cite{AS}. 
However, in the Zakharov-Kuznetsov (ZK) case,
\begin{equation}\label{ZK}
-4\partial_t u  +\partial_{x_1}\left(\Delta u +  3 u^2\right) =0,
\end{equation}
where $u=u\left(t,x\right)$, $x=\left(x_1,x'\right)\in\mathbb R^d$, $x_1\in \mathbb R$, $x'\in \mathbb R^{d-1}$, one can say the following:
\begin{thm}[ZK case]\label{MT_ZK}
Let $u= 2\partial_{x_1}^2 F(\Theta) $ be a smooth solution of \eqref{ZK} with $\Theta>0$ and \eqref{F_conds} satisfied.  Then $u=Q_k$, $k>0$ (the KdV soliton) as in \eqref{Qk} and $F=\log$ if and only if 
\begin{equation}\label{nuevas}
Ai\left(\Theta\right)=W_1\left(\Theta\right) = W_{x_j}^F\left(\Theta\right)=0, \quad j=2,\ldots, d,
\end{equation}
and where
\[
W_1\left(\Theta\right):=\left(\partial_{x_1}^2\Theta\right)^2 - \partial_{x_1} \Theta \partial_{x_1}^4 \Theta, \quad W_{x_j}^F\left(\Theta\right):= \partial_{x_j}^2\Theta - F'\left( \Theta\right)\left(\partial_{x_j} \Theta\right)^2.
\]
\end{thm}
Similarly, in the modified Zakharov-Kuznetsov (mZK) case,
\begin{equation}\label{KdVmZKin}
\left( -4u_t+ \partial_{x_1}^3 u +6u^2 \partial_{x_1} u \right)+\partial_{x_1} \left(\Delta_{c} u\right) =0,
\end{equation}
we have the following result:
\begin{thm}[mZK case]\label{MT_mZK}
Let $u= 2\partial_{x_1}F\left(\Theta\right) $ be a smooth solution of \eqref{KdVmZKin} with smooth profile $F:\mathbb R\to \mathbb R$ satisfying $F\left(0\right)=0$ and $F$ strictly increasing in $\mathbb R$. Then a nontrivial $\Theta$ is a mKdV soliton  \eqref{ppal_form:mKdV}  and $F=2 \arctan$, if and only if 
\begin{equation}\label{nuevas_2}
Ai\left(\Theta\right)=W\left(\Theta\right) = \Lambda_{x_j}^F\left(\Theta\right)=0, \quad j=2,\ldots, d,
\end{equation}
and where
\[
W\left(\Theta\right) := \Theta_{xx}^2 -\Theta_x\Theta_{xx}, \quad \Lambda_{x_j}^F\left(\Theta\right):=\Theta _{x_jx_j} F'\left(\Theta\right)+\Theta_{x_j}^2 F''\left(\Theta\right).
\]
\end{thm}

\begin{rem}[The KdV and ZK cases]\label{rem:KdV}
An important outcome of the proofs will be its robust character. Indeed, Theorem \ref{MT1} has natural counterparts in the case of the 1D KdV model and ZK model (Theorem \ref{MT_ZK}), where similar notions of Airy and Heat operators are introduced. See Appendix \ref{rem:KdVmKdV} and Section \ref{sec:ZK} respectively, for the complete details.
\end{rem}

\begin{rem}[The mKdV and mZK cases]\label{rem:mKdV}
In concordance with the previous remark, it should be natural that the same ideas can be applied as well for the so-called mKdV and mZK models. It turns out that this models has a particular rich structure involving even more demanding special solutions that complicates matters. We provide for completeness a suitable short treatment of the problem in  Appendix \ref{rem:KdVmKdV} and Subsection \ref{sec:mZK}. See also \cite{FFMP} for a detailed account of the difficulties found when dealing with mKdV models.
\end{rem}

\subsection{Previous results}
We mention some key results obtained for KP-II models during the past years. Bourgain \cite{Bourgain1993} showed that KP is globally well-posed (GWP) in $L^2\left(\mathbb{R}^2\right)$ (see also Ukai \cite{Ukai} and I\'orio-Nunes \cite{IN} for early results). Bourgain's result was later improved by Takaoka-Tzvetkov \cite{Takaoka-Tzvetkov-2001}, Isaza-Mejia \cite{Isaza-Mejia-2001}, Hadac \cite{Hadac-2008} and Hadac-Herr-Koch \cite{Hadac-Herr-Koch-2009}. Molinet, Saut and Tzvetkov \cite{Molinet-Saut-Tzvetkov-2011} proved global well-posedness of KP along the KdV line-soliton in $L^2\left(\mathbb{R}\times \mathbb{T}\right)$  and $L^2\left(\mathbb R^2\right)$. 

\medskip

The long time behavior of small KP solutions has been studied by Hayashi-Naumkin-Saut and Hayashi-Naumkin \cite{HNS,HN}, see also recent improvements by Niizato \cite{Ni}. de Bouard-Martel \cite{dBM} showed that KP has no ``lump'' structures, namely compact-in-space solutions. Any KdV soliton becomes an (infinite energy) line-soliton solution of KP. This structure is stable, as proved by Mizumachi and Tzvetkov \cite{MT}, and asymptotically stable, see Mizumachi \cite{Mizu1,Mizu2}. The linear stability of the 2-soliton was recently proved by Mizumachi \cite{Mizu3}. Finally, Izasa-Linares-Ponce \cite{ILP_2} showed propagation of regularity for this model.  

\medskip

Numerical studies of KP solutions have been performed in \cite{KS2}, see also \cite{KS3,KS} for a detailed account of the KP literature via PDE methods. Multi-line-soliton structures are known to exist via Inverse Scattering Transforms (IST) methods \cite{AS,Kono}. Their stability in rigorous terms has been recently considered by Wu \cite{Wu,Wu2}.  See also \cite{RT1,RT2} for a detailed theory of transversal stability and instability of PDE models of water waves, that applies to one line-solitons as the ones studied in Theorem \ref{MT1}-\ref{MT1b}. The description of small data can be found in \cite{W,Sung}. Recently, and following \cite{MMPP}, in \cite{MMPP1} it was shown that every solution $u$ of KP obtained from arbitrary initial data $u_0$ in $L^2\left(\mathbb R^2\right)$ satisfies $\liminf_{t\to \infty} \int_{K} u^2\left(t,x,y\right)\,\mathrm{d}x\mathrm{d}y=  0$, with $K \subseteq \mathbb R^2$ compact. Finally, Kenig and Martel \cite{KM} showed that for any $\beta>0$ and  initial data small in $L^1\cap L^2$, $\lim_{t\to \infty} \int_{x>\beta t} u^2\left(t,x,y\right)\,\mathrm{d}x\mathrm{d}y=  0$.

\subsection*{Organization of this work} This work is organized as follows. In Section \ref{Sec:2} we introduce the basic elements needed for the proof of the main results. Section \ref{Sec:3} is devoted to recall standard results on KP solitons. In Section \ref{Sec:5} we present general results and properties about phases $\Theta$ satisfying \eqref{KP3partes}. On the other hand, Section \ref{Sec:6} presents properties satisfied by classical soliton structures. Finally, in Section \ref{Sec:7} we prove the main results, Theorems \ref{MT1}, \ref{MT1b}, \ref{MT2}, and \ref{MT3}. Section \ref{sec:ZK} considers the ZK and mZK cases. Appendix \ref{AppA} contains some useful computations needed in the paper, and finally Appendix \ref{rem:KdVmKdV} is devoted to the proof of similar results in the case of KdV and mKdV models.

\subsection*{Acknowledgments} Part of this work was done while the third and fourth authors were visiting UACh (Valdivia, Chile), Georgia Tech and Texas A\&M Mathematics departments. We thank these institutions for their warm hospitality and support. 

\section{Preliminaries}\label{Sec:2}

In this section we first mention some simple but important facts related to solutions of the KP model.

\subsection{Invariances} 
First of all, it is clear that \eqref{KP3partes} is invariant under space and time shifts in the phase $\Theta$. Additionally,  \eqref{KP3partes} has the natural scaling invariance associated to KP:
\[
\Theta \left(t,x,y\right) \longrightarrow \Theta\left(\lambda^3 t,\lambda x, \pm \lambda^2 y\right), \quad \lambda>0.
\]
KP-II obeys the Galilean Transform \cite{MMPP1}
\[
    u\left(t, x,y\right) \longrightarrow  u\left(t, x - \frac{4\beta}{3}y + \frac{4\beta ^{2}}{3}t, y - 2\beta t \right) =: u_{\beta}\left(t,\tilde{x},\tilde{y}\right), \quad \beta\in\mathbb R.
\]
This invariance naturally translates into the phase $\Theta$
\begin{equation}\label{galilean_theta}
    \Theta_{\beta}\left(t, x,y\right) = \Theta \left(t, x-\frac{4\beta}{3}y +\frac{4\beta^{2}}{3}t,y-2\beta t \right) = \Theta\left(t, \tilde{x},\tilde{y}\right), 
\end{equation}
that also satisfies \eqref{KP3partes} provided $\Theta$ does. 

\subsection{Kernel in the solitonic representation}
Notice that the formulation \eqref{eqn:FT} involves a nontrivial kernel. 

\begin{lem}
One has  $\partial^{2}_{x} \log\left(\Theta\left(t,x,y\right)\right) = 0$ for all $(t,x,y)\in \mathbb R^3$ if and only if the phase $\Theta$ satisfies $\Theta\left(t,x,y\right) = \exp\left( a\left(t,y\right)x + b\left(t,y\right)\right)$, for any well-defined functions $a$ and $b$.
\end{lem}
\begin{proof} It is a consequence of direct integrations.
\end{proof}
As a consequence of the previous result, if $F=\log$, any phase that can be expressed in the form $\Theta = \exp\left(f\left(t,x,y\right)\right)$ with $f\left(t,x,y\right)$ any smooth linear affine function in the variable $x$ gives a trivial solution. This is the kernel of the operator $\partial^{2}_{x} \log$ which permits to construct the KP multi-soliton solutions, based on this first seed.

\subsection{Quick review of the simplest KP solitons}\label{Sec:3} Recall the line soliton introduced in \eqref{linesoliton} and \eqref{soliton family}. The line that separates these regions correspond to $\theta_{1} = \theta_{2}$. In this case, this line is called $\left[1,2\right]$-soliton \cite{Kodama2017}.  In the general case of resonant structures such as \eqref{resonant}, they are called $\left[i,j\right]$-solitons and are formed by the intersection of the corresponding exponentials $i$ and $j$. Each $\left[i,j\right]$-soliton has the same local structure as a line-soliton, which is described by the form \cite[p. 5]{Kodama2017}
\begin{equation*}
    u = A_{\left[i,j\right]}\sech^{2} \left(\frac{1}{2}\left(K_{\left[i,j\right]}\cdot \left(x,y\right) - \Omega_{\left[i,j\right]}t + \Theta^{0}_{\left[i,j\right]}\right)\right),
\end{equation*}
with $\Theta^{0}_{\left[i,j\right]}$ a constant. The parameters $A_{\left[i,j\right]}$, $K_{\left[i,j\right]}$ and $\Omega_{\left[i,j\right]}$ are known as the amplitude, wave-vector and frequency, respectively, and are defined by
\[
\begin{aligned}
    A_{\left[i,j\right]} &= \frac{1}{2}\left(k_{j} - k_{i}\right)^{2},\\
    K_{\left[i,j\right]} &= \left(k_{j} - k_{i}, k^{2}_{j} - k^{2}_{i}\right) = \left(k_{j} - k_{i}\right)\left(1, k_{j} + k_{i}\right),\\
    \Omega_{\left[i,j\right]} &= -\left(k^{3}_{j} - k^{3}_{i}\right) = -\left(k_{j} - k_{i}\right)\left(k^{2}_{i} + k_{i}k_{j}+ k^{2}_{j}\right).
\end{aligned}
\]
If one denotes $\psi_{\left[i,j\right]}$  the angle measured counterclockwise between the $\left[i,j\right]$-soliton and the $y$-axis, then $\tan\left(\psi_{\left[i,j\right]}\right)=k_{i}+k_{j}$. 

\medskip

Now consider the case of resonant solitons \eqref{resonant} with $M=3$.  As in the previous case, it is possible to determine the dominant exponentials and analyze the structure of the solution in the $xy$-plane. Indeed, the line-soliton at positive $y$, corresponding to the $\left[1,3\right]$-soliton, is located on the phase transition $x + cy = constant$ with direction parameter $c = k_{ 1} + k_{3}$. In the same way, the line-soliton located at $y$ negative, corresponding to the $\left[1,2\right]$-soliton and $\left[2,3\right]$-soliton are located over their respective phase transitions with direction parameter $c = k_{1} + k_{2}$ and $c = k_{2} + k_{3}$, respectively. The \emph{resonance condition} of these three line-solitons is given by
\begin{equation*}
    K_{\left[1,3\right]} = K_{\left[1,2\right]} + K_{\left[2,3\right]}, \quad \Omega_{\left[1,3\right]} = \Omega_{\left[1,2\right]} + \Omega_{\left[2,3\right]}, 
\end{equation*}
and both are satisfied when $K_{\left[i,j\right]} = \left(k_{j} - k_{i},k^{2}_{j} - k^{2}_{i}\right)$ and $\Omega_{\left[i,j\right]} = -\left(k^{3}_{j} - k^{3}_{i}\right)$. 

\medskip

In a simple way it is possible to extend the previous result for a general solution constructed from a $\Theta$ composed with an arbitrary number $M$ of exponentials, as in \eqref{resonant}.

\begin{thm}[\cite{Kodama2017}, Proposition 1.2]
Let $\Theta$ be an $M$ resonant phase as in \eqref{resonant}. Then the solution $u$ has the following asymptotic characteristics:
\begin{enumerate}
    \item[$(i)$] For values of $y\gg 1$, there is only one soliton of the form $\left[1,M\right]$-soliton.
    \item[$(ii)$] For values of $y\ll -1$, there are $M-1$ line-solitons of the form $\left[k,k+1\right]$-soliton, with $k = 1,2,\dots,M-1$, located counter-clockwise from the negative part to the positive part of the $x$-axis.
\end{enumerate}
\end{thm}

\begin{defn}
Let $N < M$. The Grassmannian $\Gr\left(N,M\right)$ are all the matrices which represents an $N$-dimensional sub-vectorial space contained in a $M$-dimensional vectorial space.  
\end{defn}
It can be verified the following isomorphism: $\Gr\left(N,M\right)\cong\GL_{N}\left(\mathbb R\right)\setminus M_{N\times M}\left(\mathbb R\right)$, where $\GL_{N}\left(\mathbb R\right)$ consists of those matrices with dimension $n\times n$, real coefficients and whose determinant is non-zero, and $M_{N\times M}(\mathbb{R})$ is the set of all $N\times M$ full-rank matrices.

\medskip

Let $\left\{ \Theta_i : i = 1, \ldots , N\right\}$ be linearly independent solutions of $Ai\left(\Theta\right)=H\left(\Theta\right)=0$.  Let  $\Theta:= \Wr\left(\Theta_{1}, \ldots , \Theta_{N}\right)$ be the Wronskian of the functions $\Theta_{i}$ with respect to the variable $x$ (usually called a $\tau$-function). It is well-known and not difficult to see (see \cite{Kodama2017}) that $u\left(t,x,y\right)=2\partial^{2}_{x} \log \left(\Theta\left(t,x,y\right)\right)$ satisfies the KP equation. A particular choice for $\Theta_i$ is given by 
\[
 \Theta_{i}\left(t,x,y\right)=\sum^{M}_{j=1}a_{ij}\exp\left(\theta_{j}\left(t,x,y\right)\right), \quad \hbox{with} \quad \theta_{j}=k_{j}x+k^{2}_{j}y+k^{3}_{j}t,
\]
where $A := \left(a_{ij}\right)$ is an $N\times M$ matrix. Thus each KP soliton expressed in the previous form is parametrized by $M$ parameters $\left(k_{1}, \ldots , k_{M}\right)$ and an $N \times M$ matrix $A$. The matrix $A$ will be identified as a point of the real Grassmannain $\Gr\left(N, M\right)$.

\subsection{Linear ODEs related to $F$}

In this section we describe some ODE theory related to the equations that $F$ must satisfy. Indeed, from \eqref{rho_def0} consider the auxiliary variable
\[
\rho\left(s\right):= F''\left(s\right) + F'^2\left(s\right).
\]
\begin{lem}\label{dem:ODE1}
If $F=\log$, then $\rho=0$. If $\rho=0$ and $F\left(1\right)=0$, $F'\left(1\right)=1$, then $F=\log.$
\end{lem}

\begin{proof}

It follows from directly  solving the ODE $\rho\left(s\right)=0$.
\end{proof}

\begin{lem}\label{dem:ODE2}
Let $F$ be a smooth profile such that $F\left(1\right)=0$ and $F'\left(1\right)=1$. Then the following are satisfied, for any $s\geq 1$:
\begin{enumerate}
\item[$\left(i\right)$] If $\left(F'' + F'^2\right)\left(s\right)=0$, then $\left(F'''' + 6F''^{2}\right)\left(s\right)=0$. 
\smallskip
\item[$\left(ii\right)$] If now $F''\left(1\right)=-1$ and $F'''\left(1\right)=2$, and $\left(F'''' + 6F''^{2}\right)\left(s\right)=0$, then one has $\left(F'' + F'^2\right)\left(s\right)=0$.
\smallskip
\item[$\left(iii\right)$] Assume now that $F\left(1\right)=0$, $F'\left(1\right)=1$, $F''\left(1\right)=-1$, $F'''\left(1\right)=2$, $s \in\left[1,\infty\right)$ and
\begin{equation}\label{EDO_rho_new}
\begin{aligned}
& h_1(s) \left(F'''' + 6F''^{2}\right)(s)+ h_2(s) \left(F'' + F'^{2}\right)'(s)+ h_3(s)\left(F'' + F'^{2}\right)(s)=0,
\end{aligned}
\end{equation}
for some continuous $h_1,h_2,h_3:[1,\infty)\to \mathbb R$, $h_1>0$. Then $F= \log.$  
\end{enumerate}
\end{lem}
\begin{proof}
The proof of $\left(i\right)$ is a consequence of the following identity. One has 
\begin{equation}\label{ODE_rho}
\left(F''+F'^2\right)'' -2F'\left(F''+F'^2\right)'  +4F''\left(F''+F'^2\right)= F''''+6F''^2.
\end{equation}
Therefore, $F'' + F'^2=0$ implies $F'''' + 6F''^{2}=0$, proving $\left(i\right)$. 

\medskip

Proof of $\left(ii\right)$. In the case where $F''''+6F''^2 = 0$ one has that \eqref{ODE_rho} can be written in terms of $\rho$ as
\begin{equation}\label{edo_rho}
\rho'' -2F' \rho' +4F'' \rho= 0.
\end{equation}
This is a second order linear ODE with continuous coefficients. This solution has a  basis of solutions of dimension 2, say $\left\{\rho_1,\rho_2\right\}$. Consequently,
\[
\rho\left(s\right)= C_1 \rho_1\left(s\right) +C_2 \rho_2\left(s\right), \quad C_1,C_2 \in\mathbb R.
\]
Recall that $F\left(1\right)=0$ and $F'\left(1\right)=1$, and $\rho= F'' + F'^2,$ $\rho' = F''' +2F'F''$. Since $F''\left(1\right)=-1$ and $F'''\left(1\right)=2$, one has $\rho\left(1\right)=\rho'\left(1\right)=0$, leading to $\rho\left(s\right)=0$ for all $s\geq 1$, proving $\left(ii\right)$. 

\medskip

Proof of $\left(iii\right)$. Thanks to \eqref{ODE_rho}, equation \eqref{EDO_rho_new} is equivalent to 
\[
h_1(s) \left(\rho'' \left(s\right)- 2F'\left(s\right)\rho' \left(s\right) + 4F''\left(s\right)\rho\left(s\right)\right) +h_2(s) \rho' \left(s\right)+ h_3(s)\rho \left(s\right)=0,
\]
with $\rho= F''+F'^2$. This system is analogous to \eqref{edo_rho}, and proceeding as in step $\left(ii\right)$ we get from Lemma \ref{dem:ODE1} that $\rho\left(s\right)=0$ for all $s\geq 1$, and $F=\log$. 
\end{proof}

Recall from Definition \ref{Thetas} that $\left(\Theta,F\right)$ are of $\mathcal T$-type if \eqref{eqn:T_type} is satisfied.  The following corollary is a direct result of $\left(iii\right)$ in {\color{black} Lemma \ref{dem:ODE2}.}

\begin{cor}
Let $u = 2\partial^{2}_{x}F\left(\Theta\right)$ solution of \eqref{eq:KP}, with $F\left(1\right)=0$, $F'\left(1\right)=1$, $F''\left(1\right)=-1$, and $F'''\left(1\right)=2$. Then $\Theta$ is $\mathcal T$-type if and only if $F=\log$.
\end{cor}
\begin{proof}
    Direct from \eqref{casicasi} and $\left(iii\right)$ in Lemma \ref{dem:ODE2}.
\end{proof}

\section{Wronskian structures}\label{Sec:5}

\subsection{Simple Phases}

The purpose of this section is to establish simple properties for the smooth phases satisfying \eqref{KP3partes}. We start with the following
\begin{lem}\label{FandPhase}
Assume that $\Theta>0$. Then the following are satisfied:
\begin{enumerate}
\item[(i)] If $F''\left(\Theta\right)\Theta^{2}_{x} + F'\left(\Theta\right)\Theta_{xx} = 0$ for all $\left(t,x,y\right) \in \mathbb R^{3}$, then $u$ is the trivial solution. 
\medskip
\item[(ii)] If $F = \log$, then a positive phase $\Theta$ that satisfies $\Theta \Theta_{xx} = \Theta^{2}_{x}$, for all $\left(t,x,y\right) \in \mathbb R^{3}$, gives the trivial solution.
\medskip
\item[(iii)] If $F = \log$ then every $\Theta$ of the form $\Theta\left(t,x,y\right) = A\left(t,y\right)\exp\left(kx\right)$ with $A>0$, gives the trivial solution.
\end{enumerate}
\end{lem}

\begin{proof}
Proof of (i). Direct from the fact that $u = 2\partial^{2}_{x}F\left(\Theta\right) = 2\left(F''\left(\Theta\right)\Theta^{2}_{x} + F'\left(\Theta\right)\Theta_{xx}\right)$. Proof of (ii): also direct from the fact that $ F''\left(\Theta\right)\Theta^{2}_{x} + F'\left(\Theta\right)\Theta_{xx} =\frac{\Theta \Theta_{xx} - \Theta^{2}_{x}}{\Theta^{2}}$, and the previous result. Finally, (iii) is consequence of the fact that in this case $\Theta F''\left(\Theta\right) = -F'\left(\Theta\right)$, for all $\left(t,x,y\right) \in \mathbb R^{3}$.
\end{proof}

\begin{rem}
Note that in Lemma \ref{FandPhase} we are assuming that the conditions are satisfied for all $\left(t,x,y\right) \in \mathbb R^{3}$. In the case where the conditions are satisfied just for certain points $\left(t,x,y\right)\in\mathbb R^{3}$, this points will correspond to zeros of the associated solution $u$.
\end{rem}

\begin{cor}
  If $\Theta>0$ is a smooth phase such that $\Theta \Theta_{xx} - \Theta^{2}_{x} = 0$, then $W_{x}\left(\Theta\right)= 0$.
\end{cor}
\begin{proof}
Taking derivative in $x$, $0= \Theta_{x} \Theta_{xx} + \Theta \Theta_{xxx} - 2\Theta_{x} \Theta_{xx} =  \Theta \Theta_{xxx} - \Theta_{x} \Theta_{xx}$. Once again, taking derivative in $x$, $0=\Theta_{x} \Theta_{xxx} + \Theta \Theta_{xxxx} - \Theta^{2}_{xx} - \Theta_{x} \Theta_{xxx} = \Theta W_{x}\left(\Theta\right)$. 
\end{proof}

\subsection{General vs. simple Wronskians}
It is noted that in Definition \ref{Thetas}, Wronskians \eqref{Wronskians} do not coincide with the expected value if taken from \eqref{KP3partes}. Indeed, the correct definition should be
\begin{equation}\label{general_W}
\begin{aligned}
    W^{F}_{y}\left(\Theta\right) :=&~{} \Theta_{yy} - F'\left(\Theta\right)\Theta^{2}_{y}, \\
    W^{F}_{x}\left(\Theta\right) :=&~{} \Theta_{xxxx} - F'\left(\Theta\right)\Theta^{2}_{xx}.
\end{aligned}
\end{equation}
Here $F$ is taken general. One recovers the values stated in Definition \ref{Thetas} if $F=\log$.  Notice that from \eqref{KP3partes_new} in terms of $Ai\left(\Theta\right)$, one gets
\begin{equation}\label{KP3partes_new_new}
\begin{aligned}
    &\left(F'''' + 6F''^{2}\right)\left(\Theta\right)\Theta^{4}_{x} + 6\left(F'' + F'^{2}\right)'\left(\Theta\right)\Theta^{2}_{x}\Theta_{xx} + 3\left(F'' + F'^{2}\right)\left(\Theta\right)\left(\Theta^{2}_{xx} + \Theta^{2}_{y}\right)  \\
    &-4\left(F'\left(\Theta\right)Ai\left(\Theta\right)\right)_{x}+ 3F'\left(\Theta\right)\left(W^{F}_{y}\left(\Theta\right) - W^{F}_{x}\left(\Theta\right)\right)=0.
\end{aligned}
\end{equation}
Having this structure in mind, let us study the phases related to the Wronskian conditions \eqref{Wronskians}. 

\begin{lem}
   Assume that $u = 2\partial^{2}_{x} \log\left(\Theta\right)$ is solution to KP with $\Theta>0$ and $Ai\left(\Theta\right) = 0$. Then $W_{y}\left(\Theta\right) = W_{x}\left(\Theta\right)$.
\end{lem}

\begin{proof}
Since $F=\log$ one has from Lemmas \ref{dem:ODE1} and \ref{dem:ODE2} that $\rho = F'' + F'^2 = 0$ and $F'''' + 6F''^{2}=0$. Using $Ai\left(\Theta\right) = 0$, from \eqref{KP3partes_new_new} and the fact that $F'\neq 0$, we get
\[
\begin{aligned}
    W^{F}_{y}\left(\Theta\right) - W^{F}_{x}\left(\Theta\right)= 0.
\end{aligned}
\]
The conclusion is obtained by recalling that $F=\log$. 
\end{proof}

Now we put our attention to the following rigidity property.

\begin{lem}\label{HWxWy}
Given any everywhere smooth functions $F$ and  $\Theta>0$, if $H\left(\Theta\right) = 0$, then $W^{F}_{y}\left(\Theta\right) - W^{F}_{x}\left(\Theta\right)=W_{y}\left(\Theta\right) - W_{x}\left(\Theta\right)=0$.  
\end{lem}

\begin{proof}
Fix smooth functions $F$ and $\Theta>0$.  By hypothesis  $H\left(\Theta\right) = 0$. Then $H\left(\Theta\right)_{y}=H\left(\Theta\right)_{xx}=0$. Consequently, $0=H\left(\Theta\right)_{y} + H\left(\Theta\right)_{xx} = \Theta_{yxx} - \Theta_{xxxx} + \Theta_{yy} - \Theta_{xxy}$. Since $\Theta_{xxy} = \Theta_{yxx}$, one obtains $\Theta_{yy} = \Theta_{xxxx}$. Finally, from \eqref{general_W},
\[
    W^{F}_{x}\left(\Theta\right) = \Theta_{xxxx} - F'\left(\Theta\right) \Theta^{2}_{xx} = \Theta_{yy} - F'\left(\Theta\right)\Theta^{2}_{y} = W^{F}_{y}\left(\Theta\right).
\]
This proves that $ W^{F}_{x}\left(\Theta\right)=W^{F}_{y}\left(\Theta\right)$. Now we prove that the previous result is independent of $F$. First of all, one has again from \eqref{general_W},
\[
W^{F}_{y}\left(\Theta\right) - W^{F}_{x}\left(\Theta\right) = \left(\Theta_{yy} - \Theta_{xxxx}\right) + F'\left(\Theta\right)\left(\Theta^{2}_{xx} - \Theta^{2}_{y}\right) .
\]
We compute,
\begin{equation*}
\begin{aligned}
 &~{} W^{F}_{y}\left(\Theta\right) - W^{F}_{x}\left(\Theta\right) \\
&~{} =   W_y\left(\Theta\right) -W_x\left(\Theta\right)+ \frac{1}{\Theta}\Theta^{2}_{y}  -\frac{1}{\Theta}\Theta^{2}_{xx}  + F'\left(\Theta\right)\left(\Theta^{2}_{xx} - \Theta^{2}_{y}\right)\\
&~{} =  W_y\left(\Theta\right) -W_x\left(\Theta\right) + \left(\frac{1}{\Theta} - F'\left(\Theta\right)\right)H\left(\Theta\right)\left(\Theta_{y} + \Theta_{xx}\right).
\end{aligned}
\end{equation*}
Since $H\left(\Theta\right)=0$, by the previous result $ W^{F}_{y}\left(\Theta\right) - W^{F}_{x}\left(\Theta\right) =0$ and $W_y\left(\Theta\right) -W_x\left(\Theta\right) =0$.
\end{proof}

\begin{rem}
The converse of Lemma \ref{HWxWy} does not hold in general. Indeed, the condition $H\left(\Theta\right) = 0$ is sufficient to cancel the third part of the equation \eqref{KP3partes_new}. However, the conditions $W^{F}_{x}\left(\Theta\right) = W^{F}_{y}\left(\Theta\right)$ are not sufficient to cancel $H\left(\Theta\right)$. See Corollary \ref{coro22} for additional details.
\end{rem}

\begin{lem}\label{lem:Wy}
 Let $\Theta>0$ be a everywhere smooth real-valued phase. Then $\Theta$ is of $y$-Wronskian type if and only if $\Theta = A\left(t,x\right)\exp\left(c\left(t,x\right)y\right)$, for arbitrary $A>0$, $c$ real-valued.
\end{lem}
\begin{proof}
The sufficient condition is clear. We prove the necessary condition. By hypothesis, $\Theta W_{y}\left(\Theta\right) = 0$, i.e. $\Theta \Theta_{yy} - \Theta^{2}_{y}= 0$. Since $\Theta>0$, this equation can be written as  $\Theta^{2} \partial_{y} \left(\ \frac{\Theta_y}{\Theta} \right) = 0$. Hence,  $\Theta_y = c\left(t,x\right) \Theta$, with $c\left(t,x\right)$ a well-defined function in $\mathbb R^{2}$. Then $\Theta \left(t,x,y\right) = A\left(t,x\right)\exp\left(c\left(t,x\right)y\right)$. Note that $A\left(t,x\right) > 0$, since $\Theta \left(t,x,y\right) > 0$.
\end{proof}

\begin{rem}
It is noticed that phases satisfying the $y$-Wronskian condition are extremely rigid. On the contrary, the $x$-Wronskian condition seems less demanding.
\end{rem}

\begin{cor}\label{wywx0}
 Let $\Theta>0$ be a smooth phase. Then $\Theta$ is of $x$-Wronskian and $y$-Wronskian type if and only if  
 \begin{equation*}
 \Theta\left(t,x,y\right) = A\left(t,x\right)\exp\left( \left(c_{1}\left(t\right) + c_{2}\left(t\right)x\right)y\right),
 \end{equation*}
with $c_1,c_2$ are time-dependent, smooth arbitrary functions, and $A>0$ is an $x$-Wronskian type function. In the case where $c_2\left(t\right)\neq 0$, one has $A\left(t,x\right) =  c_{3}\left(t\right)\exp\left(c_{4}\left(t\right)x\right)$, with $c_3\left(t\right)>0$ and $c_4\left(t\right)\in\mathbb R$ is smooth and arbitrary. 
\end{cor}

\begin{proof}
As in the proof of Lemma \ref{lem:Wy}, we only prove the necessary condition. Thanks to this last result, $\Theta \left(t,x,y\right) = A\left(t,x\right)\exp\left(c\left(t,x\right)y\right)$. Assuming now $W_{x}\left(\Theta\right) = 0$, i.e. $\Theta \Theta_{xxxx} - \Theta^{2}_{xx} = 0$, and replacing $\Theta$ in this equation, one gets the next degree 3 polynomial in the variable $y$:
\[
\begin{aligned}
    p\left(y\right) &= y^{3}\left(4A^{2}c^{2}_{x}c_{xx}\right)\\
    &+ y^{2}\left(-4A^{2}_{x}c^{2}_{x} + 4AA_{xx}c^{2}_{x} + 8AA_{x}c_{x}c_{xx} + 2A^{2}c^{2}_{xx} + 4A^{2}c_{x}c_{xxx}\right)\\
    &+ y\left(-4A_{x}A_{xx}c_{x} + 4AA_{xx}c_{xx} + 4AA_{xxx}c_{x} + 4AA_{x}c_{xxx} + A^{2}c_{xxxx}\right)\\
    &+ AA_{xxxx} - A^{2}_{xx} = 0.
\end{aligned}
\]
By the linear independence, each of the coefficients that multiplies $y^{i}$ with $i \in \left\{0,1,2,3\right\}$, must be equal to zero. Making the coefficient that multiplies $y^{3}$ equal to 0 one gets $4A^{2}c^{2}_{x}c_{xx} = 0$. Since $A>0$, $\left(c^{3}_{x}\right)_{x} = 0.$ Consequently, $c \left(t,x\right)= c_{1}\left(t\right) + c_{2}\left(t\right)x$. Now the remaining equations in powers of $y$ return the equations
\[
\begin{aligned}
    AA_{xxxx} - A^{2}_{xx} &= 0,\\
    4c_{2}\left(t\right)\left(AA_{xxx} - A_{x}A_{xx}\right) &= 0,\\
    4c^{2}_{2}\left(t\right)\left(AA_{xx} - A^{2}_{x}\right) &= 0.
\end{aligned}
\]
Assume $c_{2}\left(t\right) \neq 0$. The third equation implies the first and second ones. We are left to solve $AA_{xx} = A^{2}_{x}$, or $A^2\left( \dfrac{A_x}{A}\right)_x=0$. The solution is
\begin{equation*}
    A\left(t,x\right) = c_{3}\left(t\right)\exp\left(c_{4}\left(t\right)x\right).
\end{equation*}
Notice that $A$ is of $x$-Wronskian type, that is, $AA_{xxxx} = A^{2}_{xx}$. Replacing $A$ in the phase $\Theta$, 
\begin{equation*}
    \Theta\left(t,x,y\right) = c_{3}\left(t\right)\exp\big(c_{4}\left(t\right)x + \left(c_{1}\left(t\right) + c_{2}\left(t\right)x\right)y\big),
\end{equation*}
as desired.

\medskip

If now $c_{2}\left(t\right) = 0$, then $A$ is of $x$-Wronskian type. This ends the proof.
\end{proof}

\begin{cor}\label{coro22}
Assume $\Theta>0$ everywhere. The conditions $W_{y}\left(\Theta\right)= W_{x}\left(\Theta\right) = 0$ do not necessarily imply $H\left(\Theta\right) = 0$. 
\end{cor}
 
\begin{proof}
From Corollary \ref{wywx0}, and assuming $c_2\left(t\right)$ different from zero, we obtain that necessarily $\Theta\left(t,x,y\right) = c_{3}\left(t\right)\exp\left(c_{4}\left(t\right)x + \left(c_{1}\left(t\right) + c_{2}\left(t\right)x\right)y\right)$. Then
\begin{equation*}
    H\left(\Theta\right) = \Theta_{y} - \Theta_{xx} = \Theta \left(\left(c_1+ c_2x \right) - \left(c_2 y + c_4\right)^{2}\right).
\end{equation*}
Then $H$ is identically zero only if $c_2 = 0$ and $c_1 = c_4^{2}$. Therefore, in general $W_{x}\left(\Theta\right) = W_{y}\left(\Theta\right) = 0$ do not imply $H\left(\Theta\right) = 0$, except if $c_2 = 0$ and $c_1 = c_4^{2}$.
\end{proof}

\begin{lem}\label{4p11}
    If $W_{y}\left(\Theta\right) = W_{x}\left(\Theta\right) = Ai\left(\Theta\right) = 0$, then $\Theta = A\left(t,x\right)\exp\left(cy\right)$, with $A\left(t,x\right)$ being of Airy type and with $c \in \mathbb R$.
\end{lem}
\begin{proof}
From Corollary \ref{wywx0}, if $W_{y}\left(\Theta\right) = W_{x}\left(\Theta\right)=0$ one has 
\[
 \Theta\left(t,x,y\right) = A\left(t,x\right)\exp\left(\left(c_{1}\left(t\right) + c_{2}\left(t\right)x\right)y\right).
 \]  
If now $Ai\left(\Theta\right) = 0$,
\[
\begin{aligned}
0= &~{} -y \exp\left(y\left(c_1\left(t\right)+x c_2\left(t\right)\right)\right) \left(3 c_2\left(t\right) A_{xx}\left(t,x\right)+c_1'\left(t\right)\left(-A\left(t,x\right)\right)-x c_2'\left(t\right)A\left(t,x\right)\right) \\
&~{} -3 y^2 c_2^2\left(t\right) A_x\left(t,x\right)\exp\left(y\left(c_1\left(t\right)+xc_2\left(t\right)\right)\right) \\
&~{} -\left(A_{xxx}\left(t,x\right)-A_t\left(t,x\right)\right)\exp\left(y\left(c_1\left(t\right)+xc_2\left(t\right)\right)\right) \\
&~{} -y^3 c_2^3\left(t\right)A\left(t,x\right)\left(\exp\left(y\left(c_1\left(t\right)+xc_2\left(t\right)\right)\right)\right).
\end{aligned}
\]
If $A=0$, we are done. Assuming $c_2\left(t\right)=0$, one gets $c_1'\left(t\right)=0$ and then $c_1$ is constant, and $A$ satisfies Airy. Then  $\Theta = A\left(t,x\right)\exp\left(cy\right)$, with $A\left(t,x\right)$ being of Airy type.
\end{proof}

\begin{lem}
  Let $\Theta>0$ be a smooth phase, and $F$ smooth such that $u=2\partial_x^2 F\left(\Theta\right)$ solves KP \eqref{eq:KP}. Then  $Ai\left(\Theta\right) = 0$ and $H\left(\Theta\right) = 0 $ imply  $\mathcal T\left(\Theta\right) = 0$. 
\end{lem}
\begin{proof}
Direct from the definition of $\mathcal T\left(\Theta\right)$ in \eqref{eqn:T_type}.
\end{proof}

Classical resonant solitons are in the class $Ai\left(\Theta\right) = 0$ and $H\left(\Theta\right) = 0 $. From the previous result, their phases do satisfy $\mathcal T\left(\Theta\right) = 0$ as well. The following is a sort of reciprocal saving Corollary \ref{coro22}.

\begin{lem}\label{WsAidaF}
  Let $u=2\partial^{2}_{x}F\left(\Theta\right)$ solution of \eqref{eq:KP}. Let $F$ be a smooth profile satisfying $F\left(1\right)=0$, $F'\left(1\right)=1$, $F''\left(1\right)=-1$ and $F'''\left(1\right)=2$. Let $\Theta>0$ be a smooth phase such that $\Theta_x\neq 0$, $W_{y}^F\left(\Theta\right) = W_{x}^F\left(\Theta\right)$ and $Ai\left(\Theta\right) = 0$. Then $F=\log$ and consequently $W_{y}^F=W_y$, $W_{x}^F=W_x$.
\end{lem}
\begin{proof}
From \eqref{KP3partes_new} we have
\[
\begin{aligned}
&    \left(\rho'' \left(s\right)- 2F'\left(\Theta\right)\rho' \left(s\right) + 4F''\left(\Theta\right)\rho \left(s\right)\right)\Theta^{4}_{x} + 6\rho' \left(s\right)\Theta^{2}_{x}\Theta_{xx} + 3\rho \left(s\right)\left(\Theta^{2}_{xx} + \Theta^{2}_{y}\right)\\
  & \quad -4F''\left(\Theta\right)\Theta_{x}Ai\left(\Theta\right) - 4F'\left(\Theta\right)Ai\left(\Theta\right)_{x}\\
   &\quad  + 3F'\left(\Theta\right)\left(W^{F}_{y}\left(\Theta\right) - W^{F}_{x}\left(\Theta\right)\right) = 0.
\end{aligned}
\]
Using the hypotheses, only the first part remains:
\begin{equation*}
    \left(\rho'' \left(s\right) - 2F'\left(\Theta\right)\rho' \left(s\right) + 4F''\left(\Theta\right)\rho\left(s\right)\right)\Theta^{4}_{x} + 6\rho' \left(s\right)\Theta^{2}_{x}\Theta_{xx} + 3\rho \left(s\right)\left(\Theta^{2}_{xx} + \Theta^{2}_{y}\right) = 0.
\end{equation*}
This is an ODE for $\rho$ with variable coefficients, exactly with the form \eqref{EDO_rho_new}. Lemma \ref{dem:ODE2} ensures that $\rho = 0$. Lemma \ref{dem:ODE1} implies that $F=\log$ and by definition $W_{y}^F=W_y$, $W_{x}^F=W_x$.
\end{proof}

\section{Soliton structures}\label{Sec:6}

\subsection{Airy and Heat structures}

Recall that the Airy type condition defined in Definition \ref{Thetas} is described by the identity $Ai\left(\Theta\right)= \Theta_t -\Theta_{xxx}=0$. Unfortunately, this condition does not coincide with linear part of the KP equation. Later (see Appendix \ref{rem:KdVmKdV}) we will check that in the KdV this is not the case: being of Airy type implies that $\Theta$ satisfies the linear part of KdV.  

\medskip

Some simple solutions of $Ai\left(\Theta\right)=0$ are
\begin{align*}
\Theta_1\left(t,x,y\right) &= A_0\exp\left(k_{1}x + k_{1}^{2}y + k_{1}^{3}t\right) + B_0\exp\left(k_{2}x + k_{2}^{2}y + k_{2}^{3}t\right) + C_0, \quad 
\end{align*}
\begin{align*}
\Theta_2\left(t,x,y\right) = A\left(y\right)\exp\left(kx + k^{3}t\right),
\end{align*}
where $A_0,B_0,C_0,k_{1},k_{2},k \in \mathbb R$ and $A>0$ is any well-defined function. Associated to the profile $F=\log$, the phase  $\Theta_1$ corresponds to a line-soliton phase \eqref{linesoliton} while  $\Theta_{2}$ gives the trivial solution (Lemma \ref{FandPhase} $(iii)$). Finally, the phase $\Theta_1$ serves as an example exhibiting both Airy and Heat type.

\begin{lem}
Assume that $\Theta$ is of $\mathcal T$-type and $F'\left(\Theta\right)$ is different from zero for $\left(t,x,y\right) \in \mathbb R^{3}$. If $\Theta$ is of Heat type, then there exists $c_0\left(t,y\right)\in\mathbb R$ such that 
\begin{equation}\label{sol_Ai}
Ai\left(\Theta\right)= \frac{c_0\left(t,y\right)}{F'\left(\Theta\right)}.
\end{equation}
\end{lem}

\begin{proof}
Notice that $\Theta$ being of $\mathcal T$-type as in \eqref{eqn:T_type} is equivalent to have
\[
\begin{aligned}
0 = &~{} -4F''\left(\Theta\right)\Theta_xAi\left(\Theta\right) - 4F'\left(\Theta\right)Ai\left(\Theta\right)_{x} + 3F'\left(\Theta\right)\left(H\left(\Theta\right)_{y} + H\left(\Theta\right)_{xx}\right) \\
&~{} - 3F'\left(\Theta\right)^2H\left(\Theta\right)\left(\Theta_{y} + \Theta_{xx}\right).
\end{aligned}
\]
Equivalently,
\[
\begin{aligned}
 \left(F'\left(\Theta\right)Ai\left(\Theta\right)\right)_x = &~{} \frac{3F'\left(\Theta\right)}{4}\left(\left(H\left(\Theta\right)_y + H\left(\Theta\right)_{xx}\right) - F'\left(\Theta\right)H\left(\Theta\right)\left(\Theta_y + \Theta_{xx}\right)\right).
\end{aligned}
\]
Then \eqref{sol_Ai} follows directly from $H\left(\Theta\right) = 0$ for all $\left(t,x,y\right)$ and solving the corresponding ODE for $Ai\left(\Theta\right)$.
\end{proof}

\begin{lem}\label{lemvert}  Let $\Theta>0$ be a smooth phase satisfying $W_x\left(\Theta\right)=W_y\left(\Theta\right)=0$. Then the following conditions are satisfied:
\begin{enumerate}
\item[(i)] If $H\left(\Theta\right)=0$, then 
\[
\begin{aligned}
\Theta\left(t,x,y\right) = &~{} a_0(t)\exp\left(k_1(t) x+k_1^2(t) y\right)+a_1(t)\exp\left(-k_1(t) x+k_1^2(t) y\right),
\end{aligned}
\]
with $a_0,a_1,k_1\geq 0$.
\smallskip
\item[(ii)] If $Ai\left(\Theta\right)=0$, then $\Theta = A\left(t,x\right)\exp\left(c_{1}y\right)$ where $c_1\in \mathbb R$ and $A$ is of Airy and $W_x$ type.
\smallskip
\item[(iii)] If both $H\left(\Theta\right)=Ai\left(\Theta\right)=0$, then   \[\Theta=a_1 \exp\left(k_1 x+k_1^2y+k_1^{3}t\right)+a_2\exp\left(-k_1 x+k_1^2y-k_1^{3}t\right),\]
for some  $a_1,a_2>0$ and $k_1$ constants.
\end{enumerate}
\end{lem}

\begin{proof}
Proof of (i). Assume $H\left(\Theta\right) = 0$. From $W_x\left(\Theta\right)=W_y\left(\Theta\right)=0$, Corollary \ref{wywx0} yields 
\[
\Theta\left(t,x,y\right) = A\left(t,x\right)\exp\left(\left(c_{1}\left(t\right) + c_{2}\left(t\right)x\right)y\right)
\]
with $A$ of $W_x$ type. Now if $c_2\left(t\right)$ is different from zero, from the proof of  Corollary \ref{coro22}  we first have
\[
\Theta\left(t,x,y\right) = c_{3}\left(t\right)\exp\left(c_{4}\left(t\right)x + \left(c_{1}\left(t\right) + c_{2}\left(t\right)x\right)y\right),
\]
and from $H(\Theta)=0$ we must have $c_2(t) = 0$ and $c_1 (t)= c_4^{2}(t)$. We conclude
\[
\Theta\left(t,x,y\right) = c_{3}\left(t\right)\exp\left(c_{4}\left(t\right)x + c_4^2\left(t\right) y\right).
\]
Denoting $a_0:= c_3$, $k_1:=c_4$, we conclude this case.

\medskip

 In the case $c_2\left(t\right)=0$ we first have $\Theta\left(t,x,y\right) = A\left(t,x\right)\exp\left(c_{1}\left(t\right)y\right)$. Since $H\left(\Theta\right)=0$ necessarily $A\left(t,x\right)c_1\left(t\right) -A_{xx}\left(t,x\right)=0$. Depending on the sign of $c_1\left(t\right)$, we have
\[
A\left(t,x\right)= a_0\left(t\right) e^{\sqrt{c_1\left(t\right)} x} +a_1\left(t\right) e^{-\sqrt{c_1\left(t\right)} x}, \quad c_1\left(t\right)>0,
\] 
\[
A\left(t,x\right)= a_0\left(t\right) +a_1\left(t\right) x, \quad c_1\left(t\right)=0,
\]
or
\[
A\left(t,x\right)= a_0\left(t\right) \cos\left(\sqrt{-c_1\left(t\right)} x\right)  +a_1\left(t\right) \sin\left(\sqrt{-c_1\left(t\right)} x\right), \quad -c_1\left(t\right)>0.
\]
Since $\Theta>0$ we discard the third case, and in the second case we get $c_2\left(t\right)=0$, leading to a solution included in the first case. Note that naturally the condition $A$ is of $W_x$ type is satisfied. Denoting $k_1:= \sqrt{c_1}$ we conclude.  This proves  (i). 

\medskip

Proof of (ii). As in the previous item, from Corollary \ref{wywx0} we have
\begin{equation*}
    \Theta = A\left(t,x\right)\exp\left(\left(c_{1}\left(t\right)+c_{2}\left(t\right)x\right)y\right).
\end{equation*}
This phase will assume different values depending on the value of $c_{2}\left(t\right)$. 

\medskip

\emph{ Case $c_{2}\left(t\right)=0$}. First, we will examine the case in which $c_{2}\left(t\right)=0$. In this case from Corollary \ref{wywx0} $\Theta = A\left(t,x\right)\exp\left(c_{1}\left(t\right)y\right)$, with $W_x\left(A\right)=0$. First we compute
\begin{equation*}
\begin{aligned}
    \Theta_{t}&=A_{t}\exp\left(c_{1}y\right)+Ac'_{1}y\exp\left(c_{1}y\right),\quad \Theta_{xxx}=A_{xxx}\exp\left(c_{1}y\right).
\end{aligned}
\end{equation*}
Consequently, imposing the Airy condition,
\begin{equation*}
\begin{aligned}
    Ai\left(\Theta\right) &= A_{t}\exp\left(c_{1}y\right) + Ac'_{1}\exp\left(c_{1}y\right) - A_{xxx}\exp\left(c_{1}y\right)\\
&= \exp\left(c_{1}y\right)\left(A_{t}-A_{xxx}\right)+y\exp\left(c_{1}y\right) Ac'_{1} =0. 
\end{aligned}
\end{equation*}
The equality leads to a first-degree polynomial in the variable $y$ equals zero. Setting each coefficient equal to 0 yields the following system 
\begin{equation*}
\begin{aligned}
    A_{t}-A_{xxx}=0, \quad Ac'_{1}=0.
\end{aligned}
\end{equation*}
From the first equation, $A\left(t,x\right)$ is Airy type and from the second one, $c_{1}\left(t\right) = c_{1}$ with $c_{1}\in\mathbb R$ an arbitrary constant. This means that the phase is
\begin{equation}\label{c2=0}
    \Theta = A\left(t,x\right)\exp\left(c_{1}y\right).
\end{equation}

\medskip

\noindent
\emph{ Case $c_{2}\left(t\right)\neq 0$}.  Now, from Corollary \ref{wywx0} it is known that when $c_{2}\left(t\right)$ is different from zero the phase is
\begin{equation*}
    \Theta = c_{3}\left(t\right)\exp\left(c_{1}\left(t\right)y+c_{2}\left(t\right)xy+c_{4}\left(t\right)x\right).
\end{equation*}
One has
\begin{equation*}
\begin{aligned}
    \Theta_{t}&=c'_{3}\exp\left(c_{1}\left(t\right)y+c_{2}\left(t\right)xy+c_{4}\left(t\right)x\right) \\
    & \quad + c_{3}\left(c'_{1}y+c'_{2}xy+c'_{4}x\right)\exp\left(c_{1}\left(t\right)y+c_{2}\left(t\right)xy+c_{4}\left(t\right)x\right),\\
    \Theta_{xxx}&=c_{3}\left(c_{2}y+c_{4}\right)^{3}\exp\left(c_{1}\left(t\right)y+c_{2}\left(t\right)xy+c_{4}\left(t\right)x\right).
\end{aligned}
\end{equation*}
Imposing the Airy condition,
\begin{equation*}
\begin{aligned}
     0= &~{} Ai\left(\Theta\right)\\
     = &~{} \exp\left(c_{1}\left(t\right)y+c_{2}\left(t\right)xy+c_{4}\left(t\right)x\right)\\
    &~{} \times \left(\left(c'_{3}-c^{3}_{4}\right)+y\left(c_{3}c'_{1}-3c_{2}c^{2}_{4}\right)+x\left(c_{3}c'_{4}\right)+xy\left(c_{3}c'_{2}\right)+y^{2}\left(-3c^{2}_{2}c_{4}\right)+y^{3}\left(-c_{3}c^{2}_{2}\right)\right).
\end{aligned}
\end{equation*}
This equality leads to a third-degree polynomial with variables $x$ and $y$. Setting each coefficient equal to 0 yields the following system of six equations,
\begin{equation*}
\begin{aligned}
    c'_{3}-c^{3}_{4}&=0,\quad  c_{3}c'_{1}-3c_{2}c^{2}_{4}=0,\quad    c_{3}c'_{4}=0,\\
    c_{3}c'_{2}&=0,\quad c^{2}_{2}c_{4}=0,\quad  c_{3}c^{2}_{2}=0.
\end{aligned}
\end{equation*}
From the last equation $c_{3}=0$ implies $\Theta = 0$, then we discard this trivial case. We assume $c_3\neq 0$ and  $c_{2}\left(t\right)=0$. Then the system is reduced to
\begin{equation*}
\begin{aligned}
    c'_{3}-c^{3}_{4}&=0,\quad c_{3}c'_{1}=0,\quad c_{3}c'_{4}=0.
\end{aligned}
\end{equation*}
From the third and the second equation $c_{4}(t)=c_{4}$ and $c_{1}(t)=c_{1}$, respectively, with $c_{1}$, $c_{4}\in\mathbb R$ arbitrary constants. Then, from the first equation $c_{3}(t)=c^{3}_{4}t+c_{3}$, with $c_{3}\in\mathbb R$ an arbitrary constant. The condition $\Theta>0$ implies $c_4=0$. Thus, the phase is
\begin{equation*}
\begin{aligned}
    \Theta=c_3 \exp\left(c_{1}y+c_{4}x\right).
\end{aligned}
\end{equation*}
Using again the Airy condition, one has $c_{4}=0$ and
\begin{equation*}
    \Theta = c_{5}\exp\left(c_{1}y\right),
\end{equation*} 
which is a particular case of \eqref{c2=0}, and a trivial solution to \eqref{eq:KP}. This proves (ii).

\medskip

Proof of (iii).  From (i) and (ii), we get
\[
\Theta\left(t,x,y\right) = a_1\left(t\right)\exp\left(k_1 x+k_1^2 y\right)+a_2\left(t\right)\exp\left(-k_1 x+k_1^2 y\right).
\]
Additionally, to satisfy the Airy condition, 
\[
\left(a_1' -a_1k_1^3\right)\exp\left(k_1 x+k_1^2 y\right) +\left(a_2'  + a_2 k_1^3\right)\exp\left(-k_1 x+k_1^2 y\right)=0.
\]
We obtain $a_1\left(t\right)= a_1 e^{k_1^3 t}$, $a_2\left(t\right)= a_2 e^{-k_1^3 t}$, where $a_1$ and $a_2$ are constants. Therefore,
 \[
 A\left(t,x\right)=a_1\exp\left(k_1 x +k_1^2 y + k_1^3 t\right)+a_2\exp\left(-k_1 x +k_1^2 y -k_1^3 t\right).
 \]
 This proves (iii).
\end{proof}

Finally, we provide a quick method to construct an Airy-Heat phase $\Theta$. The relevance here is that the Airy-Heat type is a property still satisfied by infinitely many functions, specially in the KP case. 

\begin{lem}
Let $\Theta_0=\Theta_0\left(t,x\right)$ be any solution of Airy $Ai\left(\Theta_0\right)=0$ such that there are $C_1,C_2>0$ under which  $|\Theta_0\left(t\right)|\le C_1 \exp\left(C_2|x|\right)$. Then $\Theta\left(t,x,y\right):=\exp\left(y\partial_x^2\right)\Theta_0$ solves $H\left(\Theta\right)=0$ for $y\geq 0$.
\end{lem}

\begin{proof}
This result is clear from the formula
\[
{\color{black}
\Theta\left(t,x,y\right)=\exp\left(y\partial_x^2\right)\Theta_0\left(t,x\right) = \frac1{\left(4\pi y\right)^{1/2}} \int_{\mathbb R} \exp \left( -\frac{\left(x-s\right)^2}{4 y}\right)\Theta_0\left(t,s\right)\mathrm{d}s
}\]
and the growth of $\Theta_0$.
\end{proof}

\begin{rem}
As already mentioned in the introduction (see Remark \ref{cagazo}), a similar result for the case of the Airy equation $Ai\left(\Theta\right)=0$ with initial condition exponentially growing in $x$ is far from being obvious. 
\end{rem}

\subsection{Soliton structures}

Recall from \eqref{linesoliton} that a simple KP line soliton is obtained by the structure
\begin{equation}\label{theta_linesoliton}
   \Theta\left(t,x,y\right)= a_{1}\exp\left(\theta_1\right) + a_{2}\exp\left(\theta_{2}\right),
\end{equation}
where $a_1,a_2>0$, and $\theta_j: =k_j x + k_j^2 y + k_j^3 t$, $k_1,k_2\in\mathbb R$. We shall say that $\Theta$ represents a line-soliton if $\Theta$ has the previous form. Notice that unless $F=\log$, one does not have that $u=2\partial_x^2 F\left( \Theta\right)$ is the standard KP line-soliton \eqref{soliton family}. 
 
\begin{lem}\label{lem:linesoliton}
If $\Theta>0$ smooth represents a line-soliton, then $Ai\left(\Theta\right) =H\left(\Theta\right) = 0$. Moreover, the case $k_{1} = \pm k_{2}$ is the unique instance in which $W_{x}\left(\Theta\right) = W_{y}\left(\Theta\right) = 0$.
\end{lem}

\begin{rem}
Notice that the condition $k_{1} = -k_{2}$ corresponds to the case of the KdV soliton (vertical line-soliton) and $k_1=k_2$ the trivial solution.
\end{rem}

\begin{proof}[Proof of Lemma \ref{lem:linesoliton}]
The proof comes from \eqref{theta_linesoliton}. Indeed
\begin{align*}
    Ai\left(\Theta\right) = &~{}  \Theta_{t} - \Theta_{xxx}\\
    =&~{} k^{3}_{1}a_{1}\exp\left(\theta_1\right) + k^{3}_{2}a_{2}\exp\left(\theta_2\right) - \left(k^{3}_{1}a_{1}\exp\left(\theta_1\right) + k^{3}_{2}a_{2}\exp\left(\theta_2\right)\right) = 0.
\end{align*}
Additionally, 
\begin{align*}
    H\left(\Theta\right) =&~{} \Theta_{y} - \Theta_{xx}\\
    =&~{}  k^{2}_{1}a_{1}\exp\left(\theta_1\right) + k^{2}_{2}a_{2}\exp\left(\theta_2\right) - \left(k^{2}_{1}a_{1}\exp\left(\theta_1\right) + k^{2}_{2}a_{2}\exp\left(\theta_2\right)\right) = 0.
\end{align*}
Finally,
\begin{align*}
    \Theta W_{y}\left(\Theta\right) &= \Theta \Theta_{yy} - \Theta^{2}_{y}\\
    &= \left(a_{1}\exp\left(\theta_1\right) + a_{2}\exp\left(\theta_2\right)\right)\left(k^{4}_{1}a_{1}\exp\left(\theta_1\right) + k^{4}_{2}a_{2}\exp\left(\theta_2\right)\right)\\
    & \quad  - \left(k^{2}_{1}a_{1}\exp\left(\theta_1\right) + k^{2}_{2}a_{2}\exp\left(\theta_2\right)\right)^{2}\\
    &=k^{4}_{1}a^{2}_{1}\exp\left(2\theta_1\right) + k^{4}_{2}a_{1}a_{2}\exp\left(\theta_{1} + \theta_{2}\right) \\
    & \quad + k^{4}_{1}a_{1}a_{2}\exp\left(\theta_1+\theta_2\right) + k^{4}_{2}a^{2}_{2}\exp\left(2\theta_2\right) \\
    &\quad- k^{4}_{1}a^{2}_{1}\exp\left(2\theta_1\right)- 2k^{2}_{1}k^{2}_{2}a_{1}a_{2}\exp\left(\theta_1+\theta_2\right) - k^{4}_{2}a^{2}_{2}\exp\left(2\theta_{2}\right)\\
    &=\left(k^{4}_{1} + k^{4}_{2} - 2k^{2}_{1}k^{2}_{2}\right)a_{1}a_{2}\exp\left(\theta_1+\theta_2\right).
\end{align*}
We conclude that
\begin{equation}\label{thetaW_y}    
  \Theta  W_{y}\left(\Theta\right)   =\left(k^{2}_{1} - k^{2}_{2}\right)^{2}a_{1}a_{2}\exp\left(\theta_1+\theta_2\right).
\end{equation}
Notice that $W_{y}\left(\Theta\right) = 0$ if and only if $k_{1} = \pm k_{2}$. Since $\Theta$ is a Heat type phase, by Lemma \ref{HWxWy}, 
\begin{equation*}
    W_{x}\left(\Theta\right) = W_{y}\left(\Theta\right) = 0.
\end{equation*}
The proof is complete.
\end{proof}

The previous result can be extended to smooth phases $\Theta$ of the form, 
\begin{equation}\label{theta_multi_linesoliton}
    \Theta = \sum^{M}_{j=1}a_{j}\exp\left(\theta_{j}\right), \quad a_j>0, \quad \theta_j:= k_j x+ k_j^2 y +k_j^3 t, \quad k_j\in\mathbb R.
\end{equation}
Here, $k_1<k_2<\cdots<k_M$. In Kodama \cite{Kodama2017} this phase represents a multi-soliton structure graphically represented by $M-1$ legs on the region $y$ negative and $1$ leg in the positive part of the $y$-axis. Recall $\mathcal W_n$ introduced in \eqref{Wn}.

\begin{lem}\label{WM}
If $\Theta\in \mathcal W_M$, then
\[
\Theta W_y\left(\Theta\right) \in \mathcal W_{\frac12M\left(M-1\right)}.
\]
In particular, $\Theta\in\mathcal W_3$ implies $\Theta W_y\left(\Theta\right) \in \mathcal W_3$.
\end{lem}

\begin{proof}
Assume that $\Theta\in\mathcal W_M$, that is for $0\leq k_1\left(t,x\right)<k_2\left(t,x\right)<\ldots < {\color{black} k_M}\left(t,x\right)$,
\[
\Theta = \sum_{j=1}^M a_j\exp\left(\theta_j\right), \quad a_j\left(t,x\right)>0, \quad  \theta_j\left(t,x,y\right)= k_j\left(t,x\right)y. 
\]
{\color{black} Note now that for each real numbers $A_i, B_j$ with $i,j\in \{1,\cdots,M\}$ we have \[\sum^{M}_{n = 1}A_{n}\sum^{M}_{n = 1}B_{n} = \sum^{2M}_{n = 2}C_{n}\hbox{, where } C_{n} =  \sum_{\substack{1 \le i, j \le M \\ i + j = n}} A_i B_j. \]
Thus
\begin{align*}
     \Theta W_{y}\left(\Theta\right) &= \Theta \Theta_{yy} - \Theta^{2}_{y}\\
    &= \left(\sum^{M}_{i=1}a_{i}\exp\left(\theta_{i}\right)\right)\left(\sum^{M}_{i=1} k^{4}_{i}a_{i}\exp\left(\theta_{i}\right)\right) - \left(\sum^{M}_{i=1} k^{2}_{i}a_{i}\exp\left(\theta_{i}\right)\right)^{2}\\
    &= \sum^{2M}_{n = 2} \sum_{\substack{1 \le i, j \le M \\ i + j = n}}k^{4}_{j}a_{i}a_{j}\exp\left(\theta_{i} + \theta_{j}\right)\\
    & \quad  - \sum^{2M}_{n = 2} \sum_{\substack{1 \le i, j \le M \\ i + j = n}} k^{2}_{i}k^{2}_{j}a_{i}a_{j}\exp\left(\theta_{i} + \theta_{j}\right)\\
    &= \sum^{2M}_{n = 2} \sum_{\substack{1 \le i, j \le M \\ i + j = n}} a_{i}a_{j}\left(k^{4}_{j} - k^{2}_{i}k^{2}_{j}\right)\exp\left(\theta_{i} + \theta_{j}\right).
\end{align*}

We observe that every diagonal contribution, where the two indices coincide, the factor $k_i^4-k_i^2 k_i^2$ is identically zero. For the remaining off-diagonal terms, using $\theta_{i} + \theta_{j}=\theta_{j} + \theta_{i}$, we can pair the summands that come from $(i, j)$ and $(j, i)$. Adding those two contributions, associated to $\exp\left(\theta_{i} + \theta_{j}\right)$, replaces the original coefficient $k_j^4-k_i^2 k_j^2$ by the symmetric combination $\left(k_i^2-k_j^2\right)^2$. After this last expresión  reduces to a single sum obtaining

$$
\Theta W_{y}\left(\Theta\right)=\sum_{1 \leq i<j \leq M} a_i a_j\left(k_i^2-k_j^2\right)^2 e^{\theta_i+\theta_j}.
$$

Thus under the hypotheses of the lemma,  are $\frac12 M(M-1)$ linearly independent terms in the last expression, proving the required inclusion in $\mathcal W_{\frac12M\left(M-1\right)}$. 

}
\end{proof}

\begin{cor}\label{resonantes}
If $\Theta>0$ as in \eqref{theta_multi_linesoliton} generates an $M$ resonant soliton, then $Ai\left(\Theta\right) = H\left(\Theta\right) = 0$ but  $W_{x}\left(\Theta\right) = W_{y}\left(\Theta\right)$ are different from zero unless $M=1$ or $M=2$.
\end{cor}

\begin{proof}
The proof of $Ai\left(\Theta\right) = H\left(\Theta\right) = 0$  is direct. Indeed, computing the corresponding derivatives and replacing its values in $H\left(\Theta\right)$,
\begin{equation*}
\begin{aligned}
    H\left(\Theta\right) = &~{} \Theta_{y} - \Theta_{xx} \\
    = &~{}   \sum^{M}_{i=1}k^{2}_{i}a_{1}\exp\left( k_{i}x + k^{2}_{i}y + k^{3}_{i}t \right) - \sum^{M}_{i=1}k^{2}_{i}a_{1}\exp\left( k_{i}x + k^{2}_{i}y + k^{3}_{i}t \right) = 0.
\end{aligned}
\end{equation*}
Similarly, 
\begin{equation*}
\begin{aligned}
    Ai\left(\Theta\right) = &~{} \Theta_{t} - \Theta_{xxx} \\
    =&~{}  \sum^{M}_{i=1}k^{3}_{i}a_{i}\exp\left( k_{i}x + k^{2}_{i}y + k^{3}_{i}t \right) - \sum^{M}_{i=1}k^{3}_{i}a_{i}\exp\left( k_{i}x + k^{2}_{i}y + k^{3}_{i}t \right) = 0.
\end{aligned}
\end{equation*}
Therefore, if $\Theta\left(t,x,y\right) = \sum^{M}_{i=1}a_{i}\exp\left( k_{i}x + k^{2}_{i}y + k^{3}_{i}t \right)$ then $H\left(\Theta\right) = Ai\left(\Theta\right) = 0$.

\medskip

Now, from the proof of Lemma \ref{WM}, one has $\Theta W_{y}\left(\Theta\right) = 0$ if and only if,
\begin{equation*}
    k^{2}_{i} = k^{2}_{j},
\end{equation*}
for all $i,j \in \left\{1,\ldots,M\right\}$. If $M\ge 3$  at least one term of the sum is different to zero, since $k_1<k_2<\ldots <k_M$ imply $W_{y}$ different from zero.  In the cases $M=2$ or $M=1$, from Lemma \ref{lemvert} (iii) one has that $W_{y} = 0$ is equivalent to a phase associated  KdV soliton or a trivial solution, respectively.
\end{proof}

\subsection{2-solitons}

For the following result, recall the 2-soliton phase introduced in \eqref{2soliton}.
\begin{lem}\label{2solitonphase}
 Let $\Theta>0$ be a 2-soliton of scaling parameters $k_1<k_2<k_3<k_4$, and where each soliton correspond to a line-soliton. Then the following are satisfied:
 
\begin{enumerate}
 \item[$(i)$] In general $Ai\left(\Theta\right)$, $H\left(\Theta\right)$, $W_{y}\left(\Theta\right)$ and $W_{x}\left(\Theta\right)$ are different from zero.
 \medskip
 \item[$(ii)$] In general, $Ai\left(\Theta\right)$, $H\left(\Theta\right) \in \mathcal W_4$ and $\Theta W_{y}\left(\Theta\right), \Theta W_{x}\left(\Theta\right) \in \mathcal W_5$.
 \medskip
 \item[$(iii)$] If now $k_{1} = -k_{2}$ and $k_{3} = -k_{4}$ (that is, the case of 2 vertical line-solitons), then $W_{y}\left(\Theta\right) = 0$.
\end{enumerate}
\end{lem}

\begin{proof}
After computing (see Appendix \ref{A}), 
\begin{equation}\label{Aicomputation}
\begin{aligned}
    Ai\left(\Theta\right) =&~{} \Theta_{t} - \Theta_{xxx}\\
    =& ~{}\left(k_{3} - k_{1}\right)\left(k^{3}_{3} + k^{3}_{1}\right)\exp\left(\theta_{1} + \theta_{3}\right) + \left(k_{4} - k_{1}\right)\left(k^{3}_{4} + k^{3}_{1}\right)\exp\left(\theta_{1} + \theta_{4}\right) \\
    &~{}+ \left(k_{3} - k_{2}\right)\left(k^{3}_{3} + k^{3}_{2}\right)\exp\left(\theta_{2} + \theta_{3}\right) + \left(k_{4} - k_{2}\right)\left(k^{3}_{4} + k^{3}_{2}\right)\exp\left(\theta_{2} + \theta_{4}\right) \\
    &~{}-\left(k_{3} - k_{1}\right)\left(k_{3} + k_{1}\right)^{3}\exp\left(\theta_{1} + \theta_{3}\right) - \left(k_{4} - k_{1}\right)\left(k_{4} + k_{1}\right)^{3}\exp\left(\theta_{1} + \theta_{4}\right)\\
    &~{}- \left(k_{3} - k_{2}\right)\left(k_{3} + k_{2}\right)^{3}\exp\left(\theta_{2} + \theta_{3}\right) - \left(k_{4} - k_{2}\right)\left(k_{4} + k_{2}\right)^{3}\exp\left(\theta_{2} + \theta_{4}\right).
\end{aligned}
\end{equation}
notice that $Ai\left(\Theta\right)=0$  if and only if the associated exponentials coefficients are zero. Thus, for each $\left(j,i\right)\in\left\{\left(1,3\right),\left(1,4\right),\left(2,3\right),\left(2,4\right)\right\}$
\begin{equation*}
    \begin{aligned}
     &   \left(k_{i}-k_{j}\right)\left(k^{3}_{i}+k^{3}_{j}\right) - \left(k_{i}-k_{j}\right)\left(k_{i}+k_{j}\right)^{3}\\
       &\quad = k^{4}_{i} + k_{i}k^{3}_{j} - k^{3}_{i}k_{j} - k^{4}_{j} - \left(k_{i}-k_{j}\right)\left(k^{3}_{i}+3k^{2}_{i}k_{j}+3k_{i}k^{2}_{j}+k^{3}_{j}\right)\\
        &\quad =k^{4}_{i} + k_{i}k^{3}_{j} - k^{3}_{i}k_{j} - k^{4}_{i} - 3k^{3}_{i}k_{j} - 3 k^{2}_{i}k^{2}_{j} - k_{i}k^{3}_{j} + k^{3}_{i}k_{j} + 3k^{2}_{i}k^{2}_{j} + 3k_{i}k^{3}_{j} + k^{4}_{j}\\
       & \quad = -3k^{3}_{i}k_{j} + 3k_{i}k^{3}_{j} \\ 
       & \quad =3k_{i}k_{j}\left(k^{2}_{j}-k^{2}_{i}\right).
    \end{aligned}
\end{equation*}
Denoting $E_{ij}=\left(k_{j}-k_{i}\right)\exp\left(\theta_{i}+\theta_{j}\right)$, then $Ai\left(\Theta\right)$ can be rewritten as 
\begin{equation*}
\begin{aligned}
    Ai\left(\Theta\right) =  &~{} -3\Big( k_{1}k_{3}\left(k_{1}+k_{3}\right)E_{13} + k_{1}k_{4}\left(k_{1}+k_{4}\right)E_{14} \\
    &~{} \qquad + k_{2}k_{3}\left(k_{2}+k_{3}\right)E_{23} + k_{2}k_{4}\left(k_{2}+k_{4}\right)E_{24}\Big).
\end{aligned}
\end{equation*}
This proves that $Ai\left(\Theta\right) \in \mathcal W_4$.\footnote{
Assume $ k_1<k_2<k_3<k_4$. Let us study the condition $Ai\left(\Theta\right)=0$. If $k_1=0$, then $k_2=-k_3=-k_4$, which is impossible. A similar argument holds if now $k_2 =0$, or $k_3=0$, or $k_4=0$. Therefore, each $k_i$ must be nonzero. However, in this case $k_1=-k_3 = -k_4$, also impossible. Thus, $Ai\left(\Theta\right)\neq 0$ for a nondegenerate 2-soliton.}
Replacing the values of the derivatives in $H\left(\Theta\right)$,
\begin{align*}
    H\left(\Theta\right) &= \Theta_{y} - \Theta_{xx}\\
    &=\left(k_{3} - k_{1}\right)\left(k^{2}_{3} + k^{2}_{1}\right)\exp\left(\theta_{1} + \theta_{3}\right) + \left(k_{4} - k_{1}\right)\left(k^{2}_{4} + k^{2}_{1}\right)\exp\left(\theta_{1} + \theta_{4}\right) \\
    &\quad + \left(k_{3} - k_{2}\right)\left(k^{2}_{3} + k^{2}_{2}\right)\exp\left(\theta_{2} + \theta_{3}\right) + \left(k_{4} - k_{2}\right)\left(k^{2}_{4} + k^{2}_{2}\right)\exp\left(\theta_{2} + \theta_{4}\right) \\
    & \quad - \left(k_{3} - k_{1}\right)\left(k_{3} + k_{1}\right)^{2}\exp\left(\theta_{1} + \theta_{3}\right) - \left(k_{4} - k_{1}\right)\left(k_{4} + k_{1}\right)^{2}\exp\left(\theta_{1} + \theta_{4}\right)\\
    & \quad - \left(k_{3} - k_{2}\right)\left(k_{3} + k_{2}\right)^{2}\exp\left(\theta_{2} + \theta_{3}\right) - \left(k_{4} - k_{2}\right)\left(k_{4} + k_{2}\right)^{2}\exp\left(\theta_{2} + \theta_{4}\right).
\end{align*}
Repeating the procedure of the Airy condition, $H\left(\Theta\right)$ can be rewritten as
\begin{equation*}
    H\left(\Theta\right) = -2\left(k_{1}k_{3}E_{13} + k_{1}k_{4}E_{14} + k_{2}k_{3}E_{23} + k_{2}k_{4}E_{24}\right).
\end{equation*}
The condition $k_1<k_2<k_3<k_4$ naturally forbids $ H\left(\Theta\right) =0.$

\medskip

Now, replacing the values of the derivatives in $W_{y}\left(\Theta\right)$, it follows
\begin{equation*}
\begin{aligned}
  &  \Theta W_{y}\left(\Theta\right) \\
    &~{}=  \Theta \Theta_{yy} - \Theta^{2}_{y}\\
    &~{} = \Big(\left(k_{3} - k_{1}\right)\exp\left(\theta_{1} + \theta_{3}\right) + \left(k_{4} - k_{1}\right)\exp\left(\theta_{1} + \theta_{4}\right) \\
    &\qquad + \left(k_{3} - k_{2}\right)\exp\left(\theta_{2} + \theta_{3}\right) + \left(k_{4} - k_{2}\right)\exp\left(\theta_{2} + \theta_{4}\right)\Big)\\
    &\qquad \cdot \Big(\left(k_{3} - k_{1}\right)\left(k^{2}_{3} + k^{2}_{1}\right)^{2}\exp\left(\theta_{1} + \theta_{3}\right) \\
    &\qquad \quad + \left(k_{4} - k_{1}\right)\left(k^{2}_{4} + k^{2}_{1}\right)^{2}\exp\left(\theta_{1} + \theta_{4}\right) + \left(k_{3} - k_{2}\right)\left(k^{2}_{3} + k^{2}_{2}\right)^{2}\exp\left(\theta_{2} + \theta_{3}\right)\\
    &\qquad \quad + \left(k_{4} - k_{2}\right)\left(k^{2}_{4} + k^{2}_{2}\right)^{2}\exp\left(\theta_{2} + \theta_{4}\right)\Big) \\
    &\qquad - \Big(\left(k_{3} - k_{1}\right)\left(k^{2}_{3} + k^{2}_{1}\right)\exp\left(\theta_{1} + \theta_{3}\right) + \left(k_{4} - k_{1}\right)\left(k^{2}_{4} + k^{2}_{1}\right)\exp\left(\theta_{1} + \theta_{4}\right)\\
    &\qquad \quad + \left(k_{3} - k_{2}\right)\left(k^{2}_{3} + k^{2}_{2}\right)\exp\left(\theta_{2} + \theta_{3}\right) + \left(k_{4} - k_{2}\right)\left(k^{2}_{4} + k^{2}_{2}\right)\exp\left(\theta_{2} + \theta_{4}\right)\Big)^{2}.
\end{aligned}
\end{equation*}
Developing and rearranging,  one has
\begin{equation}\label{Wy_2soliton}
\begin{aligned} 
    & \Theta W_{y}\left(\Theta\right)\\
    & = \left(k^{2}_{1}-k^{2}_{2}\right)^{2}\Big(\exp\left(\theta_{1}+\theta_{2}+2\theta_{4}\right)+\exp\left(\theta_{1}+\theta_{2}+2\theta_{3}\right)\Big) \\
& \quad + \left(k^{2}_{3}-k^{2}_{4}\right)^{2}\Big( \exp\left(2\theta_{1}+\theta_{3}+\theta_{4}\right)+\exp\left(2\theta_{2}+\theta_{3}+\theta_{4}\right)\Big)\\
&\quad +  k_{1234} \exp\left(\theta_{1}+\theta_{2}+\theta_{3}+\theta_{4}\right),
\end{aligned}
\end{equation}
where
\[
\begin{aligned}
k_{1234}:=  &~{}\left(k_{3}-k_{1}\right)\left(k_{4}-k_{2}\right)\left(\left(k^{2}_{1}+k^{2}_{3}\right)-\left(k^{2}_{2}+k^{2}_{4}\right)\right)^{2} \\
&~{} + \left(k_{3}-k_{2}\right)\left(k_{4}-k_{1}\right)\left(\left(k^{2}_{1}+k^{2}_{4}\right)-\left(k^{2}_{2}+k^{2}_{3}\right)\right)^{2}.
\end{aligned}
\]
A simple observation reveals that $\Theta W_{y}\left(\Theta\right) =0$ if $k_{1234}=0$ and $k_1=\pm k_2$, $k_3=\pm k_4$. Since $k_1<k_2<k_3<k_4$, one necessarily has $k_1=-k_2$ and $k_3=-k_4$. In this case, we also have $k_{1234}=0$, making this assumption unnecessary. Except by this particular case, one naturally concludes from \eqref{Wy_2soliton} that $\Theta W_y\left(\Theta\right) \in \mathcal W_5.$
Therefore, $W_{y}$ can be zero in the case of a 2-soliton when $k_{1}=-k_{2}$ and $k_{3}=-k_{4}$. This can be true when all constants have different values. 

\medskip

As mentioned earlier, $H\left(\Theta\right)$ is not equal to zero and therefore it is necessary to verify both conditions, $\Theta W_{y}\left(\Theta\right) = 0$ and $\Theta W_{x}\left(\Theta\right) = 0$, separately.  Then, 
\begin{equation*}
\begin{aligned}
   & \Theta W_{x}\left(\Theta\right) \\
   &= \Theta \Theta_{xxxx} - \Theta^{2}_{xx}\\
    &= \Big(\left(k_{3} - k_{1}\right)\exp\left(\theta_{1} + \theta_{3}\right) + \left(k_{4} - k_{1}\right)\exp\left(\theta_{1} + \theta_{4}\right) + \left(k_{3} - k_{2}\right)\exp\left(\theta_{2} + \theta_{3}\right) \\
    &\qquad + \left(k_{4} - k_{2}\right)\exp\left(\theta_{2} + \theta_{4}\right)\Big) \cdot \Big(\left(k_{3} - k_{1}\right)\left(k_{3} + k_{1}\right)^{4}\exp\left(\theta_{1} + \theta_{3}\right) \\
    &\qquad + \left(k_{4} - k_{1}\right)\left(k_{4} + k_{1}\right)^{4}\exp\left(\theta_{1} + \theta_{4}\right) + \left(k_{3} - k_{2}\right)\left(k_{3} + k_{2}\right)^{4}\exp\left(\theta_{2} + \theta_{3}\right) \\
    &\qquad + \left(k_{4} - k_{2}\right)\left(k_{4} + k_{2}\right)^{4}\exp\left(\theta_{2} + \theta_{4}\right)\Big) -\Big(\left(k_{3}-k_{1}\right)\left(k_{3}+k_{1}\right)^{2}\exp\left(\theta_{1}+\theta_{3}\right)\\
    &\qquad +\left(k_{4}-k_{1}\right)\left(k_{4}+k_{1}\right)^{2}\exp\left(\theta_{1}+\theta_{4}\right) +\left(k_{3}-k_{2}\right)\left(k_{3}+k_{2}\right)^{2}\exp\left(\theta_{2}+\theta_{3}\right)\\
    &\qquad +\left(k_{4}-k_{2}\right)\left(k_{4}+k_{2}\right)^{2}\exp\left(\theta_{2}+\theta_{4}\right)\Big)^{2}.
\end{aligned}
\end{equation*}
Developing and grouping terms, we arrive at
\begin{equation}\label{Wx_2soliton}
    \begin{aligned}
& \Theta   W_{x}\left(\Theta\right) \\
&~{} = \Big(\left(k_{1}+k_{4}\right)^{2}-\left(k_{1}+k_{3}\right)^{2}\Big)^{2}E_{13}E_{14} + \Big(\left(k_{2}+k_{3}\right)^{2}-\left(k_{1}+k_{3}\right)^{2}\Big)^{2}E_{13}E_{23} \\
&\quad + \Bigg(  \Big(\left(k_{2}+k_{4}\right)^{2}-\left(k_{1}+k_{3}\right)^{2}\Big)^{2}+ \Big(\left(k_{2}+k_{3}\right)^{2}-\left(k_{1}+k_{4}\right)^{2}\Big)^{2} \Bigg) E_{14}E_{23} \\
&\quad + \Big(\left(k_{1}+k_{4}\right)^{2}-\left(k_{2}+k_{4}\right)^{2}\Big)^{2}E_{14}E_{24} + \Big(\left(k_{2}+k_{3}\right)^{2}-\left(k_{2}+k_{4}\right)^{2}\Big)^{2}E_{23}E_{24}.
    \end{aligned}
\end{equation} 
Clearly from \eqref{Wx_2soliton} one concludes that in general $\Theta W_x\left(\Theta\right) \in\mathcal W_5$. In order to get zero value, one should have
\[
\abs{k_{1}+k_{4}} =\abs{k_{1}+k_{3}} =\abs{k_{2}+k_{4}} =\abs{k_{2}+k_{3}} . 
\]
Under the hypothesis $k_{1}<k_{2}<k_{3}<k_{4}$, this is never satisfied. Indeed, 
\[
k_1+k_3 <k_2+k_3<k_2+k_4, \quad k_1+k_3<k_1+k_4;
\]
contradicting the equality of absolute values.  Then $\Theta W_{x}\left(\Theta\right)$ is always different from zero to a phase of a 2-soliton.
\end{proof}

\begin{cor}\label{coro2soliton}
In general, given the profile $F=\log$, any 2-soliton as in \eqref{2soliton} only satisfies $\mathcal T=0$.
\end{cor}

\begin{proof}
Direct from \eqref{casicasi}, Lemma \ref{dem:ODE1} and Lemma \ref{2solitonphase} (i).
\end{proof}

\section{Proof of Main Results}\label{Sec:7}

\subsection{Proof of Theorem \ref{MT1}}

Assume $\Theta>0$ is a smooth KdV line-soliton phase with a profile $F= \log$. Its phase is given by
\begin{equation}\label{phase_basica}
\begin{aligned}
    \Theta \left(t,x,y\right) &= \exp\left(kx + k^{2}y + k^{3}t\right) + \exp\left(-kx + k^{2}y - k^{3}t\right)\\
    &= \exp\left(\theta_{1}\right) + \exp\left(\theta_{2}\right),
\end{aligned}
\end{equation}
with $\theta_{1} = kx + k^{2}y + k^{3}t$ and $\theta_{2} = -kx + k^{2}y - k^{3}t$. It will be shown that the KdV phase satisfies the following  conditions
\begin{equation*}
    \begin{aligned}
    H\left(\Theta\right)   =  Ai\left(\Theta\right) =  W_{x}\left(\Theta\right) =  W_{y}\left(\Theta\right) = 0.
    \end{aligned}
\end{equation*}
Computing derivatives, we obtain
\[
\begin{aligned}
    \Theta_{x} &= k\exp\left(\theta_{1}\right) - k\exp\left(\theta_{2}\right); \quad \Theta_{xx} = k^{2}\exp\left(\theta_{1}\right) + k^{2}\exp\left(\theta_{2}\right); \\
     \Theta_{xxx} & = k^{3}\exp\left(\theta_{1}\right) - k^{3}\exp\left(\theta_{2}\right); \quad \Theta_{xxxx} = k^{4}\exp\left(\theta_{1}\right) + k^{4}\exp\left(\theta_{2}\right);\\
    \Theta_{y} &= k^{2}\exp\left(\theta_{1}\right) + k^{2}\exp\left(\theta_{2}\right); \quad \Theta_{yy} = k^{4}\exp\left(\theta_{1}\right) + k^{4}\exp\left(\theta_{2}\right); \\
    \Theta_{t} &= k^{3}_{1}\exp\left(\theta_{1}\right) - k^{3}_{1}\exp\left(\theta_{2}\right).
\end{aligned}
\]
Replacing the values of $\Theta_{y}$ and $\Theta_{xx}$ in $H\left(\Theta\right)$ by corresponding expressions above,  one has
\[
\begin{aligned}
    H\left(\Theta\right) &= \Theta_{y} - \Theta_{xx}\\
    &=\left(k^{2}\exp\left(\theta_{1}\right) + k^{2}\exp\left(\theta_{2}\right)\right) - \left(k^{2}\exp\left(\theta_{1}\right) + k^{2}\exp\left(\theta_{2}\right)\right)\\
    &= \left(k^{2}\exp\left(\theta_{1}\right) -  k^{2}\exp\left(\theta_{1}\right)\right) + \left(k^{2}\exp\left(\theta_{2}\right)  -  k^{2}\exp\left(\theta_{2}\right)\right)=0.
\end{aligned}
\]
Hence $H\left(\Theta\right) = 0$, then  $\Theta$ is of Heat type. Since $H\left(\Theta\right) = 0$  from Lemma \ref{HWxWy}, $W_{x}\left(\Theta\right) = W_{y}\left(\Theta\right)$.
Now, replacing the values of the derivatives of $\Theta$ in $W_{x}\left(\Theta\right)$,
\begin{align*}
    W_{x}\left(\Theta\right) &= \Theta_{xxxx} - \frac{\Theta^{2}_{xx}}{\Theta}\\
    &= \left(k^{4}\exp\left(\theta_{1}\right) + k^{4}\exp\left(\theta_{2}\right)\right) - \frac{\left( k^{2}\exp\left(\theta_{1}\right) + k^{2}\exp\left(\theta_{2}\right)\right)^{2}}{\left(\exp\left(\theta_{1}\right) + \exp\left(\theta_{2}\right)\right)}\\
    &= k^{4}\frac{\left(\exp\left(\theta_{1}\right) + \exp\left(\theta_{2}\right)\right)^{2}}{\left(\exp\left(\theta_{1}\right) + \exp\left(\theta_{2}\right)\right)} - k^{4}\frac{\left(\exp\left(\theta_{1}\right) + \exp\left(\theta_{2}\right)\right)^{2}}{\left(\exp\left(\theta_{1}\right) + \exp\left(\theta_{2}\right)\right)}= 0.
\end{align*}
Since $H\left(\Theta\right) = 0$, this also means that $W_{y}\left(\Theta\right) = 0$.
\medskip

Finally, replacing the derivatives of $\Theta$ in $Ai\left(\Theta\right)$
\begin{align*}
    \Theta_{t} - \Theta_{xxx} &= \left(k^{3}\exp\left(\theta_{1}\right) - k^{3}\exp\left(\theta_{2}\right)\right) - \left(k^{3}\exp\left(\theta_{1}\right) - k^{3}\exp\left(\theta_{2}\right)\right) = 0.
\end{align*}
Then, if $u = 2\partial^{2}_{x}\log{\Theta}$ is a KdV line-soliton solution of KP (i.e. $\Theta = \exp\left(kx + k^{2}y + k^{3}t\right) + \exp\left(-kx + k^{2}y - k^{3}t\right)$) then $H\left(\Theta\right) = Ai\left(\Theta\right) = W_{x}\left(\Theta\right) = W_{y}\left(\Theta\right) = 0$.
\medskip

Conversely, it will now be demonstrated that if a phase $\Theta$ satisfies $H\left(\Theta\right) = W_{x}^F\left(\Theta\right) = W_{y}^F\left(\Theta\right) = Ai\left(\Theta\right) = 0$, then the corresponding solution $u=2\partial_x^2 F\left(\Theta\right)$ is  a KdV vertical line-soliton, that is,  $F= \log$ and $\Theta$ as in \eqref{phase_basica}.
 \medskip

Since $H\left(\Theta\right)=Ai\left(\Theta\right)=W_x^F \left(\Theta\right)=W_y^F\left(\Theta\right)=0$, it is sufficient to look at equation \eqref{rhoKP}, to conclude that
\begin{equation*}
     \left(\rho \left(s\right)'' - 2F'\left(\Theta\right)\rho \left(s\right)' + 4F''\left(\Theta\right)\rho \left(s\right)\right)\Theta^{4}_{x} + 6\rho \left(s\right) '\Theta^{2}_{x}\Theta_{xx} + 3\rho \left(s\right)\left(\Theta^{2}_{xx} + \Theta^{2}_{y}\right) = 0.
\end{equation*}
By the hypothesis on $F$ and the values of its derivatives at $s=1$, if there exists a solution to the KP equation of the form $u = 2\partial^{2}_{x}F\left(\Theta\right)$, \eqref{rhoKP} is satisfied, and by Lemma \ref{dem:ODE2} $\left(iii\right)$, $F=\log$.

\medskip

Now, Lemma \ref{HWxWy} implies that $W^{F}_{y}\left(\Theta\right) - W^{F}_{x}\left(\Theta\right)=W_{y}\left(\Theta\right) - W_{x}\left(\Theta\right)=0$. Since $F=\log$, $W_{y}\left(\Theta\right)=W_{x}\left(\Theta\right)=0$. Since $W_{y}\left(\Theta\right) = W_{x}\left(\Theta\right) = Ai\left(\Theta\right) = 0$, Lemma \ref{4p11} ensures that $\Theta = A\left(t,x\right)\exp\left(cy\right)$, with $A\left(t,x\right)$ being of Airy type and with $c \in \mathbb R$.
 
\medskip

Finally, the condition $ H\left(\Theta\right) = 0$ implies 
\begin{equation*}
    \Theta _{y} - \Theta _{xx} = \left(c A - A_{xx}\right)\exp\left(c y\right)= 0.
\end{equation*}
We first treat the case $c=0$. In this case
\[
A\left(t,x\right) = c_{A,1}\left(t\right)+ c_{A,2}\left(t\right)x.
\]
Since $A$ must satisfy the Airy equation for all $\left(t,x,y\right)\in\mathbb R^{3}$, one gets $c_{A,1}'\left(t\right)+ c_{A,2}'\left(t\right)x =0$, implying that $c_{A,1}$ and $c_{A,2}$ are constants. $\Theta$ is given in this case by
\[
\Theta= c_{A,1} + c_{A,2} x, 
\] 
corresponding to a singular soliton solution, which is discarded by smoothness assumptions. 

\medskip

Now we assume $c$ different from zero. Here, 
\begin{equation}\label{Ademvertical}
    A\left(t,x\right) = c_{A,1}\left(t\right)\exp\left(kx\right) + c_{A,2}\left(t\right)\exp\left(-kx\right), \quad k=\sqrt{c}\in\mathbb C-\{0\}.
\end{equation}
Solving again the Airy equation for $A$, it follows 
\begin{equation*}
\begin{aligned}
A_{t} - A_{xxx} &= c'_{A,1}\left(t\right)\exp\left( kx\right) + c'_{A,2}\exp\left(-kx\right) \\
    &- \left(k^3 c_{A,1} \exp\left(kx\right) - k^3 c_{A,2}\exp\left(-kx\right)\right)= 0.
\end{aligned}
\end{equation*}
By the  linear independence, 
\begin{equation*}
    \begin{aligned}
        c'_{A,1}\left(t\right) - k^3 c_{A,1}\left(t\right)= 0;\\
        c'_{A,2}\left(t\right) + k^3 c_{A,2}\left(t\right)= 0.
    \end{aligned}
\end{equation*}
This are independents ODE's for $c_{A,1}\left(t\right)$ and $c_{A,2}\left(t\right)$. Solving them, it is obtained 
\begin{equation}\label{final_final}
        c_{A,1}\left(t\right) = c_{A,1,0} \exp\left(k^3 t\right), \quad  c_{A,2}\left(t\right) = c_{A,2,0} \exp\left(-k^3 t\right),
\end{equation}
with $ c_{A,1,0}, c_{A,2,0} \in \mathbb R$ arbitrary constants. We obtain from \eqref{Ademvertical} and \eqref{final_final} 
\begin{equation*}
    \Theta \left(t,x,y\right) =  c_{A,1,0} \exp\left(kx + k^{2}y + k^{3}t\right) +  c_{A,2,0}\exp\left(-kx + k^{2}y - k^{3}t\right).
\end{equation*}
The condition $\Theta\in\mathbb R$ implies that  $k$ is real-valued. Also, $\Theta>0$ implies $  c_{A,1,0},  c_{A,2,0}>0$. This finally shows that $\Theta$ corresponds to the phase of a KdV vertical line-soliton. The proof is complete.

\subsection{Proof of Theorem \ref{MT1b}}
The proof of this result is based in two lemmas. Since Theorem \ref{MT1} considers the case of KdV line solitons, we focus on the most demanding case of oblique solitons $\left(A>0\right)$.

\begin{lem}\label{Lema2a}
 Let $u$ be a smooth solution to \eqref{eq:KP} of the form \eqref{eqn:FT}, with a smooth profile $F\left(\Theta\right) $ such that $F\left(1\right)=0$, $F'\left(1\right)=1$, $F''\left(1\right)=-1$, and $F'''\left(1\right)=2$. 
Then if $u$ is a line-soliton of the form \eqref{soliton family}-\eqref{linesoliton} and $F=\log$, one has that $H\left(\Theta\right) = Ai\left(\Theta\right) = 0$, and 
\begin{equation}\label{thetaW_y_new}
\begin{aligned}
& \Theta W_{x}\left(\Theta\right) = \Theta W_{y}\left(\Theta\right) \\
& = a_{1}a_{2}\left(k^{2}_{1}-k^{2}_{2}\right)^{2}\exp\left(\left(k_{1}+k_{2}\right)x+\left(k^{2}_{1}+k^{2}_{2}\right)y+\left(k^{3}_{1}+k^{3}_{2}\right)t\right), 
\end{aligned}
\end{equation}
for some particular $a_1, a_2>0$, $k_1,k_2\in\mathbb R$.
%
\end{lem}
\begin{proof}
By Lemma \ref{lem:linesoliton}, we know that $H\left(\Theta\right) = Ai\left(\Theta\right) = 0$. Thanks to \eqref{thetaW_y} and Lemma \ref{HWxWy}, we conclude  \eqref{thetaW_y_new}. This proves Lemma \ref{Lema2a}.
\end{proof}

\begin{lem}\label{Lema2b}
   Let $u$ be a smooth solution to \eqref{eq:KP} of the form \eqref{eqn:FT}, with $\Theta>0$ smooth and real-valued and $F$ a smooth profile such that $F\left(1\right)=0$, $F'\left(1\right)=1$, $F''\left(1\right)=-1$, and $F'''\left(1\right)=2$. 
If $H\left(\Theta\right) = Ai\left(\Theta\right) = 0$, and 
\begin{equation}\label{condition_1dim}
\Theta W_{y}\left(\Theta\right) = A\left(t,x\right) \exp \left( k\left(t,x\right) y \right), 
\end{equation}
for some particular $A>0$, $k>0$, then $u$ is a line-soliton of the form \eqref{soliton family}-\eqref{linesoliton} and $F=\log$.
\end{lem}

\begin{rem}
Lemmas \ref{Lema2a} and \ref{Lema2b} conclude the proof of Theorem \ref{MT1b}.
\end{rem}

\begin{proof}[Proof of Lemma \ref{Lema2b}]
\emph{Step 1.} Since $H\left(\Theta\right)=Ai\left(\Theta\right)=0$, Lemmas \ref{HWxWy} and \ref{WsAidaF} and the hypotheses on $F$ ensure that $F=\log$. Let us assume \eqref{condition_1dim}.  Then
\[
    \Theta^{2}\partial^{2}_{y}\left(\log \left(\Theta\right)\right) = A\left(t,x\right) \exp \left( k\left(t,x\right) y \right). 
\]
This equation can be studied like a nonlinear second order ODE for $\Theta$ in the variable $y$. 
Considering the change of variable $f= \log  \Theta$ one gets
\begin{equation*}
   \exp\left(2f\right)f''=A\exp\left(ky\right), \quad \hbox{then} \quad  \tilde{f}''=2A\exp\left(-\tilde{f}\right),
\end{equation*}
with $\tilde{f}=-ky+2f$, and consequently $\tilde{f}''=2f''$. This is a classical Toda equation. The general solution $\tilde{f}$ is given by 
\[
    \tilde{f}=2 \log\left(\frac{\sqrt{A}}{\sqrt{c_{1}}}  \left( \exp \left(\frac12\eta \right)+  \exp \left(-\frac12\eta \right) \right)\right),
\]
with $\eta:= \sqrt{c_{1}}\left(y+c_{2}\right)$, $c_1>0$, $c_2\in\mathbb R$. 
Therefore, 
\begin{equation*}
\begin{aligned}
    \Theta = \exp\left(f\right) &= \exp\left(\frac{ky}{2}\right)\exp\left(\frac{\tilde{f}}{2}\right)\\
&=\frac{\sqrt{A}}{\sqrt{c_{1}}}\left(\exp\left(\left(\frac{k+\sqrt{c_{1}}}{2}\right)\left(y+c_{2}\right)\right)+\exp\left(\left(\frac{k-\sqrt{c_{1}}}{2}\right)\left(y+c_{2}\right)\right) \right).
\end{aligned}
\end{equation*}
Defining 
\[
A_1\left(t,x\right)=\frac{\sqrt{A}}{\sqrt{c_{1}}}\exp\left(c_{2}\left(\frac{k+\sqrt{c_{1}}}{2}\right)\right), \quad A_2\left(t,x\right)=\frac{\sqrt{A}}{\sqrt{c_{1}}}\exp\left(c_{2}\left(\frac{k-\sqrt{c_{1}}}{2}\right)\right),
\]
and
\[
B_{1}\left(t,x\right)=\left(\frac{k+\sqrt{c_{1}}}{2}\right)>0,\quad B_{2}\left(t,x\right)=\left(\frac{k-\sqrt{c_{1}}}{2}\right),
\]
$\Theta$ has the form $\Theta = A_1\left(t,x\right)\exp\left(B_{1}\left(t,x\right)y\right) + A_2\left(t,x\right)\exp\left(B_{2}\left(t,x\right)y\right).$\\

\emph{Step 2.} Since $c_1>0$, one gets $B_1$ different from $B_2$. Replacing $\Theta$ in the condition $H\left(\Theta\right)=0$, and using Appendix \ref{derivadas_teta} in the case of two linearly independent exponentials,
\begin{equation}\label{Computation of H}
\begin{aligned}
H\left(\Theta\right)= &~{} \left(A_{1}B_{1}\exp\left(B_{1}y\right)+A_{2}B_{2}\exp\left(B_{2}y\right)\right) \\
& - \Big(A_{1,xx}\exp\left(B_{1}y\right)+A_{2,xx}\exp\left(B_{2}y\right) \\
&\qquad  + y\exp\left(B_{1}y\right)\left(2A_{1,x}B_{1,x}+A_{1}B_{1,xx}\right) + y\exp\left(B_{2}y\right)\left(2A_{2,x}B_{2,x}+A_{2}B_{2,xx}\right)\\
&\qquad +y^{2}\exp\left(B_{1}y\right)A_{1}B^{2}_{1,x}+y^{2}\exp\left(B_{2}y\right)A_{2}B^{2}_{2,x}\Big)\\
=&~{} \left(A_{1}B_{1}-A_{1,xx}\right)\exp\left(B_{1}y\right)+\left(A_{2}B_{2}-A_{2,xx}\right)\exp\left(B_{2}y\right)\\
&-\left(2A_{1,x}B_{1,x}+A_{1}B_{1,xx}\right)y\exp\left(B_{1}y\right)-\left(2A_{2,x}B_{2,x}+A_{2}B_{2,xx}\right)y\exp\left(B_{2}y\right)\\
&-A_{1}B^{2}_{1,x}y^{2}\exp\left(B_{1}y\right)-A_{2}B^{2}_{2,x}y^{2}\exp\left(B_{2}y\right).
\end{aligned}
\end{equation}
In order for  $H(\Theta)$ to be null, it is necessary that each of the coefficients multiplying a term $y^{i}\exp\left(B_{j}y\right)$ with $i\in\left\{0,1,2\right\}$ and $j\in\left\{1,2\right\}$ be equal to zero. Taking the expression above into consideration, the following system of equations is obtained
\begin{equation*}
\begin{aligned}
& A_{1}B_{1}=A_{1,xx}, \quad A_{2}B_{2}=A_{2,xx},\\
& 2A_{1,x}B_{1,x}+A_{1}B_{1,xx}=0, \quad 2A_{2,x}B_{2,x}+A_{2}B_{2,xx}=0,\\
& A_{1}B^{2}_{1,x}=0, \quad A_{2}B^{2}_{2,x}=0.
\end{aligned}
\end{equation*}
Since $A_1,A_2>0$, from the two bottom equations, $B_{i}\left(t,x\right)=B_{i}\left(t\right)$ for $i\in\left\{1,2\right\}$. This reduces the system to the top two equations, from which it can be concluded that $A_{i}\left(t,x\right)=c_{i}\left(t\right)\exp\left(\sqrt{B_{i}\left(t\right)}x\right)$. Since $\Theta>0$ is real-valued, it is required $B_i>0$. Therefore, the phase takes the form
\begin{equation*}
    \Theta = c_{1}\left(t\right)\exp\left(\sqrt{B_{1}\left(t\right)}x+B_{1}\left(t\right)y\right) + c_{2}\left(t\right)\exp\left(\sqrt{B_{2}\left(t\right)}x+B_{2}\left(t\right)y\right).
\end{equation*}
Then, using again Appendix \ref{derivadas_teta} and inserting those terms into $Ai\left(\Theta\right)=0$, we get
\begin{equation*}
\begin{aligned}
Ai\left(\Theta\right)= &~{} \Big(c_{1,t}\exp\left(\sqrt{B_{1}}x+B_{1}y\right)+\frac{c_{1}B_{1,t}}{2\sqrt{B_{1}}}x\exp\left(\sqrt{B_{1}}x+B_{1}y\right)\\
&~{} +c_{1}B_{1,t}y\exp\left(\sqrt{B_{1}}x+B_{1}y\right) +c_{2,t}\exp\left(\sqrt{B_{2}}x+B_{2}y\right) \\
&+\frac{c_{2}B_{2,t}}{2\sqrt{B_{2}}}x\exp\left(\sqrt{B_{2}}x+B_{2}y\right)+c_{2}B_{2,t}y\exp\left(\sqrt{B_{2}}x+B_{2}y\right)\Big)\\
&-\Big(c_{1}\sqrt{B^{3}_{1}}\exp\left(\sqrt{B_{1}}x+B_{1}y\right)+c_{2}\sqrt{B^{3}_{2}}\exp\left(\sqrt{B_{2}}x+B_{2}y\right)\Big).
\end{aligned}
\end{equation*}
Simplifying the expression above, one has
\begin{equation*}
\begin{aligned}
Ai\left(\Theta\right)= &~{} \left(c_{1,t}-c_{1}\sqrt{B^{3}_{1}}\right)\exp\left(\sqrt{B_{1}}x+B_{1}y\right)+\left(c_{2,t}-c_{2}\sqrt{B^{3}_{2}}\right)\exp\left(\sqrt{B_{2}}x+B_{2}y\right) \\
&+ \frac{c_{1}B_{1,t}}{2\sqrt{B_{1}}}x\exp\left(\sqrt{B_{1}}x+B_{1}y\right)+\frac{c_{2}B_{2,t}}{2\sqrt{B_{2}}}x\exp\left(\sqrt{B_{2}}x+B_{2}y\right)\\
&+c_{1}B_{1,t}y\exp\left(\sqrt{B_{1}}x+B_{1}y\right)+c_{2}B_{2,t}y\exp\left(\sqrt{B_{2}}x+B_{2}y\right).
\end{aligned}
\end{equation*}
To ensuring that the expression is equal to zero, it is necessary that each coefficient multiplying an exponential term be null for all values of $\left(x,y\right)\in\mathbb R^{2}$. Taking the above into consideration, the following system of equations is obtained
\begin{equation*}
\begin{aligned}
& c_{1,t}-c_{1}\sqrt{B^{3}_{1}}=0, \quad c_{2,t}-c_{2}\sqrt{B^{3}_{2}}=0,\\
& c_{1}B_{1,t}=0, \quad c_{2}B_{2,t}=0.
\end{aligned}
\end{equation*}
Note that the last two equations are derived from the coefficients multiplying an exponential term, multiplied either by $x$ or by $y$. From the last two equations, it is concluded that $B_{i}\left(t\right)=B_{i}$ for $i\in\left\{1,2\right\}$ and with $B_{i}\in\mathbb R$. Taking this into account in the first two equations is obtained $c_{i}=a_{i}\exp\left(\sqrt{B^{3}_{i}}t\right)$ for $i\in\left\{1,2\right\}$ and with $a_{i}>0$. In conclusion, the phase is
\begin{equation*}
    \Theta = a_{1}\exp\left(\sqrt{B_{1}}x+B_{1}y+\sqrt{B^{3}_{1}}t\right) + a_{2}\exp\left(\sqrt{B_{2}}x+B_{2}y+\sqrt{B^{3}_{2}}t\right).
\end{equation*}
where $a_{1}$, $a_{2}>0$ are arbitrary constants. Denoting $k_i:= \sqrt{B_i}$, we obtain the desired conclusion.
\end{proof}

\subsection{Proof of Theorem \ref{MT2}}
{\color{black}
Let $u$ be a (smooth) solution of \eqref{eq:KP} of the form \eqref{eqn:FT} with a smooth real-valued phase $\Theta>0$ satisfying for $k=0,1,2,3$, the hypotheses \eqref{hypohypo}:
\[
\partial_x^k \Theta\left(t,0,0\right), \quad \partial_x^k \partial_y \Theta\left(t,0,0\right) \quad \hbox{uniquely prescribed.}
\]
Assume that the smooth profile $F$ satisfies \eqref{F_conds}. Assume that $\Theta$ corresponds to an $M$ resonant multi-soliton \eqref{resonant_0} and $F=\log$. Since $\Theta>0$ as in \eqref{theta_multi_linesoliton} generates an $M$ resonant soliton, then from Corollary \ref{resonantes} $Ai\left(\Theta\right) = H\left(\Theta\right) = 0$.}
Also, Lemma \ref{WM} in this particular case $\left(k_j^2\geq 0\right)$ states that  $\Theta W_y\left(\Theta\right)= \Theta W_x\left(\Theta\right) \in \mathcal W_{\frac12M\left(M-1\right)}$. 
\medskip

Now we prove the opposite. Assume that $\Theta W_y\left(\Theta\right)\in \mathcal W_{\frac12M\left(M-1\right)}$ has the form
\begin{equation}\label{sistema_del_demonio}
\Theta W_y\left(\Theta\right) = \sum_{j=1}^{\frac12M\left(M-1\right)} b_j \exp\left(m_j y\right).
\end{equation}
Recall that each $b_j>0$. There are $\frac12M\left(M-1\right)$ terms in the summation above. Relabeling and  arranging terms above  as  members of a $M\times M$ lower triangular matrix, we get
\[
\Theta W_y\left(\Theta\right) = \sum^{M}_{n = 1} \sum^{n}_{i=1} b_{n,i} \exp\left(m_{n,i}y\right).
\]
 For the moment, assume 
 \begin{equation}\label{guess}
 \Theta= \sum_{j=1}^M a_j\left(t,x\right) \exp\left(\theta_j\right),
 \end{equation}
 with $\theta_j\left(t,x,y\right)=k_j\left(t,x\right) y$, be a phase in $\mathcal W_M$. From Corollary \ref{WM},
 \[
\Theta W_y\left(\Theta\right)  =  \sum^{M}_{n = 1} \sum^{n}_{i=1}a_{i}a_{n-i+1} \left(k_{n-i+1} -k_{i}\right)^{2}\exp\left(\theta_{i} + \theta_{n-i+1}\right).
 \]
The system 
\[
k_{i} + k_{n-i+1} = m_{n,i},  \quad  i \leq n-i+1,
\]
reads
\[
\begin{aligned}
& k_1 + k_{1} = m_{1,1}\\
& k_1 + k_{2} = m_{2,1}\\
& k_1 + k_{3} = m_{3,1}\\
& k_2 + k_{2} = m_{3,2}\\
&  \cdots \\
& k_1 + k_{n} = m_{n,1}\\
& k_2 + k_{n-1} = m_{n,2}\\
& k_3 + k_{n-2} = m_{n,3}\\
& \cdots \\
\end{aligned}
\]
and has a unique solution on $k=\left(k_1,\ldots,k_N\right)$, thanks to a nonsingular determinant matrix. The second system is given by
\[
\begin{aligned}
& a_{i}a_{n-i+1} \left(k_{n-i+1} -k_{i}\right)^{2} = b_{n,i},\\
\end{aligned}
\]
which implies
\[
\begin{aligned}
& \qquad   \log a_i + \log a_{n-i+1} = \log b_{n,i} -2\log \left| k_{n-i+1} -k_{i}\right|.
\end{aligned}
\]
This system also has a unique solution for $a=\left(a_1,a_2,\ldots, a_N\right)$ exactly following the previous argument. Consequently, $\Theta W_y\left(\Theta\right) \in \mathcal W_{\frac12M\left(M-1\right)}$ has always a solution in $\mathcal W_M$. The fact that $ \Theta W_x\left(\Theta\right) \in \mathcal W_{\frac12M\left(M-1\right)}$ is direct.

\medskip

Now we prove uniqueness. The key now is to use that $\Theta W_x\left(\Theta\right) \in \mathcal W_{\frac12M\left(M-1\right)}$ and has a unique value. Let $F\in \mathcal W_{\frac12M\left(M-1\right)}$ and let $\Theta_1,\Theta_2$ be such that $ \Theta_j W_x\left(\Theta_j\right) =F$. Therefore,
\[
\Theta_1 \Theta_{1,xxxx}-\Theta_{1,xx}^2 = \Theta_2 \Theta_{2,xxxx}-\Theta_{2,xx}^2. 
\]
Let $\Pi_0\left(x;t\right):= \left(\Theta_1-\Theta_2\right)\left(t,x,y=0\right)$. Then $\Pi_0$ satisfies the fourth order linear ODE on $x$:
\[
\Theta_2 \Pi_0'''' - \left(\Theta_{1,xx} +\Theta_{2,xx}\right)\Pi_0'' + \Theta_{1,xxxx}\Pi_0 =0,
\]
and thanks to the hypothesis \eqref{hypohypo}, $\Pi_0\left(x=0;t\right)=\Pi_0'\left(x=0;t\right)=\Pi_0''\left(x=0;t\right)=\Pi_0'''\left(x=0;t\right)=0$, leading to $\Pi_0(x;t)\equiv 0$ and consequently $\Theta_1\left(t,x,0\right) = \Theta_2\left(t,x,0\right)$. Additionally, $\Pi_1\left(x;t\right):= \left(\Theta_{1,y}-\Theta_{2,y}\right)\left(t,x,y=0\right)$ satisfies
\[
\Theta_2 \Pi_1'''' -2\Theta_{1,xx} \Pi_1'' + \Theta_{1,xxxx}\Pi_1 =0.
\]
(Notice that we have used that $\Pi_0\left(x;t\right)=\Pi_{0,xx}\left(x;t\right)=\Pi_{0,xxxx}\left(x;t\right)=0$.) Again, thanks to  \eqref{hypohypo} we conclude that $\Pi_1\left(x;t\right)\equiv 0$, leading to $\Theta_{1,y}\left(t,x,0\right) = \Theta_{2,y}\left(t,x,0\right)$.

\medskip

Now we extend the previous uniqueness. The argument is similar to the previous case. Let $G\in \mathcal W_{\frac12M\left(M-1\right)}$ unique and let $\Theta_1,\Theta_2$ be such that $ \Theta_j W_y\left(\Theta_j\right) =G$. Therefore,
\[
\Theta_1 \Theta_{1,yy}-\Theta_{1,y}^2 = \Theta_2 \Theta_{2,yy}-\Theta_{2,y}^2. 
\]
Let $\Pi_2 \left(y;t,x\right):= \left(\Theta_1-\Theta_2\right)\left(t,x,y\right)$. Then $\Pi_2$ satisfies the second order linear ODE on $y$:
\[
\Theta_2 \Pi_2'' - \left(\Theta_{1,y} +\Theta_{2,y}\right)\Pi_2' + \Theta_{1,yy}\Pi_2 =0,
\]
and thanks to the previous step, one has $\Pi_2\left(0;t,x\right)=\Pi_{2,y}\left(0;t,x\right)=0$. Therefore,  $\Pi\left(y;t,x\right)=0$ and the required uniqueness holds. This shows that \eqref{guess} is the unique solution of \eqref{sistema_del_demonio}.

\medskip

Now we improve the coefficients using that $\Theta$ satisfies zero Heat and Airy, showing that $\Theta$ is a multi-line-soliton. Since Heat and Airy are linear equations, the proof is similar to the proof of Lemma \ref{Lema2b}, Step 2. More precisely, 
\begin{equation}\label{calculo de H}
\begin{aligned}
H\left(\Theta\right)=&~{} \sum_{j=1}^{M} \left\{ \left(a_{j}k_{j}-a_{j,xx}\right)-\left(2a_{j,x}k_{j,x}+a_{j}k_{j,xx}\right)y-a_{j}k^{2}_{j,x}y^{2} \right\} \exp\left(k_{j}y\right) =0.
\end{aligned}
\end{equation}
By the linear independence, real-valued character and positivity of $\Theta$, we conclude $k_j=k_j(t)\geq 0$ and 
\[
a_{j}\left(t,x\right)=a_{0,j}(t)\exp\left(\sqrt{k_{j} (t)}x\right) +a_{1,j}(t)  \exp\left(-\sqrt{k_{j}(t) }x\right), \quad a_{0,j}, a_{1,j}\geq 0.
\]
Computing Airy, one gets
\begin{equation}\label{calculo de Ai}
\begin{aligned}
Ai\left(\Theta\right) =&~{}\sum_{j=1}^M \left\{  \left(a_{0,j,t}-a_{0,j}\sqrt{k^{3}_{j}}\right)  + \frac{a_{0,j}k_{j,t}}{2\sqrt{k_{j}}}x +a_{0,j}k_{j,t}y \right\} \exp\left(\sqrt{k_{j}}x+k_{j}y\right)\\
&~{} +\sum_{j=1}^M \left\{  \left(a_{1,j,t} + a_{1,j}\sqrt{k^{3}_{j}}\right)  - \frac{a_{1,j}k_{j,t}}{2\sqrt{k_{j}}}x +a_{1,j}k_{j,t}y \right\} \exp\left(-\sqrt{k_{j}}x+k_{j}y\right)\\
=&~{}0.
\end{aligned}
\end{equation}
From \eqref{calculo de H} and \eqref{calculo de Ai}, and proceeding exactly as in the previous proof of Lemma \ref{Lema2b}, Step 2, one gets 
\[
a_{0,j}=a_{0,0,j}\exp\left(\sqrt{k^{3}_{j}}t\right), \quad a_{1,j}=a_{1,0,j}\exp\left(-\sqrt{k^{3}_{j}}t\right),
\]
with $k_j$ constants. Renaming $k_j \mapsto k_j^2$, from \eqref{guess} we obtain
\[
\Theta= \sum_{j=1}^M \left( a_{0,0,j} \exp\left( k_{j} x+k^{3}_{j} t\right) + a_{1,0,j} \exp\left( -k_{j} x-k^{3}_{j} t\right)\right) \exp\left(k_j^2 y\right).
\]
This ends the proof.

\subsection{Proof of Theorem \ref{MT3}} We are now ready to prove Theorem \ref{MT3}, the case of KP 2-solitons.  Let $u$ be a solution of \eqref{eq:KP} of the form \eqref{eqn:FT} with a smooth real-valued phase $\Theta>0$ satisfying $\Theta\left(y=0\right)$ and $\Theta_y\left(y=0\right)$ uniquely prescribed and $F=\log$. 

\medskip

Assume that $\Theta$ corresponds to a $2$-soliton \eqref{2soliton} with $k_1<k_2<k_3<k_4$. Let us prove that $H\left(\Theta\right),  Ai\left(\Theta\right)  \in \mathcal W_4$ and $ \Theta W_y\left(\Theta\right), \Theta W_x\left(\Theta\right) \in \mathcal W_5$. Thanks to Lemma \ref{2solitonphase} $\left(i\right)$ and $\left(ii\right)$, this part is already proved. Finally, $Ai\left(\Theta\right) = \frac32 \partial_x H\left(\Theta\right)$ is direct.

\medskip

Assume now that $H\left(\Theta\right),  Ai\left(\Theta\right)  \in \mathcal W_4$ and $ \Theta W_y\left(\Theta\right), \Theta W_x\left(\Theta\right) \in \mathcal W_5$. We have $ \Theta W_y\left(\Theta\right)\in \mathcal W_6$, and the proof of Theorem \ref{MT2} and hypotheses \eqref{hypohypo} on $\Theta$ ensure that $\Theta\in \mathcal W_4$. As in \eqref{guess}, one has
 \begin{equation}\label{guess_0}
 \Theta= \sum_{j=1}^4 a_j\left(t,x\right) \exp\left(\theta_j\right), \quad a_j> 0,
 \end{equation}
with $\theta_j\left(t,x,y\right)=k_j\left(t,x\right) y$. The image of \eqref{guess_0} under $\Theta W_y\left(\Theta\right)$ is given by
 \[
 \begin{aligned}
\Theta W_y\left(\Theta\right)  =  &~{}\sum^{4}_{n = 1} \sum^{n}_{i=1}a_{i}a_{n-i+1} \left(k_{n-i+1} -k_{i}\right)^{2}\exp\left(\theta_{i} + \theta_{n-i+1}\right)\\
=&~{} a_1 a_2 \left(k_1-k_2\right)^2 E_{12} + a_1 a_3 \left(k_1-k_3\right)^2E_{13} + a_1 a_4 \left(k_1-k_4\right)^2 E_{14} \\
&~{} +a_2 a_3 \left(k_2-k_3\right)^2E_{23}  + a_2 a_4 \left(k_2-k_4\right)^2E_{24}  +a_3 a_4 \left(k_3-k_4\right)^2E_{34},
\end{aligned}
\]
where $E_{ij} := \exp (\theta_i +\theta_j)$.
Since $\Theta W_y\left(\Theta\right) \in \mathcal W_5$, at least one exponential is linearly dependent with the rest of exponentials. This implies that 
$k_i + k_j =k_{i'} +k_{j'}$ for some $i$ different from $i'$, $j$ different from $ j'$ and therefore one term above is redundant. With no loss of generality, we assume  $k_1<k_2<k_3 <k_4$ and $k_1+k_4=k_2+k_3$. In this case we obtain $k_4= k_2+k_3-k_1>0$, and from  \eqref{guess_0},
 \begin{equation}\label{guess_1}
 \Theta\left(t,x,y\right)= a_1 \exp\left(k_1 y\right) + a_2 \exp\left(k_2y\right) +a_3 \exp\left(k_3 y\right) + a_4 \exp\left(\left(k_2+k_3-k_1\right)y\right),
 \end{equation}
 and
 \[
 \begin{aligned}
\Theta W_y\left(\Theta\right) =&~{} a_1 a_2 \left(k_1-k_2\right)^2 E_{12} + a_1 a_3 \left(k_1-k_3\right)^2E_{13} \\
&~{} + \left(a_1 a_4 \left(k_1-k_4\right)^2 +a_2 a_3 \left(k_2-k_3\right)^2\right)E_{14} \\
&~{}   + a_2 a_4 \left(k_2-k_4\right)^2E_{24}  +a_3 a_4 \left(k_3-k_4\right)^2E_{34}   .
\end{aligned}
\]
$\Theta$ in \eqref{guess_1} can be written as follows: for $\widetilde k_i\geq 0,$
\[
\widetilde k_1^2+ \widetilde k_3^2: = k_1, \quad \widetilde k_1^2+ \widetilde k_4^2: = k_2, \quad \widetilde k_2^2+\widetilde k_3^2:= k_3.
\]
Then $k_4= k_2+k_3-k_1= \widetilde k_2^2+\widetilde k_4^2$, exactly as in \eqref{2soliton}. Therefore, our new variables will be $\widetilde k_i$, but in order to avoid too much notation, we drop the tildes. Replacing in \eqref{guess_1}, we obtain a new representation of $\Theta$:
 \begin{equation}\label{guess_2}
\begin{aligned}
 \Theta\left(t,x,y\right)= &~{} a_1 \exp\left(\left( k_1^2+ k_3^2\right) y\right) + a_2 \exp\left( \left( k_1^2+ k_4^2\right) y\right) \\
 &~{} +a_3 \exp\left(\left( k_2^2+ k_3^2\right) y\right) + a_4 \exp\left(\left( k_2^2+ k_4^2\right)y\right).
\end{aligned}
\end{equation}
(Compare with \eqref{2soliton}.) Repeating again \eqref{calculo de H} with \eqref{guess_2}, we obtain e.g.
\[
\begin{aligned}
& H\left(a_1 \exp\left(\left( k_1^2+ k_3^2\right) y\right)  \right)\\
&~{}  = \left( a_1\left( k_1^2+ k_3^2\right) -a_{1,xx}\right) \exp\left(\left( k_1^2+ k_3^2\right) y\right) \\
&~{} \quad +\left( - 4 a_{1,x} (k_1 k_{1,x} +k_3k_{3,x})  - 2a_1(k_1 k_{1,x} +k_3k_{3,x})_x \right) y\exp\left(\left( k_1^2+ k_3^2\right) y\right) \\
&~{} \quad  - 4a_1(k_1 k_{1,x} +k_3k_{3,x})^2 y^2 \exp\left(\left( k_1^2+ k_3^2\right) y\right).
\end{aligned}
\]
By the linear independence among the exponentials, the nontrivial character of $\Theta$, and the hypothesis $H\left(\Theta\right)\in\mathcal W_4$, it is clear that one will obtain $k_1 k_{1,x} +k_3k_{3,x} =0$, so that if $k_{j}\left(t,0\right)=: k_j\left(t\right)$,
\[
\left(k_1^2+ k_3^2\right)\left(t,x\right) =k_1^2\left(t\right)+ k_3^2\left(t\right) . 
\]
A similar argument reveals that
\[
\begin{aligned}
& \left(k_1^2+ k_4^2\right)\left(t,x\right) =k_1^2\left(t\right)+ k_4^2\left(t\right), \\
& \left(k_2^2+ k_3^2\right)\left(t,x\right) =k_2^2\left(t\right)+ k_3^2\left(t\right), \\
& \left(k_2^2+ k_4^2\right)\left(t,x\right) =k_2^2\left(t\right)+ k_4^2\left(t\right).
\end{aligned}
\]
Consequently, we get in \eqref{guess_2}
 \begin{equation}\label{guess_3}
\begin{aligned}
 \Theta\left(t,x,y\right)= &~{} a_1\left(t,x\right)\exp\left(\left( k_1^2\left(t\right)+ k_3^2\left(t\right)\right) y\right) + a_2\left(t,x\right) \exp\left( \left( k_1^2\left(t\right)+ k_4^2\left(t\right)\right) y\right) \\
 &~{} +a_3 \left(t,x\right) \exp\left(\left( k_2^2\left(t\right)+ k_3^2\left(t\right)\right) y\right) + a_4 \left(t,x\right) \exp\left(\left( k_2^2\left(t\right)+ k_4^2\left(t\right)\right)y\right).
\end{aligned}
\end{equation}
Now, repeating \eqref{calculo de Ai} with \eqref{guess_3} one has
\[
\begin{aligned}
& Ai\left(a_1\left(t,x\right) \exp \left(\left( k_1^2\left(t\right)+ k_3^2\left(t\right)\right) y\right) \right) \\
&~{} =  \left( a_{1,t}\left(t,x\right) +2\left(k_1k_{1,t} + k_2k_{2,t}\right)y a_1\left(t,x\right)  -a_{1,xxx}\left(t,x\right)\right)\exp \left(\left( k_1^2\left(t\right)+ k_3^2\left(t\right)\right) y\right),
\end{aligned}
\]
revealing that $k_1^2\left(t\right)+k_2^2\left(t\right)= k_1^2\left(0\right)+k_2^2\left(0\right)= : k_1^2+k_2^2$ are constants independent of time. Similarly,
\[
\begin{aligned}
& \left(k_1^2+ k_4^2\right)\left(t\right) =k_1^2\left(0\right)+ k_4^2\left(0\right) =: k_1^2+ k_4^2, \\
& \left(k_2^2+ k_3^2\right)\left(t\right) =k_2^2\left(0\right)+ k_3^2\left(0\right) =: k_2^2+ k_3^2, \\
& \left(k_2^2+ k_4^2\right)\left(t\right) =k_2^2\left(0\right)+ k_4^2\left(0\right)=: k_2^2+ k_4^2.
\end{aligned}
\]
Consequently, we get in \eqref{guess_3},
 \begin{equation}\label{guess_4}
\begin{aligned}
 \Theta\left(t,x,y\right)= &~{} a_1 \left(t,x\right)\exp\left(\left( k_1^2+ k_3^2\right) y\right) + a_2 \left(t,x\right) \exp\left( \left( k_1^2+ k_4^2\right) y\right) \\
 &~{} +a_3 \left(t,x\right) \exp\left(\left( k_2^2+ k_3^2\right) y\right) + a_4 \left(t,x\right) \exp\left(\left( k_2^2+ k_4^2\right)y\right).
\end{aligned}
\end{equation}
From the hypothesis \eqref{hypohypo}, and following $\Theta\left(0,x,0\right)$, $\Theta_y\left(0,x,0\right)$, $\Theta_{yy}\left(0,x,0\right)$ and $\Theta_{yyy}\left(0,x,0\right)$ are uniquely determined by the values from \eqref{2soliton}, leading to the equations
\[
\begin{aligned}
& a_1\left(0,x\right) + a_2\left(0,x\right) +a_3\left(0,x\right) + a_4\left(0,x\right)= \Theta\left(0,x,0\right) \\
& \left(k_1^2+k_3^2\right) a_1\left(0,x\right) + \left(k_1^2+k_4^2\right) a_2\left(0,x\right) \\
& \quad +  \left(k_2^2+k_3^2\right)a_3\left(0,x\right) + \left(k_2^2+k_4^2\right) a_4\left(0,x\right)= \Theta_y\left(0,x,0\right) \\
& \left(k_1^2+k_3^2\right)^2  a_1\left(0,x\right) +  \left(k_1^2+k_4^2\right)^2a_2\left(0,x\right) \\
& \quad +\left(k_2^2+k_3^2\right)^2 a_3\left(0,x\right) + \left(k_2^2+k_4^2\right)^2 a_4\left(0,x\right)= \Theta_{yy}\left(0,x,0\right) \\
& \left(k_1^2+k_3^2\right)^3  a_1\left(0,x\right) + \left(k_1^2+k_4^2\right)^3  a_2\left(0,x\right) \\
& \quad +\left(k_2^2+k_3^2\right)^3 a_3\left(0,x\right) + \left(k_2^2+k_4^2\right)^3 a_4\left(0,x\right)= \Theta_{yyy}\left(0,x,0\right).
\end{aligned}
\]
This is a classical invertible system thanks to the Vandermonde determinant and the condition $ k_1<k_2<k_3<k_4$. Therefore, $a_j\left(0,x\right)$, $j=1,2,3,4$ are uniquely determined:
\[
a_1\left(0,x\right)= \left(k_3-k_1\right) \exp\left( \left(k_1+k_3\right) x\right), \quad a_2\left(0,x\right)= \left(k_4-k_1\right) \exp\left( \left(k_1+k_4\right) x\right),
\]
and
\[
a_3\left(0,x\right)= \left(k_3-k_2\right) \exp\left( \left(k_2+k_3\right) x\right), \quad a_4\left(0,x\right)= \left(k_4-k_2\right) \exp\left( \left(k_2+k_4\right) x\right).
\]
Using \eqref{guess_4}, we compute now $Ai\left(\Theta\right) - \frac32 \partial_x H\left(\Theta\right)$:
\[
\begin{aligned}
 Ai\left(\Theta\right) - \frac32 \partial_x H\left(\Theta\right) = &~{} \left( a_{1,t} + \frac12 a_{1,xxx} - \frac32 \left( k_1^2 + k_3^2\right)a_{1,x}\right) \exp\left(\left( k_1^2+ k_3^2\right) y\right) \\
&~{}  + \left( a_{2,t} + \frac12 a_{2,xxx} - \frac32 \left( k_1^2 + k_4^2\right)a_{2,x} \right) \exp\left(\left( k_1^2+ k_4^2\right) y\right) \\
&~{}   + \left( a_{3,t} + \frac12 a_{3,xxx} - \frac32 \left( k_2^2 + k_3^2\right)a_{3,x}\right) \exp\left(\left( k_2^2+ k_3^2\right) y\right) \\
&~{}    + \left( a_{4,t} + \frac12 a_{4,xxx} - \frac32 \left( k_2^2 + k_4^2\right)a_{4,x}\right) \exp\left(\left( k_2^2+ k_4^2\right) y\right) =0.
\end{aligned}
\]
Therefore,
\begin{equation}\label{airy_final}
\begin{aligned}
& a_{1,t} + \frac12 \partial_x \left( a_{1,xx} -3 \left( k_1^2 + k_3^2\right)a_{1}\right) =0,\\
& a_1\left(0,x\right)= \left(k_3-k_1\right) \exp\left( \left(k_1+k_3\right) x\right).
\end{aligned}
\end{equation}
We shall prove that the unique solution to this problem is 
\begin{equation}\label{guess_5}
a_1\left(t,x\right)= \left(k_3-k_1\right) \exp\left( \left(k_1+k_3\right) x+ \left(k_1^3+k_3^3\right)t \right).
\end{equation}
If we assume this equality, it is not hard to see that a similar argument reveals that 
\[
\begin{aligned}
a_2\left(t,x\right)= &~{} \left(k_4-k_1\right) \exp\left( \left(k_1+k_4\right) x+ \left(k_1^3+k_4^3\right)t \right), \\
a_3\left(t,x\right)= &~{} \left(k_3-k_2\right) \exp\left( \left(k_2+k_3\right) x+ \left(k_2^3+k_3^3\right)t \right),
\end{aligned}
\]
and 
\[
a_4\left(t,x\right)= \left(k_4-k_2\right) \exp\left( \left(k_2+k_4\right) x+ \left(k_2^3+k_4^3\right)t \right).
\]
This finally proves  Theorem \ref{MT3}. Let us show \eqref{guess_5}. Clearly the RHS of \eqref{guess_5} is a valid solution to \eqref{airy_final} satisfying the initial condition $a_1\left(0,x\right)= \left(k_3-k_1\right) \exp\left( \left(k_1+k_3\right) x\right)$. Let us show that it is the unique one. 

\medskip

First, notice that if $\widetilde a_1\left(t,y\right)$ is a function such that $a_1\left(t,x\right) = \widetilde a_1 \left( t, 2^{1/3} \left( x+ \frac32\left(k_1^2+k_3^2\right)t  \right)\right)$, then
\[
 \widetilde a_{1,t}  + \widetilde a_{1,xxx}   =0.
\]
Consequently, by the following uniqueness we get the desired result.

\begin{lem}[Uniqueness of exponentially growing Airy solutions]
Let $m_1,m_2>0$. There is a unique solution $u$ of 
\[
\partial_t u +\partial_x^3 u =0, \quad u\left(t=0,x\right) =  m_1 \exp\left( m_2  x\right),
\]
and it is given by 
\[
u\left(t,x\right)= m_1 \exp\left( m_2  x - m_2^3 t \right).
\]
\end{lem}
\begin{proof}
The existence is exactly given by the explicit formula. Let us see the uniqueness, which is equivalent to prove that
\[
\partial_t u +\partial_x^3 u =0, \quad u\left(t=0,x\right) =  0,
\]
has solution $u=0.$ Assume $x>0$. Thanks to the exponential bound on $x$ for $\Theta$, this is obtained by simply taking the Laplace transform, solving the corresponding obtained ODE, and using the uniqueness of the inverse Laplace transform in the exponentially growing class of solutions. The remaining case $x<0$ is proved similarly by changing $u(t,x)$ to $u(t,-x)$. 
\end{proof}

\section{Proof in the ZK case}\label{sec:ZK}

In this section we provide a proof for Theorem \ref{MT_ZK} for the quadratic ZK model \eqref{ZK}. Notice that this equation can bee recast in terms of KdV \eqref{eqn:KdV} as follows:
\begin{equation}\label{KdVZK}
0= \left( -4u_t+ \partial_{x_1}^3 u +6u \partial_{x_1} u \right)+\partial_{x_1} (\Delta_{c} u),
\end{equation}
where $\Delta_c u = \sum_{j=2}^d \partial_{x_j}^2 u$. For simplicity in the notation, let $x=x_1$. Inserting $u=2\partial^2_x F(\Theta)$ into \eqref{KdVZK}, and integrating once in $x$, we obtain
\begin{equation*}
\begin{aligned}
  0=  &~{} \left(F'''' + 6F''^{2}\right)\left(\Theta\right)\Theta_{x}^{4} + 6\left(F'' + F'^2\right)'\left(\Theta\right) \Theta_{x}^{2} \Theta_{xx}  \\
    &~{}  + \left(F'' + F'^2\right)\left(\Theta\right) \left( 3\Theta_{xx}^{2}  -4 \Theta_{x}\left(\Theta_{t} -\Theta_{xxx}\right) \right)+ F'\left(\Theta\right)\left( -4\Theta_{t} +\Theta_{xxx} \right)_x  \\
&~{} + F'^{2} \left(\Theta\right) \left( 3\left(\Theta_{xx}^{2} - \Theta_{x}\Theta_{xxx}\right) -   \Theta_{x}\left(-4\Theta_{t} + \Theta_{xxx} \right)\right)\\
&~{} +\sum_{j=2}^d  \left( \Theta _{x_jx_j} F'\left(\Theta\right)+\Theta_{x_j}^2 F''\left(\Theta\right) \right)_{xx}. 
\end{aligned}
\end{equation*}
The term $\Theta _{x_jx_j} F'\left(\Theta\right)+\Theta_{x_j}^2 F''\left(\Theta\right)$ can be written as
\[
 F'\left(\Theta\right)W_{x_j}^F\left(\Theta\right)+\Theta_{x_j}^2\rho\left(\Theta\right), \quad \rho= F''\left(\Theta\right)+ F'^2 \left(\Theta\right).
\]
Therefore, following the same ideas as in the KP  case, we arrive to the equation
\begin{equation*}
\begin{aligned}
  0=  &~{} \text{EDO}_4\left(\rho\right) -F'\left(\Theta\right)\left(Ai\left(\Theta\right)\right)_x \\
  &~{} +F'^2\left(\Theta\right)\left( 3 W_1\left(\Theta\right) + \Theta_x Ai\left(\Theta\right)\right) +\sum_{j=2}^d\left( F'\left(\Theta\right)W_{x_j}^F\left(\Theta\right)\right)_{xx}, 
\end{aligned}
\end{equation*}
where $\text{EDO}_4\left(\rho\right)$ is a fourth order linear ODE for $\rho$. Now we finish the proof. If $u$ is a KdV soliton and $F=\log$, it is easy to check that \eqref{nuevas} is satisfied. Now, if \eqref{nuevas} is satisfied, $\text{EDO}_4\left(\rho\right) =0$ and $F=\log$. Therefore, $W_{x_j}^F\left(\Theta\right)=W_{x_j}^{\log}\left(\Theta\right)=0$. From $W_1\left(\Theta\right)=0$ we get 
\[
\Theta = a\left(x_2,\ldots,x_d,t\right) \exp\left(x b\left(x_2,\ldots,x_d,t\right) \right) +c\left(x_2,\ldots,x_d,t\right). 
\]
From $W_{x_j}^{\log}\left(\Theta\right) =0$ we get that $a,b,c$ are only time dependent. Finally, from Airy we obtain that $u=Q_k$, for some constant $k>0$. 

\subsection{Proof in the mZK case}\label{sec:mZK}

In this section we provide proofs for Theorem \ref{MT_mZK} for the cubic version ZK model \eqref{ZK}. Notice that, as above, this equation can be recast in terms of mKdV \eqref{eqn:mKdV} as follows:
\begin{equation*}
0= \left( -4u_t+ \partial_{x_1}^3 u +6u^2 \partial_{x_1} u \right)+\partial_{x_1} \left(\Delta_{c} u\right),
\end{equation*}
where $\Delta_c u = \sum_{j=2}^d \partial_{x_j}^2 u$. For simplicity in the notation, let $x=x_1$. Replacing $u=2\partial_x F\left(\Theta\right)$ in \eqref{KdVZK}, and integrating once in $x$, we obtain
\begin{equation}\label{comparacionmZK}
\begin{aligned}
 0=&\left(\Theta_{xxx}-4 {\Theta}_t\right) F'\left(\Theta\right)+3 {\Theta}_{xx} {\Theta}_x F''\left(\Theta\right)+ {\Theta}_x^3 \left(F'''+ 2F'^3\right)\left(\Theta\right).\\
 &+\sum_{j=2}^d  \left( \Theta _{x_jx_j} F'\left(\Theta\right)+\Theta_{x_j}^2 F''\left(\Theta\right) \right)_{x}
 \end{aligned}
\end{equation}
Notice in contrast of ZK case,  the setting $Ai\left(\Theta\right)=0$, $W\left(\Theta\right)=0$ with  $W_{x_i}^F=0$ produces a contradiction  in \eqref{comparacionmZK}, when considering  $F=2\arctan$. Thus the  condition     $\Lambda_{x_j}^F\left(\Theta\right)=\Theta _{x_jx_j} F'\left(\Theta\right)+\Theta_{x_j}^2 F''\left(\Theta\right)=0$ enters in action. Similarly as in the ZK case, we arrive to the equation
\begin{equation}\label{nuevas2}
\begin{aligned}
  0=  &~{} \left(\Theta_{xxx}-4 {\Theta}_t\right) F'\left(\Theta\right)+3 \left( \Theta_{xx} -\frac{\Theta_x^2}{\Theta} \right)  h'\left(\Theta\right)\\
  &+ \Theta_x^2 \left(h'' + \frac{3}{s}h' +  8h^3 \right)\left(\Theta\right) +\sum_{j=2}^d\left( \Lambda_{x_j}^F\left(\Theta\right)\right)_{x}, 
\end{aligned}
\end{equation}
where $h\left(s\right)=\frac{1}{2}F'\left(s\right)$.  Now we finish the proof. If $u$ is a mKdV soliton (see \eqref{ppal_form:mKdV}), it is easy to check that \eqref{nuevas2} is satisfied. Now, if
\begin{equation}\label{conmZK}
Ai\left(\Theta\right)=W\left(\Theta\right) = \Lambda_{x_j}^F\left(\Theta\right)=0, \quad j=2,\ldots, d,
\end{equation}
the second term in \eqref{nuevas2} (being zero now) can be written as the radial solution $h=h\left(s\right)$, $s=\left|x\right| >0$ of $\Delta h + 8 h^3 =0$ in dimension 4. Requiring positive solutions (by hypothesis, $h\left(s\right):= \frac12 F'\left(s\right)>0$), these are in $H^1\left(\mathbb R^4\right)$ and by classical Talenti-Aubin arguments, one has $h\left(s\right)=\frac1{1+s^2}$, and finally $F= 2 \arctan $ if we assume $F\left(0\right)=0$. Therefore, 
\[
\Lambda_{x_j}^F\left(\Theta\right)=\Lambda_{x_j}^{2\arctan}\left(\Theta\right)=0,\quad  i=2,\ldots, d.
\] 
Now,  from $W\left(\Theta\right)=0$ we get 
\begin{equation*}
\Theta = a\left(x_2,\ldots,x_d,t\right) \exp\left(x b\left(x_2,\ldots,x_d,t\right) \right) +c\left(x_2,\ldots,x_d,t\right). 
\end{equation*}
Also since $\Lambda_{x_j}^{2\arctan}\left(\Theta\right) =0$ we will get that $a,b,c$ are only time dependent. Indeed if  $a:=a\left(x_2,\ldots,x_d,t\right)$, $b:=b\left(x_2,\ldots,x_d,t\right)$, $c:=c\left(x_2,\ldots,x_d,t\right)$,   replacing $\Theta$ in $\Lambda_{x_j}^F\left(\Theta\right)=0$ for each $j=2,\ldots, d$, one has 
\begin{equation*}
  \begin{aligned}
    0=&A_0+A_1\exp\left(bx\right)+A_2x\exp\left(bx\right)+A_3x^2\exp\left(bx\right)\\
     &+B_1\exp\left(2bx\right)+B_2x\exp\left(2bx\right)+B_3x^2\exp\left(2bx\right)\\
     &+C_1\exp\left(3bx\right)+C_2x\exp\left(3bx\right)+C_3x^2\exp\left(3bx\right),
  \end{aligned}
\end{equation*}
where the $A_0, A_1,A_2, A_3$ are described by
\begin{equation*}
  \begin{aligned}
    A_0=&-2 c  c' {}^2+c {}^2 c'' +c'' \\
    A_1=& -4 c  a'  c' +c {}^2 a'' +a'' -2 a 
     c' {}^2+2 a  c  c'' \\
    A_2=& ~{} 2 c {}^2 a'  b' +2 a'  b' -4 a  c  b'  c' +a  c {}^2 b'' +a  b'' \\
    A_3=&~{} a  c {}^2 b' {}^2+a  b' {}^2;
  \end{aligned}
\end{equation*}
the values of $B_1,B_2$ and $B_3$ are
\begin{equation*}
  \begin{aligned}
    B_1=& -4 a  a'  c' -2 c  a' {}^2+2 a  c  a'' +a {}^2 c''  \\
    B_2=&~{} 2 a {}^2 c  b'' -4 a {}^2 b'  c'  \\
    B_3=&~{} 0,
  \end{aligned}
\end{equation*}
and $C_1,C_2,C_3$ are given by
\begin{equation*}
  \begin{aligned}
    C_1=& -2 a  a' {}^2+a {}^2 a'' -x^2 a {}^3 b' {}^2\\
    C_2=& -2 a  a' {}^2+a {}^2 a'' \\
    C_3=& ~{}a {}^3 b' {}^2.
  \end{aligned}
\end{equation*}
Here the prime derivative denotes partial derivative in $x_j$. Notice that $A_i=B_i=C_i=0$ for $i=1,2, 3$. 

\medskip

Since $C_3=a^3b_{x_j}^2$ needs to be zero, we have  $b$ is constant in $x_j$. From $C_2=0$, taking in mind $w=a^{-1}$ as a change of variable of this ODE, we conclude  that 
\begin{equation}\label{amZK}
a\left(x_j\right)=\frac{1}{Ax_j^2+Bx_j +C}.
\end{equation}
Replacing  $a$  in $B_1$, after some simplifications, one has
\begin{equation}\label{mZK3}
  \begin{aligned}
    0=& \left(4 A^2 c +12 A B c' +2 A C c'' +B^2 c'' \right)x_j^2+\left(8 A^2 c' +2 A B c'' \right)x_j^3 \\
    &+A^2 c''x_j^4 +\left(4 A B c +8 A C c' +4 B^2 c' +2 B C c'' \right)x_j \\
    &-4 A C c +2 B^2 c +4 B C c' +C^2 c''.
  \end{aligned}
\end{equation}
Also recall from that $A_1$
\begin{equation}\label{mZK1}
  \begin{aligned}
0=-4 c  a'  c' +c {}^2 a'' +a'' -2 a  c' {}^2+2 a  c  c''.
  \end{aligned}
\end{equation}
Now, the  term $x_j^4$ implies $A=0$ or $c''=0$. If $A\neq 0$ from the term $x_j^3$ one has $c'=0$  then the term $x_j^2$ implies $c=0$. Thus \eqref{mZK1} yields $a''=0$ which contradicts  \eqref{amZK}. Hence $A=0$. Similarly from the term $x_j^2$ of \eqref{mZK3} one has $B=0$. Hence $a$ is constant in $x_j$. Now, $A=B=0$ in \eqref{mZK3} implies $c''=0$. Notice that  \eqref{mZK1}  now is described by
\begin{equation*}
  0=\frac{2 c  c'' }{C}-\frac{2 c' {}^2}{C}=-\frac{2 c' {}^2}{C}.
\end{equation*}
Thus $c$ is constant in $x_j$. Finally, from Airy we obtain that $u=Q_k$ as in \eqref{ppal_form:mKdV}, for some constant $k>0$. 

\appendix

\section{Proof of technical results}\label{AppA}

\subsection{Galilean actions}

Evaluation $\Theta_{\beta}$, obtained after apply the Galilean Transformation to a phase $\Theta$, in the terms that appears in \eqref{galilean_theta}. Then the terms in Definition \ref{Thetas} satisfy,
\begin{equation}\label{Galilean}
    \begin{aligned}
        H\left(\Theta_{\beta}\right) =&~{} -\frac{4\beta}{3}\partial_{\tilde x}\Theta + \partial_{\tilde y}\Theta - \partial_{\tilde x}^2 \Theta = -\frac{4\beta}{3}\partial_{\tilde x}\Theta + H\left(\Theta\right),\\
        Ai\left(\Theta_{\beta}\right) =&~{} \frac{4\beta^{2}}{3}\partial_{\tilde x}\Theta - 2\beta\partial_{\tilde y}\Theta + \partial_{\tilde t} \Theta - \partial_{\tilde x}^3 \Theta =\frac{4\beta^{2}}{3}\partial_{\tilde x}\Theta - 2\beta\partial_{\tilde y}\Theta  + Ai\left(\Theta\right),\\
        \Theta_\beta W_y\left(\Theta_\beta\right) =&~{} \Theta W_{\tilde y}\left(\Theta\right) +\frac{16}9 \beta^2 \left(\Theta \Theta_{\tilde x \tilde x} -\Theta_{\tilde x}^2\right) -\frac83\beta\left(\Theta\Theta_{\tilde x \tilde y}-\Theta_{\tilde x}\Theta_{\tilde y}\right), \\
         \Theta_\beta  W_x\left(\Theta_\beta\right) =&~{} \Theta W_{\tilde x}\left(\Theta\right) . 
    \end{aligned}
\end{equation}
In particular $H$, $Ai$ and $W_y$ do not cancel for a nontrivial Galilean version of the vertical soliton \eqref{Qk}.

\subsection{Proof of phase computations}\label{A}

In this section we prove \eqref{Aicomputation}. Given $\Theta$ as in \eqref{2soliton}, its derivatives are
\begin{align*}
    \Theta_{x} = & \left(k_{3} - k_{1}\right)\left(k_{3} + k_{1}\right)\exp\left(\theta_{1} + \theta_{3}\right) + \left(k_{4} - k_{1}\right)\left(k_{4} + k_{1}\right)\exp\left(\theta_{1} + \theta_{4}\right) \\
&+ \left(k_{3} - k_{2}\right)\left(k_{3} + k_{2}\right)\exp\left(\theta_{2} + \theta_{3}\right) + \left(k_{4} - k_{2}\right)\left(k_{4} + k_{2}\right)\exp\left(\theta_{2} + \theta_{4}\right),\\
    \Theta_{xx} = & \left(k_{3} - k_{1}\right)\left(k_{3} + k_{1}\right)^{2}\exp\left(\theta_{1} + \theta_{3}\right) + \left(k_{4} - k_{1}\right)\left(k_{4} + k_{1}\right)^{2}\exp\left(\theta_{1} + \theta_{4}\right) \\
&+ \left(k_{3} - k_{2}\right)\left(k_{3} + k_{2}\right)^{2}\exp\left(\theta_{2} + \theta_{3}\right) + \left(k_{4} - k_{2}\right)\left(k_{4} + k_{2}\right)^{2}\exp\left(\theta_{2} + \theta_{4}\right),\\
    \Theta_{xxx} = & \left(k_{3} - k_{1}\right)\left(k_{3} + k_{1}\right)^{3}\exp\left(\theta_{1} + \theta_{3}\right) + \left(k_{4} - k_{1}\right)\left(k_{4} + k_{1}\right)^{3}\exp\left(\theta_{1} + \theta_{4}\right) \\
&+ \left(k_{3} - k_{2}\right)\left(k_{3} + k_{2}\right)^{3}\exp\left(\theta_{2} + \theta_{3}\right) + \left(k_{4} - k_{2}\right)\left(k_{4} + k_{2}\right)^{3}\exp\left(\theta_{2} + \theta_{4}\right),\\
    \Theta_{xxxx} = & \left(k_{3} - k_{1}\right)\left(k_{3} + k_{1}\right)^{4}\exp\left(\theta_{1} + \theta_{3}\right) + \left(k_{4} - k_{1}\right)\left(k_{4} + k_{1}\right)^{4}\exp\left(\theta_{1} + \theta_{4}\right)\\
    &+ \left(k_{3} - k_{2}\right)\left(k_{3} + k_{2}\right)^{4}\exp\left(\theta_{2} + \theta_{3}\right) + \left(k_{4} - k_{2}\right)\left(k_{4} + k_{2}\right)^{4}\exp\left(\theta_{2} + \theta_{4}\right).
 \end{align*}
 Also,
\begin{align*}    
    \Theta_{y} = & \left(k_{3} - k_{1}\right)\left(k^{2}_{3} + k^{2}_{1}\right)\exp\left(\theta_{1} + \theta_{3}\right) + \left(k_{4} - k_{1}\right)\left(k^{2}_{4} + k^{2}_{1}\right)\exp\left(\theta_{1} + \theta_{4}\right) \\
&+ \left(k_{3} - k_{2}\right)\left(k^{2}_{3} + k^{2}_{2}\right)\exp\left(\theta_{2} + \theta_{3}\right) + \left(k_{4} - k_{2}\right)\left(k^{2}_{4} + k^{2}_{2}\right)\exp\left(\theta_{2} + \theta_{4}\right),\\
    \Theta_{yy} = & \left(k_{3} - k_{1}\right)\left(k^{2}_{3} + k^{2}_{1}\right)^{2}\exp\left(\theta_{1} + \theta_{3}\right) + \left(k_{4} - k_{1}\right)\left(k^{2}_{4} + k^{2}_{1}\right)^{2}\exp\left(\theta_{1} + \theta_{4}\right) \\
&+ \left(k_{3} - k_{2}\right)\left(k^{2}_{3} + k^{2}_{2}\right)^{2}\exp\left(\theta_{2} + \theta_{3}\right) + \left(k_{4} - k_{2}\right)\left(k^{2}_{4} + k^{2}_{2}\right)^{2}\exp\left(\theta_{2} + \theta_{4}\right),\\
    \Theta_{t} = & \left(k_{3} - k_{1}\right)\left(k^{3}_{3} + k^{3}_{1}\right)\exp\left(\theta_{1} + \theta_{3}\right) + \left(k_{4} - k_{1}\right)\left(k^{3}_{4} + k^{3}_{1}\right)\exp\left(\theta_{1} + \theta_{4}\right) \\
&+ \left(k_{3} - k_{2}\right)\left(k^{3}_{3} + k^{3}_{2}\right)\exp\left(\theta_{2} + \theta_{3}\right) + \left(k_{4} - k_{2}\right)\left(k^{3}_{4} + k^{3}_{2}\right)\exp\left(\theta_{2} + \theta_{4}\right).
\end{align*}
Consequently,
\begin{align*}
    Ai\left(\Theta\right) &= \Theta_{t} - \Theta_{xxx}\\
    &= \left(k_{3} - k_{1}\right)\left(k^{3}_{3} + k^{3}_{1}\right)\exp\left(\theta_{1} + \theta_{3}\right) + \left(k_{4} - k_{1}\right)\left(k^{3}_{4} + k^{3}_{1}\right)\exp\left(\theta_{1} + \theta_{4}\right) \\
&\quad + \left(k_{3} - k_{2}\right)\left(k^{3}_{3} + k^{3}_{2}\right)\exp\left(\theta_{2} + \theta_{3}\right) + \left(k_{4} - k_{2}\right)\left(k^{3}_{4} + k^{3}_{2}\right)\exp\left(\theta_{2} + \theta_{4}\right) \\
&\quad  -\left(k_{3} - k_{1}\right)\left(k_{3} + k_{1}\right)^{3}\exp\left(\theta_{1} + \theta_{3}\right) - \left(k_{4} - k_{1}\right)\left(k_{4} + k_{1}\right)^{3}\exp\left(\theta_{1} + \theta_{4}\right) \\
    &\quad  - \left(k_{3} - k_{2}\right)\left(k_{3} + k_{2}\right)^{3}\exp\left(\theta_{2} + \theta_{3}\right) - \left(k_{4} - k_{2}\right)\left(k_{4} + k_{2}\right)^{3}\exp\left(\theta_{2} + \theta_{4}\right) .
\end{align*}
This equation is equal to 0 when all the coefficient which multiplies the exponentials are null. Note that all the terms are exponentials, $\exp\left(\theta_{j}+\theta_{i}\right)$ multiplies by a term
\begin{equation*}
    \begin{aligned}
     &   \left(k_{i}-k_{j}\right)\left(k^{3}_{i}+k^{3}_{j}\right) - \left(k_{i}-k_{j}\right)\left(k_{i}+k_{j}\right)^{3}\\
       &\quad = k^{4}_{i} + k_{i}k^{3}_{j} - k^{3}_{i}k_{j} - k^{4}_{j} - \left(k_{i}-k_{j}\right)\left(k^{3}_{i}+3k^{2}_{i}k_{j}+3k_{i}k^{2}_{j}+k^{3}_{j}\right)\\
        &\quad =k^{4}_{i} + k_{i}k^{3}_{j} - k^{3}_{i}k_{j} - k^{4}_{i} - 3k^{3}_{i}k_{j} - 3 k^{2}_{i}k^{2}_{j} - k_{i}k^{3}_{j} + k^{3}_{i}k_{j} + 3k^{2}_{i}k^{2}_{j} + 3k_{i}k^{3}_{j} + k^{4}_{j}\\
       &\quad  = -3k^{3}_{i}k_{j} + 3k_{i}k^{3}_{j} = 3\left(k_{i}k^{3}_{j} - k^{3}_{i}k_{j}\right).
    \end{aligned}
\end{equation*}

\subsection{Computation of derivatives of $\Theta$}\label{derivadas_teta}
We perform here some of the computations required in \eqref{Computation of H} and subsequent lines.
If $\Theta= \sum_{j=1}^M A_{j}\left(t,x\right)\exp\left(B_{j}\left(t,x\right)y\right)$, then one has
\begin{equation*}
\begin{aligned}
\Theta_{x}= &~{} \sum_{j=1}^M \left( A_{j,x}+A_{j}B_{j,x} \right)\exp\left(B_{j}y\right),\\
\Theta_{xx}=&~{} \sum_{j=1}^M \left( A_{j,xx}+2A_{j,x}B_{j,x}y+A_{j}B_{j,xx}y+A_{j}B^{2}_{j,x}y^{2} \right)\exp\left(B_{j}y\right),\\
\Theta_{xxx}=&~{} \sum_{j=1}^M \Big( A_{j,xxx}+3A_{j,xx}B_{j,x}y+3A_{j,x}B_{j,xx}y+3A_{j,x}B^{2}_{j,x}y^{2} \\
&\qquad +A_{j}B_{j,xxx}y+3A_{j}B_{j,x}B_{j,xx}y^{2}+A_{j}B^{3}_{j,x}y^{3} \Big) \exp\left(B_{j}y\right).
\end{aligned}
\end{equation*}
Additionally,
\begin{equation*}
\begin{aligned}
\Theta_{xxxx}=&~{} \sum_{j=1}^M \Big( A_{j,xxxx}+4A_{j,xxx}B_{j,x}y+6A_{j,xx}B_{j,xx}y\\
&\qquad +4A_{j,x}B_{j,xxx}y+A_{j}B_{j,xxxx}y+6A_{j,xx}B^{2}_{j,x}y^{2}\\
&\qquad +2A_{j,x}B_{j,x}B_{j,xx}y^{2}+3A_{j}B^{2}_{j,xx}y^{2}+A_{j}B_{j,x}B_{j,xxx}y^{2}\\
&\qquad +4A_{j,x}B^{3}_{j,x}y^{3} +6A_{j}B^{2}_{j,x}B_{j,xx}y^{3} +A_{j}B^{4}_{j,x}y^{4} \Big) \exp\left(B_{j}y\right), \\
\Theta_{y}=&~{} \sum_{j=1}^M A_{j}B_{j}\exp\left(B_{j}y\right), \qquad \Theta_{yy}=\sum_{j=1}^M  A_{j}B^{2}_{j}\exp\left(B_{j}y\right),\\
\Theta_{t}=&~{} \sum_{j=1}^M \left( A_{j,t}+A_{j}B_{j,t}y \right)\exp\left(B_{j}y\right).
\end{aligned}
\end{equation*}

\section{The KdV and mKdV cases}\label{rem:KdVmKdV}

For the sake of completeness, we provide a sketch of proofs for  Remarks \ref{rem:KdV} and \ref{rem:mKdV}.

\subsection{The KdV case}

Consider the KdV model
\begin{equation}\label{eqn:KdV}
-4u_t +u_{xxx} +6 u u_{x} =0,
\end{equation}
where $u=u\left(t,x\right)\in\mathbb R$ and $t,x\in\mathbb R$. We study the solution $u$ of the form
\begin{equation}\label{eqn: FT_KdV}
u\left(t,x\right)= 2 \partial_x^2 F\left(\Theta\left(t,x\right)\right),
\end{equation}
where, with no loss of generality, we consider $F: \left[1,\infty\right) \longrightarrow \mathbb R$ is smooth and $\Theta=\Theta\left(t,x\right) \in \left[1,\infty\right)$ is also smooth. Since $F$ can be changed by any linear affine function, we can assume that $F\left(1\right)=0$, $F'\left(1\right)=1$. 
The simplest case, the KdV soliton, is found as
\begin{equation}\label{soliton_KdV}
F\left(s\right):= \log s, \quad \Theta\left(t,x\right)= 1+ \exp\left(ax - t a^3/4\right), \quad a \in\mathbb R.
\end{equation}
Exactly as in the KP case, we set the following definitions:

\begin{defn}[Classification of phases $\Theta$]\label{Thetas_KdV}
We shall say that $\Theta$ as in \eqref{eqn: FT_KdV} 
\begin{enumerate}
\item[(i)]is  of Airy type if for all $\left(t,x\right)\in\mathbb R^{2}$,
\[
Ai\left(\Theta\right):= -4\Theta_t + \Theta_{xxx} =0;
\]
\item[(ii)] is of Wronskian type if $\left\{\Theta_x,\Theta_{xx}\right\}\left(t,x\right)$ are linearly dependent, for any $\left(t,x\right)$;
\item[(iii)] is of $\mathcal T$-type if for $F$ fixed,
\begin{equation}\label{eqn:T_type_KdV}
\mathcal T\left(\Theta\right):=\left(Ai\left(\Theta\right)_x +F'\left(\Theta\right) \left( 3W\left(\Theta\right) -\Theta_x Ai\left(\Theta\right)\right)\right) =0,
\end{equation}
where $W\left(\Theta\right):= \Theta_{xx}^2 -\Theta_x\Theta_{xxx}$.
\end{enumerate}
\end{defn}

Some comments are necessary.

\begin{rem}
The Airy type condition naturally describes that the phase of the nonlinear KdV solution $u$ solves the classical Airy linear equation
\[
4\Theta_t -\Theta_{xxx}=0.
\]
This is an interesting coincidence that confirms the complex integrable structure of the KdV model. Standard solutions to the Airy equation have a complex oscillatory behavior for $x<0$, but some simple solutions are 
\begin{equation}\label{examples phases}
\Theta\left(t,x\right) = \exp \left( a x -\frac14 a^3 t +C \right), \quad \Theta\left(t,x\right)= t +\frac23 x^3 +C,
\end{equation}
where $C\in \mathbb R$ is any constant. The first phase corresponds to the $1$-soliton phase, and the second will represent an interesting counterexample to our main results. 
\end{rem}

\begin{rem} 
Unlike KP, here in KdV the Wronskian type phase is extremely restrictive. Multi-soliton solutions will not be of this type. 
\end{rem}

\begin{rem}
Exactly as in KP, notice that $\Theta$ of $\mathcal T$-type is a condition depending on the profile $F$, and consequently is a more complex condition than being of Wronskian or Airy type, which are independent of the profile $F$. Additionally, the $1$-soliton phase \eqref{soliton_KdV} is of Airy type as well. 
\end{rem}

Inserting \eqref{eqn: FT_KdV} to \eqref{eqn:KdV}, and arranging similar terms, one easily arrives at
\begin{equation}\label{eqn_F_Theta}
\begin{aligned}
    & \left(F'''' + 6F''^{2}\right)\left(\Theta\right)\Theta_{x}^{4} + 6\left(F'' + F'^2\right)'\left(\Theta\right) \Theta_{x}^{2} \Theta_{xx}  \\
    & \quad + \left(F'' + F'^2\right)\left(\Theta\right) \left( 3\Theta_{xx}^{2}  -4 \Theta_{x}\left(\Theta_{t} -\Theta_{xxx}\right) \right)+ F'\left(\Theta\right)\left( -4\Theta_{t} +\Theta_{xxx} \right)_x  \\
    &\quad + F'^{2} \left(\Theta\right) \left( 3\left(\Theta_{xx}^{2} - \Theta_{x}\Theta_{xxx}\right) -   \Theta_{x}\left(-4\Theta_{t} + \Theta_{xxx} \right)\right)=0.
\end{aligned}
\end{equation}
(Compare with \eqref{KP3partes}.) As in KP,  we set
\[
\rho\left(s\right):= F''\left(s\right) + F'^2\left(s\right).
\]
If $\rho=0$ and $F\left(1\right)=0$, $F'\left(1\right)=-1$, then $F=\log.$ This follows directly from solving the ODE $\rho\left(s\right)=0$. Notice that from \eqref{ODE_rho} and Definition \ref{Thetas_KdV}, one has that \eqref{eqn_F_Theta} can be written as
\[
\begin{aligned}
    & \left(\rho'' -2F' \rho' +4F'' \rho\right)\Theta_{x}^{4} + 6\rho' \Theta_{x}^{2} \Theta_{xx}  \\
    & \quad + \rho \left( 3\Theta_{xx}^{2}  -4 \Theta_{x}\left(\Theta_{t} -\Theta_{xxx}\right) \right)+ F'\left(\Theta\right)\left( Ai\left(\Theta\right) \right)_x  \\
    &\quad + F'^{2} \left(\Theta\right) \left( 3W\left(\Theta\right) -   \Theta_{x}Ai\left(\Theta\right)\right)=0,
\end{aligned}
\]
and arranging terms,
\begin{equation}\label{casicasi_MKdV}
\begin{aligned}
    & \Theta_{x}^{4}\rho'' + 2\left( 3 \Theta_{x}^{2} \Theta_{xx} - F'\left(\Theta\right)\Theta_{x}^{4} \right) \rho' \\
    & \quad +  \left( 3\Theta_{xx}^{2}  -4 \Theta_{x}(\Theta_{t} -\Theta_{xxx})+4F''\left(\Theta\right) \Theta_{x}^{4} \right) \rho + F'\left(\Theta\right) \mathcal T\left(\Theta\right)=0.
\end{aligned}
\end{equation}
Our main result is
\begin{thm}\label{MT_KdV}
The following are satisfied:
\begin{enumerate}
\item[$\left(i\right)$]
Assume that $\Theta$ is of $\mathcal T$-type, $\Theta_x$ is different from $0$, and $F$ satisfies $F''\left(1\right)=1$ and $F'''\left(1\right)=-2$. Then $F=\log $.

\medskip

\item[$\left(ii\right)$] Assume that $\Theta>0$ is of Airy and Wronskian type, $\Theta_x\neq 0$, and $F$ satisfies $F''\left(1\right)=1$ and $F'''\left(1\right)=-2$. Then $F=\log $ and $\Theta = 1+ \exp\left(ax-a^3/4 t\right)$, for any $a\in\mathbb R$.

\medskip

\item[$\left(iii\right)$] Assume that $u$ as in \eqref{eqn: FT_KdV} solves KdV with $F\left(s\right)=\log s$, and $\Theta$ is of Wronskian type. Then $u$ is a soliton.

\medskip

\item[$(iv)$] Assume that $u$ is a nontrivial multisoliton with $F=\log$. Then $\Theta$ is of $\mathcal T$-type, but it cannot be of Wronskian  or Airy type. 
\end{enumerate}
\end{thm}

The simplest phase to characterize is that of Wronskian type. 

\begin{lem}\label{lem:W}
Assume that $\Theta$ is of Wronskian type. Then 
\begin{enumerate}
\item[$\left(i\right)$] One has $W\left(\Theta\right) = \Theta_{xx}^2 -\Theta_x\Theta_{xx} =0$.
\item[$\left(ii\right)$] Additionally, there are $a=a\left(t\right) $ and $b=b\left(t\right)$ such that either
\[
\Theta\left(t,x\right)= 1 + \exp\left(a\left(t\right) x+ b\left(t\right)\right), 
\]
or
\[
\Theta\left(t,x\right)= a\left(t\right) x+ b\left(t\right).
\]
\end{enumerate}
\end{lem}

The previous result establishes that phases $\Theta$ of Wronskian type are directly related to soliton solutions. This fact is independent of the value of $F$, that will be determined independently.

\begin{proof}
The proof follows directly from Definition \ref{Thetas_KdV} $\left(ii\right)$. Indeed, the proof of $\left(i\right)$ is direct from the definition of the linear dependence and the fact that 
\[
W\left(\Theta\right)= \Theta_{xx}^2 -\Theta_x\Theta_{xxx} = \hbox{Wr}\left(\Theta_{xx},\Theta_{x}\right),
\]
where $\hbox{Wr}$ denotes the Wronskian. The proof of $\left(ii\right)$ follows from the fact that $\eta:= \Theta_x$ satisfies the ODE $\eta_{x}^2 = \eta \eta_{xx}$, which has solutions $\eta\left(x\right)= \exp\left(a\left(t\right) x +b\left(t\right)\right)$ and $\eta = a\left(t\right)$. The final result is obtained by integrating in space. 
\end{proof}


\begin{rem}
1. Notice that multisolitons cannot have a phase of Wronskian type. 

\medskip

2. An important example of phase of Airy type but not being of Wronskian type is the one given in \eqref{examples phases}: $\Theta\left(t,x\right)=t +\frac23x^3+1$. In this case, $W\left(\Theta\right)= 8x^2$. 

\medskip

3. If $\Theta$ is of Wronskian and Airy type, then is of $\mathcal T$-type. The reciprocal is clearly false. This follows directly from \eqref{eqn:T_type_KdV}. 
\end{rem}

\begin{lem}
Assume that $\Theta$ is of $\mathcal T$-type. Then 
\begin{enumerate}
\item[$\left(i\right)$] There exists $c_0\in\mathbb R$ such that 
\begin{equation}\label{sol_Ai_KdV}
Ai\left(\Theta\right)= c_0 \exp\left(F\left(\Theta\right)\right) -3 W\left(\Theta\right) +3 F\left(\Theta\right)\int_0^x \exp\left(-F\left(\Theta\right)\right) \left(W\left(\Theta\right)\right)_x \left(t,s\right) \mathrm{d}s.
\end{equation}
\item[$\left(ii\right)$] If $\Theta$ is of Airy and $\mathcal T$-type, and $F'\left(\Theta\right)$ is different from  $0$, then it is of Wronskian type.
\item[$\left(iii\right)$] If $\Theta$ is of Wronskian and $\mathcal T$-type, then there exists $c\in\mathbb R$ such that $Ai\left(\Theta\right)=c_0 \exp\left(\Theta\right)$.
\end{enumerate}
\end{lem}

\begin{proof}
Let us prove $\left(i\right)$: notice that $\Theta$ being of $\mathcal T$-type as in \eqref{eqn:T_type_KdV} is equivalent to have
\[
\left(Ai\left(\Theta\right)\right)_x -F'\left(\Theta\right) \Theta_x Ai\left(\Theta\right) = 3F'\left(\Theta\right) W\left(\Theta\right).
\]
Then \eqref{sol_Ai_KdV} follows directly from solving the corresponding ODE for $Ai\left(\Theta\right)$. In order to prove $\left(ii\right)$, notice that from \eqref{eqn:T_type_KdV} one has $0= F'\left(\Theta\right) W\left(\Theta\right)$, proving $\left(ii\right)$. Finally, to prove $\left(iii\right)$, from \eqref{sol_Ai_KdV} one has $Ai\left(\Theta\right)= c_0 \exp\left(F\left(\Theta\right)\right)$, as required.
\end{proof}

Now we are ready to prove the main result in the KdV case.

\begin{proof}[Proof of Theorem \ref{MT_KdV}] Now we are ready to prove Theorem \ref{MT_KdV}.

\medskip

\emph{Proof of $\left(i\right)$}. Under $ \mathcal T\left(\Theta\right)=0$, we have from \eqref{casicasi_MKdV} an ODE of the form
\[
 \Theta_{x}^{4}\rho'' + 2\left( 3 \Theta_{x}^{2} \Theta_{xx} - F'\left(\Theta\right)\Theta_{x}^{4} \right) \rho' +  \left( 3\Theta_{xx}^{2}  -4 \Theta_{x}\left(\Theta_{t} -\Theta_{xxx}\right)+4F''\left(\Theta\right) \Theta_{x}^{4} \right) \rho =0.
\]
This homogeneous ODE for $\rho$ has zero as the unique solution provided $\rho\left(1\right)=\rho'\left(1\right)=0$, which is indeed the case. Consequently, from Lemma \ref{dem:ODE1}, we get $F=\log.$

\medskip

\emph{Proof of $\left(ii\right)$}. From $\Theta$ being of Airy and Wronskian type we have $\Theta$ of $\mathcal T$-type. The previous result ensures $F=\log$. Finally, Lemma \ref{lem:W} proves the final result, after checking that $\Theta$ is of Airy type and $\Theta>0$.   

\medskip

\emph{Proof of $\left(iii\right)$}. If $F=\log$ then $\rho=0$, and from \eqref{casicasi_MKdV}, $\Theta^{-1} \mathcal T\left(\Theta\right)=0$. Consequently, $\Theta$ is of $\mathcal T$-type. Since it is additionally of Wronskian type, it is of Airy type and $\left(ii\right)$ applies.

\medskip

\emph{Proof of $\left(iv\right)$}. The multisoliton $\Theta$ is of $\mathcal T$-type, but it does not satisfy being of Airy or Wronskian type.  This ends the proof of Theorem \ref{MT_KdV}.
\end{proof}

\subsection{The mKdV case}

Let us consider the mKdV model
\begin{equation}\label{eqn:mKdV}
-4u_t +u_{xxx} +6 u^2 u_{x} =0,
\end{equation}
where $u=u\left(t,x\right)\in\mathbb R$ and $\left(t,x\right)\in\mathbb R^{2}$. Let us extend our previous results to the mKdV case. Consider
\begin{equation}\label{eqn: FT_mKdV}
u\left(t,x\right)=  \partial_x F\left(\Theta\left(t,x\right)\right),
\end{equation}
(notice that we only consider one derivative in space). We consider $F: \mathbb R \longrightarrow \mathbb R$ smooth and the phase $\Theta=\Theta\left(t,x\right) \in \mathbb R$ also smooth. Since $F$ can be changed by any constant, we can assume that $F\left(0\right)=0$. Classical mKdV solitons (see e.g. \cite{AM}) are given by
\begin{equation}\label{ppal_form:mKdV}
\begin{aligned}
Q_k\left(t,x\right) = &~{} 2\partial_x \arctan \left(\Theta_k\left(t,x\right)\right) \\
= &~{} k\sech \left( k x+ \frac14 k^3 t +b_0 \right), \quad \Theta_k := \exp\left( k x+ \frac14 k^3 t +b_0 \right), \quad b_0\in\mathbb R.
\end{aligned}
\end{equation}
Replacing $u$ in \eqref{eqn:mKdV} by  \eqref{eqn: FT_mKdV}, we obtain
\begin{equation}\label{comparacion}
 \left(\Theta_{xxx}-4 {\Theta}_t\right) F'\left(\Theta\right)+3 {\Theta}_{xx} {\Theta}_x F''\left(\Theta\right)+ {\Theta}_x^3 \left(F'''+ 2F'^3\right)\left(\Theta\right)=0.
\end{equation}
Trivial phases $\Theta=const.$ will be discarded now, so that we assume $\Theta_x\neq 0$. Note that $\Theta_k$ in \eqref{ppal_form:mKdV} satisfies $Ai\left(\Theta\right)=\Theta_{xxx}-4 {\Theta}_t=0$. Consequently, from \eqref{comparacion} one must have
\[
3 {\Theta}_{xx}  F''\left(\Theta\right)+ \Theta_x^2\left(F'''+ 2F'^3\right)\left(\Theta\right)=0 .
\]
Let $h\left(s\right):= \frac12 F'\left(s\right)$. The equation above  can be written as
\begin{equation}\label{mKdV_parte_final}
3 \left( \Theta_{xx} -\frac{\Theta_x^2}{\Theta} \right)  h'\left(\Theta\right)+ \Theta_x^2 \left(h'' + \frac{3}{s}h' +  8h^3 \right)\left(\Theta\right)=0 .
\end{equation}
The first term contains nothing but $W_x\left(\Theta\right):=  \Theta_{xx} -\frac{\Theta_x^2}{\Theta}$, which will be required to be zero. Notice that $\Theta_k$ in \eqref{ppal_form:mKdV} does satisfy this condition, i.e., $W_x\left(\Theta_k\right)=0$. In the opposite direction, $W_x\left(\Theta\right) =0$ implies $\Theta \left( \frac{\Theta_x}{\Theta} \right)_x =0 $, leading to $\Theta \left(t,x\right)= \exp\left(k\left(t\right) x+ b\left(t\right)\right)$. From the Airy condition $\Theta_{xxx}-4 {\Theta}_t=0$ one gets
\[
k^3 -4 k' x -4 b' =0,
\]
leading to $k\left(t\right)=k$ for some constant $k$  and $b= \frac14 k^3 t + b_0$ with some   constant $b_0$. Therefore, $ \Theta= \exp\left( k x+ \frac14 k^3 t + b_0\right)$. Notice that the case $k=0$ returns the trivial solution, therefore it will be left out of the subsequent analysis. In particular, the image of $\Theta$ is $\left(0,\infty\right)$. 

\medskip

Finally, the second term in \eqref{mKdV_parte_final} (being zero now) can be written as the radial solution $h=h\left(s\right)$, $s=\left|x\right| >0$ of {\color{black} the elliptic, nonlinear PDE} $\Delta h + 8 h^3 =0$ in dimension 4.  Requiring positive solutions (by hypothesis, $h\left(s\right):= \frac12 F'\left(s\right)>0$), these are in $H^1\left(\mathbb R^4\right)$ and by classical Talenti-Aubin arguments, one has $h\left(s\right)=\frac1{1+s^2}$, giving,  finally,  $F= 2 \arctan $ if we assume $F\left(0\right)=0$. We conclude the following result:

\begin{thm}
Let $u$ be a smooth solution to mKdV \eqref{eqn:mKdV} of the form \eqref{eqn: FT_mKdV}, with smooth profile $F:\mathbb R\to \mathbb R$ satisfying $F\left(0\right)=0$ and $F$ strictly increasing in $\mathbb R$. Then a nontrivial $\Theta$ is a soliton \eqref{ppal_form:mKdV} and $F=2 \arctan$ if and only if $W_x\left(\Theta\right)=Ai\left(\Theta\right)=0$. 
\end{thm}

Notice that the condition $F$ strictly increasing in $\mathbb R$ can be replaced by $F$ strictly monotone in $\mathbb R$, since by symmetries of the equation \eqref{eqn:mKdV}, if $u$ is a solution, then $-u$ also does. However, it is noticed that ``cn'' periodic solutions may appear if this condition is lifted.

\medskip

Now we assume that $\Theta_x$ may take the value zero. Additionally, we will assume that $Ai(\Theta)$ need not be zero as well. Coming back to \eqref{comparacion}, and following the ideas that led to \eqref{mKdV_parte_final}, we shall obtain
\[
 \left( \Theta_{xxx} -4\Theta_t \right)h  + 3 \left( \Theta_{xx} -\frac{\Theta_x^2}{\Theta} \right) \Theta_x  h'\left(\Theta\right)+ \Theta_x^3 \left(h'' + \frac{3}{s}h' +  8h^3 \right)\left(\Theta\right)=0 .
\]
This equation is valid provided $\Theta \neq 0$. If $F$ is assumed to be increasing as in the previous argument, one has $h\left(s\right)=\frac1{1+s^2}$, leading to the equation
\begin{equation}\label{mandelazo}
(1+\Theta^2) \left(4\Theta_t - \Theta_{xxx} \right)  + 6\Theta_x \left( \Theta \Theta_{xx} - \Theta_x^2 \right)=0,
\end{equation}
valid now in the full region. One can easily see that for any $\alpha,\beta>0$, $\delta=\frac14\left(3\beta^2-\alpha^2\right)$, and $\gamma=\frac14\left(\beta^2-3\alpha^2\right)$, one has that 
\[
\Theta= \pm\frac{\beta}{\alpha} \frac{\sin\left(\alpha\left(x+\delta t\right)\right)}{\cosh\left(\beta\left(x+\gamma t\right)\right)},
\]
solves the previous equation, representing the classical breather solution, see \cite{AM} and references therein. Also,
\[
\Theta= \frac{ e^{\sqrt{c_1}\left(x+c_1 t\right)} + e^{\sqrt{c_2}\left(x+c_2 t\right)} }{1-\rho^2 e^{\sqrt{c_1}\left(x+c_1 t\right) +\sqrt{c_2}\left(x+c_2 t\right)}},\quad \rho:= \frac{\sqrt{c_1} -\sqrt{c_2}}{\sqrt{c_1} + \sqrt{c_2}},
\]
with $c_1,c_2>0$, is another solution, representing the 2-soliton solution. In general, characterizing the solutions to \eqref{mandelazo} is not a simple task. In \cite{FFMP} it is proposed a first result in this direction, where one assumes that the phase $\Theta$ has a particular \emph{monochromatic} structure.

\end{document}